\documentclass{amsart}

\usepackage{amsmath,amsthm,amsfonts,amsbsy,thmtools,amssymb,hyperref,fullpage}
\usepackage{bm}
\usepackage{wasysym}
\usepackage{graphicx}
\usepackage{tikz}
\usetikzlibrary{graphs,patterns,decorations.markings,arrows,matrix}
\usetikzlibrary{calc,decorations.pathmorphing,decorations.pathreplacing,shapes}

\renewcommand{\tikz}[2]{
\begin{tikzpicture}[scale=#1,baseline=(current bounding box.center),>=stealth]
#2
\end{tikzpicture}}

\colorlet{lgray}{white!85!black}

\numberwithin{equation}{section}

\newtheorem{thm}{Theorem}[section]
\newtheorem{prop}[thm]{Proposition}

\newtheorem{cor}[thm]{Corollary}
\newtheorem{conj}[thm]{Conjecture}
\theoremstyle{remark}
\newtheorem{rmk}[thm]{Remark}
\theoremstyle{definition}
\newtheorem{defn}[thm]{Definition}
\newtheorem{ex}[thm]{Example}

\def\A{\bm{A}}
\def\B{\bm{B}}
\def\C{\bm{C}}
\def\D{\bm{D}}
\def\P{\bm{P}}
\def\V{\bm{V}}
\newcommand{\As}[2]{A_{[#1,#2]}}
\newcommand{\Ap}[2]{A_{(#1,#2]}}

\def\DD{\mathcal{D}}

\def\leq{\leqslant}
\def\geq{\geqslant}

\newcommand{\bra}[1]{\left\langle #1\right|}
\newcommand{\ket}[1]{\left|#1\right\rangle}

\newcommand\fs{\footnotesize}

\begin{document}

\title{Coloured corner processes from asymptotics of LLT polynomials}
\author{Amol Aggarwal, Alexei Borodin and Michael Wheeler}

%abstract
\begin{abstract}
We consider probability measures arising from the Cauchy summation identity for the LLT (Lascoux--Leclerc--Thibon) symmetric polynomials of rank $n \geq 1$. We study the asymptotic behaviour of these measures as one of the two sets of polynomials in the Cauchy identity stays fixed, while the other one grows to infinity. At $n=1$, this corresponds to an analogous limit of the Schur process, which is known to be given by the Gaussian Unitary Ensemble (GUE) corners process.

Our main result states that, for $n>1$, our measures asymptotically split into two parts: a continuous one and a discrete one. The continuous part is a product of $n$ GUE corners processes; the discrete part is an explicit finite distribution on interlacing $n$-colourings of $n$ interlacing triangles, which has weights that are rational functions in the LLT parameter $q$. The latter distribution has a number of interesting (partly conjectural) combinatorial properties, such as $q$-nonnegativity and enumerative phenomena underlying its support.

Our main tools are two different representations of the LLT polynomials, one as partition functions of a fermionic lattice model of rank $n$, and the other as finite-dimensional contour integrals, which were recently obtained in {\tt arXiv:2012.02376}, {\tt arXiv:2101.01605}. 
\end{abstract}

\maketitle

%contents
\setcounter{tocdepth}{1}
\makeatletter
\def\l@subsection{\@tocline{2}{0pt}{2.5pc}{5pc}{}}
\makeatother
\tableofcontents

\section{Introduction}

\subsection{Preface}

The Gaussian Unitary Ensemble (or GUE, for short) is one of the cornerstones of Random Matrix Theory that goes back to Wigner \cite{Wigner65}. It consists of Hermitian matrices $H$ distributed according to the Gaussian measure $P(dH) \sim \exp(-Tr(H^2)) dH$, which is the essentially unique\footnote{Up to shifting and scaling.} distribution on this set that satisfies two natural conditions: (a) It is invariant under any unitary conjugation; and (b) Linearly independent real and imaginary parts of matrix elements are statistically independent, see \cite[Section 2.5]{Mehta}. 

One important feature of the GUE is that it can be viewed as a universal limiting object for discrete probabilistic systems related to representation theory. The first limiting relation of this kind goes back to Kerov \cite{Kerov88} who studied the distribution of symmetry types of tensors in high tensor powers of a finite-dimensional vector space. In this case the limit is described by the distribution of spectra of traceless GUE matrices, and the condition of vanishing trace can be naturally removed by randomizing the number of tensor factors (also known as Poissonization). The result was later rediscovered by 
Tracy--Widom \cite{TracyWidom01} in the first wave of works related to the asymptotics of longest increasing sequences. A somewhat more conceptual way to view this result is that of the quasi-classical limit in representation theory, see \emph{e.g.} Heckman \cite{Heckman82}. 

GUEs of different sizes can be coupled by viewing them as upper-left corners of the same infinite Hermitian matrices. Such measures on infinite Hermitian matrices naturally appear in Asymptotic Representation Theory, see Olshanski--Vershik \cite{OlshanskiVershik}. In the framework of tiling models, thus coupled GUEs were first obtained by Johansson--Nordenstam \cite{JohanssonNordenstam} and Okounkov--Reshetikhin \cite{OkounkovReshetikhin}, and their universality in such contexts was recently shown by Aggarwal--Gorin \cite{AggarwalGorin}. The terms ``GUE minors process" and ``GUE corners process" have both been introduced for the resulting ensemble; we will use the latter one. 

If one translates the problem of analyzing tensor symmetry types to the language of symmetric functions (which in this case represent the characters of both general linear and symmetric groups), then one is looking at probability measures on partitions obtained from summands in the Cauchy summation identity for the Schur symmetric polynomials. There are two sets of Schur polynomials in the game; one of them remains fixed, in correspondence with the fixed dimension of the vector space that is being tensored, while the specialization of the other one is growing in the way corresponding to the growing tensor power\footnote{More exactly, the Poissonization parameter tends to infinity.}. The measures on partitions arising from specializations of this Cauchy identity have been known as \emph{Schur measures} since the work of Okounkov \cite{Okounkov01}.

The goal of the present work is to perform asymptotic analysis in a similar setup, but with the role of the Schur polynomials played instead by {\it LLT symmetric polynomials}. The LLT polynomials were introduced by Lascoux--Leclerc--Thibon in \cite{LLT}; an insightful and easy-to-read account of their first 25 years by Thibon can be found at \cite{Thibon}. Cauchy-type summation identities for the LLT polynomials were later obtained by Lam \cite{Lam}, and the probability measures that we study have weights proportional to the summands of such an identity. 

While the origins of the LLT polynomials were representation theoretic, {\it cf.} \cite{CarreLeclerc}, their most transparent definition is combinatorial --- they are generating functions of ribbon Young tableaux, where monomials in the variables of the polynomials are used to track the weight of the tableaux, and powers of a new parameter $q$ track the so-called \emph{spin} statistics introduced in \cite{LLT}. When $q=1$, the LLT polynomials reduce to products of Schur polynomials, the number of which (also equal to the size of the ribbons) will be called the \emph{rank}; we will denote it by $n$ throughout the paper. Thus, one can think of the LLT polynomials as a higher rank $q$-analogue of (products of) Schur polynomials. 

Neither the combinatorial nor the representation theoretic definitions of the LLT polynomials seem suitable for the asymptotic problem in question. On the other hand, we recently found an integral representation for these polynomials in \cite[Chapter 11]{ABW21}. It is the steepest descent analysis of those integral representations that allowed us to reach our main result. 

The limit that we obtained carried a couple of surprises, the main one being that it splits into a continuous and a discrete part. The continuous part is a direct product of $n$ GUE corners processes. The discrete part is a probability distribution on the (finitely many) ways to colour $n$ interlacing triangular arrays\footnote{These arrays originate from $n$ Gelfand--Tsetlin patterns drawn next to each other.} by $n$ colors so that each color interlaces (an exact definition is below). The latter distribution has a few interesting properties. 

First, its weights can be represented as certain partition functions of a \emph{fermionic lattice model of rank $n$}. The connection is in no way immediate, and it is related to the vertex model representations for the LLT polynomials obtained in \cite{ABW21}, see also Corteel--Gitlin--Keating--Meza \cite{CGKM22}. This vertex model interpretation of the limiting distribution ends up being crucial for our proof. 

Second, these weights, which are \emph{a priori} rational functions of the deformation parameter $q$, appear to be given by polynomials in $q$ with positive integer coefficients divided by a power of the $q$-factorial of $n$. We conjecture that this is always the case, even though we were only able to observe this phenomenon on the few examples we tested on a computer. The combinatorial meaning of the coefficients of the resulting polynomials also remains unclear. See Figure \ref{fig:n=2} in Appendix \ref{sec:app} below for a quick example in rank $2$. 

Third, the size of the support of the distributions, \emph{i.e.}, the number of interlacing $n$-colourings of $n$ triangular arrays appears to be combinatorially interesting. It is easy to compute for $n=1$ and $2$, when it is equal to $1$ and to a simple power of $2$, respectively. However, for $n=3$ it turns out to be equal to the number of 4-colourings of a triangle in the triangular lattice. We originally conjectured this coincidence on the basis of numerics, and it was later proved via an elegant bijective construction by Gaetz--Gao \cite{GaetzGao}. For $n=4$ the numerics suggest a similar relationship with $5$-colourings of squares in the ``king graph'', see Conjecture \ref{conj:a5}, although no proof is currently available. Finally, for $n\geq 5$ we were not able to find similar matchings. 

Recalling the appearance of the GUE corners process in random tilings, it is natural to ask if the limiting object we observed has a meaning in the world of tiling models. We believe it is indeed so, and in particular, the limiting behaviour of the random $n$-tilings of Aztec diamonds introduced by Corteel--Gitlin--Keating \cite{CGK} should have the same limit, as the size of the Aztec diamond tends to infinity, near the tangency points of the ``arctic curve" that bounds the frozen regions. The reason is that these $n$-tilings can be described via a closely related \emph{dual} Cauchy identity for the LLT polynomials. The focus on a tangency point of the arctic curve results in one set of the LLT polynomials within the identity staying fixed, while the specialization of the other one is growing with the size of the domain, much like in the limit that we investigated. We will, however, leave this connection to future studies. 

Let us now describe our results in more detail.

\subsection{Fermionic vertex models, coloured compositions and partition functions}

The vertex models that we consider in this work assign weights to collections of paths drawn on a square grid. Each vertex that is traversed by at least one path produces a weight that depends on the configuration of all the paths that go through it. The total weight for a collection of paths is the product of weights of the vertices that the paths traverse (we assume the normalization in which the weight of an empty vertex is equal to unity).

Each path carries a colour that is a number between $1$ and $n$, where $n\geq 1$ is the rank of the model. Let us first assume that each horizontal edge of the underlying square grid can carry no more than one path, while vertical edges can be occupied by multiple paths of {\it distinct colours}. Thus, the states of the horizontal edges can be encoded by an integer between $0$ and $n$, with $0$ denoting an edge that is not occupied by a path, while the states of the vertical edges can be encoded by $n$-dimensional binary strings which specify whether each colour $\{1,\dots,n\}$ appears (or not) at that edge.

Our paths will always travel upward in the vertical direction, and in the horizontal direction a path can travel rightward or leftward, depending on the specific type of vertices that are used; this choice will always be explicitly stated.

Let us now specify our vertex weights more precisely. In regions of rightward horizontal travel, our vertex weights take the following form:
\begin{align}
\label{generic-L-intro}
\tilde{L}^{(s)}_{x,q}(\A,b;\C,d)
=
\tikz{0.7}{
\node[left] at (-1.5,0) {$x \rightarrow$};
\draw[lgray,line width=1pt,->] (-1,0) -- (1,0);
\draw[lgray,line width=4pt,->] (0,-1) -- (0,1);
\node[left] at (-1,0) {\tiny $b$};\node[right] at (1,0) {\tiny $d$};
\node[below] at (0,-1) {\tiny $\A$};\node[above] at (0,1) {\tiny $\C$};
}
\quad\quad
b,d \in \{0,1,\dots,n\},
\quad
\A,\C \in \{0,1\}^n,
\end{align}
where $\tilde{L}^{(s)}_{x,q}$ is a rational function of three parameters $x,q,s$. Here $x$ is the {\it spectral parameter} associated to a row of the lattice (a different parameter may be used for each row), $q$ is the {\it quantum deformation parameter} (a global parameter that is common to all vertices), and $s$ is the {\it spin parameter}, which arises due to the fact that the vertical line of the vertex is a higher-spin module for the underlying quantized affine Lie algebra $U_q(\widehat{\mathfrak{sl}}(1|n))$. For the explicit form of these weights, see equation \eqref{fund-weights}.

In regions of leftward horizontal travel, our vertex weights are given by
\begin{align}
\label{generic-M-intro}
\tilde{M}^{(s)}_{x,q}(\A,b;\C,d)
=
\tikz{0.7}{
\node[right] at (1.5,0) {$\leftarrow x$};
\draw[lgray,line width=1.5pt,<-] (-1,0) -- (1,0);
\draw[lgray,line width=4pt,->] (0,-1) -- (0,1);
\node[left] at (-1,0) {\tiny $d$};\node[right] at (1,0) {\tiny $b$};
\node[below] at (0,-1) {\tiny $\A$};\node[above] at (0,1) {\tiny $\C$};
}
\quad\quad
b,d \in \{0,1,\dots,n\},
\quad
\A,\C \in \{0,1\}^n,
\end{align}
where $\tilde{M}^{(s)}_{x,q}$ is again a rational function of the three parameters $x,q,s$ defined above. The weights \eqref{generic-L-intro} and \eqref{generic-M-intro} are related via the simple identity
\begin{align}
\label{LM-sym-intro}
\tilde{M}^{(s)}_{x,q}(\A,b;\C,d)
=
\tilde{L}^{(1/s)}_{1/x,1/q}(\A,b;\C,d),
\end{align}
which holds for all $\A,\C \in \{0,1\}^n$ and $b,d \in \{0,1,\dots,n\}$; see equation \eqref{LM-sym}. For full details about the weights \eqref{generic-L-intro} and \eqref{generic-M-intro}, including their Yang--Baxter equations, see Sections \ref{ssec:L}--\ref{ssec:YB}. We note that the fermionic weights \eqref{generic-L-intro} and \eqref{generic-M-intro} appeared previously in \cite{ABW21}, and bosonic counterparts of them date even further back to \cite{BorodinW}.

The partition functions (and ultimately, probability measures) that we consider are all indexed by a set of objects called {\it coloured compositions}:
\begin{defn}[Definition \ref{def:cc} below]
\label{def:cc-intro}
Let $\lambda = (\lambda_1,\dots,\lambda_n)$ be a composition of length $n$. We introduce the set $\mathcal{S}_{\lambda}$ of $\lambda$-coloured compositions as follows:
\begin{align}
\label{lambda-col-intro}
\mathcal{S}_{\lambda}
=
\Big\{ 
\mu 
= 
\Big(
0 \leq \mu^{(1)}_1 < \cdots < \mu^{(1)}_{\lambda_1} \Big|
0 \leq \mu^{(2)}_1 < \cdots < \mu^{(2)}_{\lambda_2} \Big|
\cdots \Big|
0 \leq \mu^{(n)}_1 < \cdots < \mu^{(n)}_{\lambda_n}\Big)
\Big\}.
\end{align}
One may think of the elements of $\mathcal{S}_{\lambda}$ as $n$-tuples $\left(\mu^{(1)},\dots,\mu^{(n)}\right)$ of strict compositions. For each $1 \leq i \leq n$, the superscript of $\mu^{(i)}$ is its {\it colour}, and its length is $\lambda_i$. 
\end{defn}

Our first partition function of interest is denoted $f_{\mu}(\lambda;x_1,\dots,x_m;s)$. This is a (nonsymmetric) rational function in an alphabet  $(x_1,\dots,x_m)$, indexed by a composition $\lambda = (\lambda_1,\dots,\lambda_n)$ satisfying $\sum_{i=1}^{n} \lambda_i = m$, as well as a coloured composition $\mu \in \mathcal{S}_{\lambda}$. Up to an overall multiplicative factor, 
$f_{\mu}(\lambda;x_1,\dots,x_m;s)$ is defined as a partition function using the vertex weights \eqref{generic-L-intro}:
\begin{align}
\label{f-def-intro}
(-s)^{|\mu|}
\cdot
f_{\mu}(\lambda;x_1,\dots,x_m;s)
=
\tikz{0.7}{
\foreach\y in {0,...,5}{
\draw[lgray,line width=1.5pt,->] (1,\y) -- (8,\y);
}
\foreach\x in {2,...,7}{
\draw[lgray,line width=4pt,->] (\x,-1) -- (\x,6);
}
%spectral parameters
\node[left] at (-0.5,0) {$x_1 \rightarrow$};
\node[left] at (-0.5,2) {$\vdots$};
\node[left] at (-0.5,3) {$\vdots$};
\node[left] at (-0.5,5) {$x_m \rightarrow$};
%bottom labels
\node[below] at (7,-1) {$\cdots$};
\node[below] at (6,-1) {$\cdots$};
\node[below] at (5,-1) {$\cdots$};
\node[below] at (4,-1) {\footnotesize$\bm{e}_0$};
\node[below] at (3,-1) {\footnotesize$\bm{e}_0$};
\node[below] at (2,-1) {\footnotesize$\bm{e}_0$};
%top labels
\node[above] at (7,6) {$\cdots$};
\node[above] at (6,6) {$\cdots$};
\node[above] at (5,6) {$\cdots$};
\node[above] at (4,6) {\footnotesize$\bm{A}(2)$};
\node[above] at (3,6) {\footnotesize$\bm{A}(1)$};
\node[above] at (2,6) {\footnotesize$\bm{A}(0)$};
%right labels
\node[right] at (8,0) {$0$};
\node[right] at (8,1) {$0$};
\node[right] at (8,2) {$0$};
\node[right] at (8,3) {$0$};
\node[right] at (8,4) {$0$};
\node[right] at (8,5) {$0$};
%left labels
\node[left] at (1,0) {$1$};
\node[left] at (1.5,0.6) {$\vdots$};
\node[left] at (1,1) {$1$};
\node[left] at (1,2) {$\vdots$};
%\node[left] at (1.5,2.6) {$\vdots$};
\node[left] at (1,3) {$\vdots$};
%\node at (2.5,3.6) {$\vdots$};
%\node at (6.5,3.6) {$\vdots$};
\node[left] at (1,4) {$n$};
\node[left] at (1.5,4.6) {$\vdots$};
\node[left] at (1,5) {$n$};
}
\end{align}
where $\bm{e}_0$ denotes the $n$-dimensional zero vector and $\bm{A}(k) = \sum_{i=1}^{n} \bm{1}_{k \in \mu^{(i)}} \bm{e}_i$ is a binary string that encodes whether $k$ is present (or not) as a part in $\mu^{(i)}$, for all $1 \leq i \leq n$ and $k \geq 0$. A convenient visualization aid is that for each $1 \leq i \leq n$, a collection of $\lambda_i$ paths of colour $i$ enter the partition \eqref{f-def-intro} via its left boundary and travel through the lattice, ultimately exiting via the top of the columns 
$\mu^{(i)}_1 < \cdots < \mu^{(i)}_{\lambda_i}$.

In a similar vein, one may define multivariate (nonsymmetric) rational functions as partition functions constructed from the weights \eqref{generic-M-intro}. We denote these by 
$g_{\mu}(\lambda;x_1,\dots,x_m;s)$, where the specification of $(x_1,\dots,x_m)$, $\lambda$ and 
$\mu \in \mathcal{S}_{\lambda}$ is exactly as above. Up to an overall multiplicative factor, 
$g_{\mu}(\lambda;x_1,\dots,x_m;s)$ is defined as follows:
\begin{align}
\label{g-def-intro}
(-s)^{-|\mu|}
\cdot
g_{\mu}(\lambda;x_1,\dots,x_m;s)
=
\tikz{0.7}{
\foreach\y in {0,...,5}{
\draw[lgray,line width=1.5pt,<-] (1,\y) -- (8,\y);
}
\foreach\x in {2,...,7}{
\draw[lgray,line width=4pt,->] (\x,-1) -- (\x,6);
}
%spectral parameters
\node[right] at (8.5,0) {$\leftarrow x_1$};
\node[right] at (8.5,2) {$\vdots$};
\node[right] at (8.5,3) {$\vdots$};
\node[right] at (8.5,5) {$\leftarrow x_m$};
%top labels
\node[above] at (7,6) {$\cdots$};
\node[above] at (6,6) {$\cdots$};
\node[above] at (5,6) {$\cdots$};
\node[above] at (4,6) {\footnotesize$\bm{e}_0$};
\node[above] at (3,6) {\footnotesize$\bm{e}_0$};
\node[above] at (2,6) {\footnotesize$\bm{e}_0$};
%bottom labels
\node[below] at (7,-1) {$\cdots$};
\node[below] at (6,-1) {$\cdots$};
\node[below] at (5,-1) {$\cdots$};
\node[below] at (4,-1) {\footnotesize$\bm{A}(2)$};
\node[below] at (3,-1) {\footnotesize$\bm{A}(1)$};
\node[below] at (2,-1) {\footnotesize$\bm{A}(0)$};
%right labels
\node[right] at (8,0) {$0$};
\node[right] at (8,1) {$0$};
\node[right] at (8,2) {$0$};
\node[right] at (8,3) {$0$};
\node[right] at (8,4) {$0$};
\node[right] at (8,5) {$0$};
%left labels
\node[left] at (1,0) {$1$};
\node[left] at (1.5,0.6) {$\vdots$};
\node[left] at (1,1) {$1$};
\node[left] at (1,2) {$\vdots$};
%\node[left] at (1.5,2.6) {$\vdots$};
\node[left] at (1,3) {$\vdots$};
%\node at (2.5,3.6) {$\vdots$};
%\node at (6.5,3.6) {$\vdots$};
\node[left] at (1,4) {$n$};
\node[left] at (1.5,4.6) {$\vdots$};
\node[left] at (1,5) {$n$};
}
\end{align}
where (as above) $\bm{A}(k) = \sum_{i=1}^{n} \bm{1}_{k \in \mu^{(i)}} \bm{e}_i$ for all $1 \leq i \leq n$ and $k \geq 0$.

The functions $f_{\mu}(\lambda;x_1,\dots,x_m;s)$ and $g_{\mu}(\lambda;x_1,\dots,x_m;s)$ are also not new; they were introduced in \cite{ABW21}. They have a number of key properties, including exchange relations under the action of Hecke algebra (Section \ref{ssec:hecke}) and antisymmetrization identities (Section \ref{ssec:anti}). They also have meaningful $s=0$ degenerations, when they both reduce to (certain antisymmetrizations of) nonsymmetric Hall--Littlewood polynomials. Moreover, the $s=0$ degenerations of $f_{\mu}(\lambda;x_1,\dots,x_m;s)$ and $g_{\mu}(\lambda;x_1,\dots,x_m;s)$ pair together to provide an integral formula for the LLT polynomials; it is the latter fact that shall be of most interest to us in the current text.

\subsection{Two formulas for LLT polynomials}

In this section we recall two formulas for the LLT polynomials. The first is as partition functions in a fermionic $U_q (\widehat{\mathfrak{sl}}(1|n))$ vertex model, following \cite{CGKM22,ABW21}. The second is as a contour integral, following \cite{ABW21}.

We begin with the partition function representation of the LLT polynomials. To state it, we extend our previous notion of vertex models to the situation where both horizontal and vertical edges may admit multiple paths of distinct colours; as such, every edge of the underlying square grid is now labelled by an $n$-dimensional binary string which specifies whether each colour $\{1,\dots,n\}$ appears (or not) at that edge. For arbitrary binary strings $\A = (A_1,\dots,A_n)$, $\B = (B_1,\dots,B_n)$, $\C = (C_1,\dots,C_n)$, $\D = (D_1,\dots,D_n)$ we then introduce the vertex weights
\begin{align}
\label{LLT-weights-intro}
\tikz{0.7}{
\node[left] at (-1.5,0) {$x \rightarrow $};
\draw[lgray,line width=4pt,->] (-1,0) -- (1,0);
\draw[lgray,line width=4pt,->] (0,-1) -- (0,1);
\node[left] at (-1,0) {\tiny $\B$};\node[right] at (1,0) {\tiny $\D$};
\node[below] at (0,-1) {\tiny $\A$};\node[above] at (0,1) {\tiny $\C$};
}
=
\bm{1}_{\bm{C}+\bm{D} \in \{0,1\}^n}
\cdot
x^{|\D|}
q^{\varphi(\D,\C)+\varphi(\D,\D)}
\qquad
\A,\B,\C,\D \in \{0,1\}^n,
\end{align}
where $|\D| = \sum_{i=1}^{n} D_i$ and $\varphi(\bm{X},\bm{Y}) = \sum_{1 \leq i<j \leq n} X_i Y_j$ for any two vectors $\bm{X},\bm{Y} \in \mathbb{Z}^n$.

Fix a composition $\lambda = (\lambda_1,\dots,\lambda_n)$ and two coloured compositions 
$\mu,\nu \in \mathcal{S}_{\lambda}$. The skew LLT (symmetric) polynomial $\mathbb{G}_{\mu/\nu}(\lambda;x_1,\dots,x_p)$ is given by the following partition function in the model \eqref{LLT-weights-intro}:
\begin{align}
\label{G-pf-fused-intro}
\mathbb{G}_{\mu/\nu}(\lambda;x_1,\dots,x_p)
&=
\tikz{0.75}{
\foreach\y in {1,...,5}{
\draw[lgray,line width=4pt,->] (1,\y) -- (8,\y);
}
\foreach\x in {2,...,7}{
\draw[lgray,line width=4pt,->] (\x,0) -- (\x,6);
}
%spectral parameters
\node[left] at (0.5,1) {$x_1 \rightarrow$};
\node[left] at (0.5,2) {$x_2 \rightarrow$};
\node[left] at (0.5,3) {$\vdots$};
\node[left] at (0.5,4) {$\vdots$};
\node[left] at (0.5,5) {$x_p \rightarrow$};
%top labels
\node[above] at (7,6) {$\cdots$};
\node[above] at (6,6) {$\cdots$};
\node[above] at (5,6) {$\cdots$};
\node[above] at (4,6) {\footnotesize$\bm{A}(2)$};
\node[above] at (3,6) {\footnotesize$\bm{A}(1)$};
\node[above] at (2,6) {\footnotesize$\bm{A}(0)$};
%bottom labels
\node[below] at (7,0) {$\cdots$};
\node[below] at (6,0) {$\cdots$};
\node[below] at (5,0) {$\cdots$};
\node[below] at (4,0) {\footnotesize$\bm{B}(2)$};
\node[below] at (3,0) {\footnotesize$\bm{B}(1)$};
\node[below] at (2,0) {\footnotesize$\bm{B}(0)$};
%right labels
\node[right] at (8,1) {$\bm{e}_0$};
\node[right] at (8,2) {$\bm{e}_0$};
\node[right] at (8,3) {$\vdots$};
\node[right] at (8,4) {$\vdots$};
\node[right] at (8,5) {$\bm{e}_0$};
%left labels
\node[left] at (1,1) {$\bm{e}_0$};
\node[left] at (1,2) {$\bm{e}_0$};
\node[left] at (1,3) {$\vdots$};
\node[left] at (1,4) {$\vdots$};
\node[left] at (1,5) {$\bm{e}_0$};
}
\end{align}
where $\bm{A}(k) = \sum_{i=1}^{n} \bm{1}_{k \in \mu^{(i)}} \bm{e}_i$, $\bm{B}(k) = \sum_{i=1}^{n} \bm{1}_{k \in \nu^{(i)}} \bm{e}_i$ for all $1 \leq i \leq n$ and $k \geq 0$. As with our previous partition functions, there is a simple lattice path interpretation of \eqref{G-pf-fused-intro}: for each $1 \leq i \leq n$, a collection of $\lambda_i$ paths of colour $i$ enter the partition function \eqref{G-pf-fused-intro} via the base of columns $\nu^{(i)}_1 < \cdots < \nu^{(i)}_{\lambda_i}$ and exit at the top of columns $\mu^{(i)}_1 < \cdots < \mu^{(i)}_{\lambda_i}$. As such, \eqref{G-pf-fused-intro} provides a realization of the LLT polynomials in terms of $n$ overlapping ensembles of non-intersecting lattice paths. 

\begin{thm}[Theorem \ref{thm:LLT-int} below]
Fix a composition $\lambda = (\lambda_1,\dots,\lambda_n)$ such that $\sum_{i=1}^{n} \lambda_i = m$, and choose two coloured compositions $\mu,\nu \in \mathcal{S}_{\lambda}$. The LLT polynomials \eqref{G-pf-fused-intro} are given by the following integral expression:
\begin{multline}
\label{LLTint-intro}
\mathbb{G}_{\mu/\nu}(\lambda;x_1,\dots,x_p)
=
\frac{q^{m(m+1)/2}}{(q-1)^{m}}
\cdot
\left( \frac{1}{2\pi{\tt i}} \right)^{m}
\oint_{C_1}
\frac{dy_1}{y_1}
\cdots 
\oint_{C_m}
\frac{dy_m}{y_m}
\\
\times
\prod_{1 \leq i<j \leq m}
\left(
\frac{y_j-y_i}{y_j-q y_i}
\right)
f_{\breve\mu}(1^m;y_1^{-1},\dots,y_m^{-1};0)
g_{\nu}(\lambda;y_1,\dots,y_m;0)
\prod_{i=1}^{p}
\prod_{j=1}^{m}
\frac{1}{1-x_i y_j},
\end{multline}
where the contours $\{C_1,\dots,C_m\}$ are certain $q$-nested contours that all surround the origin; see the discussion at the start of Section \ref{ssec:orthog}. We have also used the notation $1^m = (1,\dots,1)$ (where $1$ appears with multiplicity $m$) and have defined $\breve\mu$ to be the unique element of $\mathcal{S}_{1^m}$ obtained by ordering the parts of $\mu$ in increasing order; see Definition \ref{def:rainbow-rec}.
\end{thm}

Throughout most of the text, we consider LLT polynomials \eqref{G-pf-fused-intro} in which $\lambda = N^n$ for some $N \geq 1$; that is, each colour is represented exactly $N$ times within the partition function \eqref{G-pf-fused-intro}. Whenever we make this choice, we write $$\mathbb{G}_{\mu/\nu}(N^n;x_1,\dots,x_p) = \mathbb{G}_{\mu/\nu}(x_1,\dots,x_p).$$ We also assign a special notation to the coloured composition in $\mathcal{S}_{N^n}$ whose parts are as small as they can be, by writing
\begin{align}
\label{nothing-intro}
\Delta
=
(0,1,\dots,N-1 | 0,1,\dots,N-1 | \cdots | 0,1,\dots,N-1)
\in 
\mathcal{S}_{N^n}.
\end{align}

\subsection{LLT Cauchy identity and Markov kernels}

The Markov kernels that we study in this work are built from the (skew) Cauchy identity for the LLT polynomials \cite{Lam,CGKM22,ABW21}:

\begin{thm}[Theorem \ref{thm:cauchy} below]
\label{thm:cauchy-intro}
Fix two positive integers $p$ and $N$, and two alphabets $(x_1,\dots,x_p)$ and $(y_1,\dots,y_N)$. Let $\nu \in \mathcal{S}_{N^n}$ be a coloured composition. The LLT polynomials \eqref{G-pf-fused-intro} satisfy the Cauchy summation identity
\begin{align}
\label{skew-cauchy-intro}
\sum_{\mu \in \mathcal{S}_{N^n}}
q^{-2\psi(\mu)}
\mathbb{G}_{\mu/\nu}(x_1,\dots,x_p)
\mathbb{G}_{\mu}(y_1,\dots,y_N)
=
\prod_{i=1}^{p}
\prod_{j=1}^{N}
\frac{1}{(x_i y_j;q)_n}
\cdot
q^{-2\psi(\nu)}
\mathbb{G}_{\nu}(y_1,\dots,y_N),
\end{align}
where $(z;q)_n = \prod_{k=1}^{n} (1-q^{k-1}z)$ denotes the standard $q$-Pochhammer function, the exponents on the left and right hand side are defined as
\begin{align*}
\psi(\mu)
=
\frac{1}{2}
\sum_{1 \leq i<j \leq n}\
\sum_{a \in \mu^{(i)}}\
\sum_{b \in \mu^{(j)}}
\bm{1}_{a>b},
\end{align*}
and $\mathbb{G}_{\mu}(y_1,\dots,y_N) \equiv \mathbb{G}_{\mu/\Delta}(y_1,\dots,y_N)$. This holds either as a formal power series, or as a numeric equality as long as $|q| <1$ and $|x_i y_j| <1$ for all $i,j$.
\end{thm}

If one divides equation \eqref{skew-cauchy-intro} by the expression appearing on the right hand side, the resulting summands comprise a probability measure on coloured compositions $\mu \in \mathcal{S}_{N^n}$, assuming that they are nonnegative. One simple choice of the alphabets $(x_1,\dots,x_p)$ and $(y_1,\dots,y_N)$ which respects this nonnegativity requirement is to set the former all equal to $1$ and the latter to a Plancherel specialization ${\rm Pl}_t$, where $t \in \mathbb{R}_{>0}$ (see Section \ref{ssec:planch}). This choice will be our sole focus in the current work; we denote the resulting Markov kernels as follows:
\begin{align}
\label{skew-cauchy3-intro}
\mathbb{P}_{t,p}(\nu \rightarrow \mu)
=
q^{-2(\psi(\mu)-\psi(\nu))}
\exp\left( -\frac{p(1-q^n)}{1-q}t \right)
\mathbb{G}_{\mu/\nu}(1^p)
\dfrac{\mathbb{G}_{\mu}({\rm Pl}_t)}
{\mathbb{G}_{\nu}({\rm Pl}_t)},
\end{align}
for any pair of coloured compositions $\mu,\nu \in \mathcal{S}_{N^n}$. In particular, will be interested in strings of random coloured compositions generated by the repeated action of $\mathbb{P}_{t,1}$ on the initial state \eqref{nothing-intro}:
\begin{align}
\label{chain-intro}
\Delta
\xrightarrow{\mathbb{P}_{t,1}}
\lambda^{[1]}
\xrightarrow{\mathbb{P}_{t,1}}
\cdots
\xrightarrow{\mathbb{P}_{t,1}}
\lambda^{[m]}
\xrightarrow{\mathbb{P}_{t,1}}
\lambda^{[m+1]}
\xrightarrow{\mathbb{P}_{t,1}}
\cdots
\xrightarrow{\mathbb{P}_{t,1}}
\lambda^{[N]}.
\end{align}
It is worth noting that any individual coloured composition within \eqref{chain-intro} is distributed according to a non-skew version of \eqref{skew-cauchy3-intro}; see, in particular, Proposition \ref{prop:non-skew} of the text. 

Our main result is a complete description of the asymptotic behaviour of the coloured compositions $\lambda^{[i]}$, $1 \leq i \leq N$, as $t \rightarrow \infty$ (with the number of steps in the chain \eqref{chain-intro} remaining finite).

\subsection{Asymptotic analysis of Markov kernels}
\label{intro:asymptotics}

Before proceeding with the asymptotics, we introduce a convenient way to encode the coloured compositions appearing in the chain \eqref{chain-intro}; we focus our attention on two neighbours in this sequence, namely $\lambda^{[m]}$ and $\lambda^{[m+1]}$. We shall begin with the assumption that these coloured compositions have pairwise distinct parts.\footnote{This is the first of several assumptions that we make prior to performing our asymptotic analysis. The justification for these assumptions is an {\it a posteori} one: any sequence of coloured compositions \eqref{chain-intro} which violates our assumptions will be shown to take up a vanishingly small part of the measure, in the limit $t \rightarrow \infty$.}

Considering firstly $\lambda^{[m]} \in \mathcal{S}_{m^n}$, we see that it may be expressed uniquely in terms of its {\it coordinates} $\ell^{[m]} = \left\{\ell^{[m]}_1 < \cdots < \ell^{[m]}_{nm}\right\} \subset \mathbb{Z}_{\geq 0}$, which are simply the parts of $\lambda^{[m]}$ listed in increasing order, and its {\it colour sequence} $c^{[m]} = \left(c^{[m]}_1,\dots,c^{[m]}_{nm}\right) \in \{1,\dots,n\}^{nm}$, which is a vector that records the colour $c^{[m]}_i$ attributed to the part $\ell^{[m]}_i$ as it occurs within $\lambda^{[m]}$; the reader is referred to Definition \ref{def:coord} for a precise formulation of these objects. After performing a similar identification for $\lambda^{[m+1]} \in \mathcal{S}_{(m+1)^n}$, we have the correspondences
\begin{align}
\label{identify-intro}
\lambda^{[m]}
\leftrightarrow 
\left( 
\ell^{[m]}_1,\dots,\ell^{[m]}_{nm} 
\Big| 
c^{[m]}_1,\dots,c^{[m]}_{nm} 
\right),
\qquad
\lambda^{[m+1]}
\leftrightarrow 
\left( 
\ell^{[m+1]}_1,\dots,\ell^{[m+1]}_{n(m+1)} 
\Big| 
c^{[m+1]}_1,\dots,c^{[m+1]}_{n(m+1)} 
\right),
\end{align} 
and work directly with the right hand sides of these expressions in our calculations.

Our next assumption concerning the coordinates $\ell^{[m]}$ and $\ell^{[m+1]}$ is that, in the limit $t \rightarrow \infty$, they arrange into $n$ {\it interlacing bundles} as follows:
\begin{align}
\label{interlace-disc-intro}
\ell^{[m+1]}_{j(m+1)+i} < \ell^{[m]}_{jm+i} < \ell^{[m+1]}_{j(m+1)+i+1},
\qquad
\forall\ i \in \{1,\dots,m\},\quad j \in \{0,\dots,n-1\}.
\end{align}
A schematic illustration of such an arrangement, for $n=3$ and varying $m$, is provided below (see also Figure \ref{fig:gue} in the main body of the text):
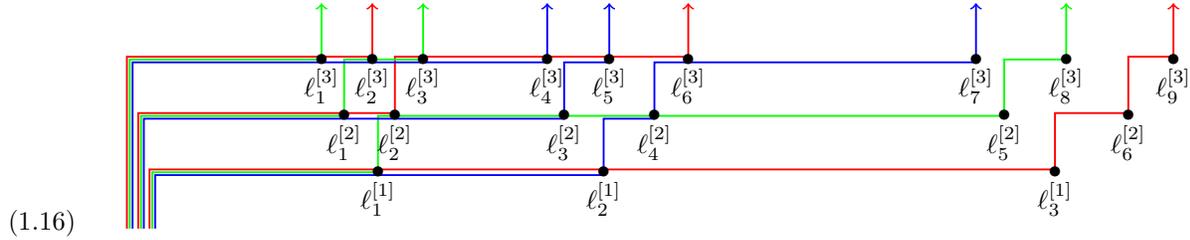
\begin{align}
\label{interlace-picture}
\begin{tikzpicture}[scale=0.75]
%level 1
\draw[red,line width=0.7pt,->] (-0.05,1) -- (-0.05,2.05) -- (16,2.05) -- (16,3.05) -- (17.3,3.05) -- (17.3,4.05) -- (18.1,4.05) -- (18.1,5);
\draw[green,line width=0.7pt,->] (0,1) -- (0,2) -- (4,2) -- (4,3) -- (15.1,3) -- (15.1,4) -- (16.2,4) -- (16.2,5);
\draw[blue,line width=0.7pt,->] (0.05,1) -- (0.05,1.95) -- (8,1.95) -- (8,2.95) -- (8.9,2.95) -- (8.9,3.95) -- (14.6,3.95) -- (14.6,5);
%%level 2
\draw[red,line width=0.7pt,->] (-0.25,1) -- (-0.25,3.05) -- (4.3,3.05) -- (4.3,4.05) -- (9.5,4.05) -- (9.5,5);
\draw[green,line width=0.7pt,->] (-0.2,1) -- (-0.2,3) -- (3.4,3) -- (3.4,4) -- (4.8,4) -- (4.8,5);
\draw[blue,line width=0.7pt,->] (-0.15,1) -- (-0.15,2.95) -- (7.3,2.95) -- (7.3,3.95) -- (8.1,3.95) -- (8.1,5);
%%level 3
\draw[red,line width=0.7pt,->] (-0.45,1) -- (-0.45,4.05) -- (3.9,4.05) -- (3.9,5);
\draw[green,line width=0.7pt,->] (-0.4,1) -- (-0.4,4) -- (3,4) -- (3,5);
\draw[blue,line width=0.7pt,->] (-0.35,1) -- (-0.35,3.95) -- (7,3.95) -- (7,5);
%nodes
\node at (4,2) {$\bullet$};
\node[below] at (4,2) {$\ell_1^{[1]}$};
\node at (8,2) {$\bullet$};
\node[below] at (8,2) {$\ell_2^{[1]}$};
\node at (16,2) {$\bullet$};
\node[below] at (16,2) {$\ell_3^{[1]}$};
\node at (3.4,3) {$\bullet$};
\node[below] at (3.4,3) {$\ell_1^{[2]}$};
\node at (4.3,3) {$\bullet$};
\node[below] at (4.3,3) {$\ell_2^{[2]}$};
\node at (7.3,3) {$\bullet$};
\node[below] at (7.3,3) {$\ell_3^{[2]}$};
\node at (8.9,3) {$\bullet$};
\node[below] at (8.9,3) {$\ell_4^{[2]}$};
\node at (15.1,3) {$\bullet$};
\node[below] at (15.1,3) {$\ell_5^{[2]}$};
\node at (17.3,3) {$\bullet$};
\node[below] at (17.3,3) {$\ell_6^{[2]}$};
\node at (3,4) {$\bullet$};
\node[below] at (3,4) {$\ell_1^{[3]}$};
\node at (3.9,4) {$\bullet$};
\node[below] at (3.9,4) {$\ell_2^{[3]}$};
\node at (4.8,4) {$\bullet$};
\node[below] at (4.8,4) {$\ell_3^{[3]}$};
\node at (7,4) {$\bullet$};
\node[below] at (7,4) {$\ell_4^{[3]}$};
\node at (8.1,4) {$\bullet$};
\node[below] at (8.1,4) {$\ell_5^{[3]}$};
\node at (9.5,4) {$\bullet$};
\node[below] at (9.5,4) {$\ell_6^{[3]}$};
\node at (14.6,4) {$\bullet$};
\node[below] at (14.6,4) {$\ell_7^{[3]}$};
\node at (16.2,4) {$\bullet$};
\node[below] at (16.2,4) {$\ell_8^{[3]}$};
\node at (18.1,4) {$\bullet$};
\node[below] at (18.1,4) {$\ell_9^{[3]}$};
\end{tikzpicture}
\end{align}
More precisely, we will assume that the coordinates $\ell^{[m]}$ and $\ell^{[m+1]}$ scale as
\begin{align}
\label{coord-scal-intro}
\ell^{[k]}_i
\mapsto
q^{n-\lceil i/k \rceil} t
+
(q^{n-\lceil i/k \rceil} t)^\frac{1}{2}
x^{[k]}_i,
\qquad
1 \leq i \leq nk,
\qquad
k \in \{m,m+1\},
\end{align}
as $t \rightarrow \infty$. Here $x^{[m]} = \left\{x^{[m]}_1 < \cdots x^{[m]}_{nm}\right\}$ and 
$x^{[m+1]} = \left\{x^{[m+1]}_1 < \cdots < x^{[m+1]}_{n(m+1)} \right\}$ are sets of reals that obey the relations \eqref{interlace-disc-intro} (with $\ell$ replaced by $x$), while $\lceil i/k \rceil$ denotes the ceiling function.

Let $\theta_1^{[j]} \leq \cdots \leq \theta_j^{[j]}$ denote the eigenvalues of the top-left $j \times j$ corner of a random $N \times N$ matrix in the Gaussian Unitary Ensemble. The joint law of the eigenvalues $\theta_i^{[j]}$, $1 \leq i \leq j$, $j \in [1,N]$ is known as the {\it GUE corners process} of rank $N$. We let
\begin{align*}
\rho_{\rm GUE}\left(
x^{[1]}\prec \cdots \prec x^{[N]} \right)
&
:=
\rho\left(\theta_i^{[j]} = x_i^{[j]},1 \leq i \leq j \leq N\right)
\end{align*}
denote the associated joint probability density, and write
\begin{align*}
\rho_{\rm GUE}\left(
x^{[m]}
\rightarrow
x^{[m+1]}
\right)
&
:=
\rho\left(
\theta_i^{[m+1]} = x_i^{[m+1]}, 1 \leq i \leq m+1 
\Big| 
\theta_i^{[m]} = x_i^{[m]}, 1 \leq i \leq m
\right)
\end{align*}
for the conditional probability density for the eigenvalues of top-left $(m+1) \times (m+1)$ corner, given those of the $m \times m$ one. See \cite{gorin-notes} and Section \ref{ssec:main} of the current text for more information on these definitions.

We are now able to state the main result of this paper.

\begin{thm}[Theorem \ref{thm:main} below]
\label{thm:main-intro}
In the asymptotic regime described by \eqref{coord-scal-intro}, the Markov kernel $\mathbb{P}_{t,1}$ weakly converges to a product of $n$ independent probability measures with densities in the GUE corners process, multiplied by a factor that depends only on the colour sequences \eqref{identify-intro}:
\begin{multline}
\label{main-result-intro}
\mathbb{P}_{t,1}\left(0\cup\lambda^{[m]} \rightarrow \lambda^{[m+1]}\right)
\\
\rightarrow
\prod_{i=1}^{n}
\rho_{\rm GUE}\left(
x_{(i-1)m+1}^{[m]},\dots,x_{im}^{[m]}
\rightarrow
x_{(i-1)(m+1)+1}^{[m+1]},\dots,x_{i(m+1)}^{[m+1]}
\right)
dx^{[m+1]}
\cdot
\mathbb{P}_{\rm col}\left(c^{[m]} \rightarrow c^{[m+1]}\right)
\end{multline}
as $t \rightarrow \infty$, where $dx^{[m+1]}$ denotes the $n(m+1)$-dimensional Lebesgue measure. The final multiplicative factor in \eqref{main-result-intro} is given explicitly by equation \eqref{col-markov-intro} below, and defines a discrete transition probability in a process on colour sequences:
\begin{align}
\label{discrete-sum-to1-intro}
\sum_{c^{[m+1]}}
\mathbb{P}_{\rm col}\left(c^{[m]} \rightarrow c^{[m+1]}\right)
=
1,
\end{align}
where the sum is taken over all $c^{[m+1]} \in \{1,\dots,n\}^{n(m+1)}$.
\end{thm}

Our proof of Theorem \ref{thm:main-intro} is by explicit analysis of \eqref{skew-cauchy3-intro} at $p=1$, employing the lattice model formula \eqref{G-pf-fused-intro} for the factor $\mathbb{G}_{\mu/\nu}(1)$ and (a Plancherel-specialized version of) the integral formula \eqref{LLTint-intro} for the functions $\mathbb{G}_{\mu}({\rm Pl}_t)$ and $\mathbb{G}_{\nu}({\rm Pl}_t)$. The study of the latter integrals proceeds by steepest descent analysis, combined with certain crucial algebraic properties of the functions \eqref{f-def-intro} and \eqref{g-def-intro} which appear in their integrands. As $t \rightarrow \infty$ one observes a remarkable factorization of these integrals into purely coordinate dependent and colour sequence dependent parts; the former can then be matched directly with transition densities for the GUE corners process. At the end of this procedure we have a leftover factor valued on colour sequences (see the second line of equation \eqref{final-formula2}, below) and {\it a priori} it is by no means obvious that this quantity defines a valid discrete probability measure. Resolution of this particular issue is the topic of Section \ref{sec:discrete-dist} (see also Section \ref{ssec:intro-triangle}, below). 

As a direct consequence of Theorem \ref{thm:main-intro} we obtain the following corollary, completely describing the behaviour of the chain of coloured compositions \eqref{chain-intro} as $t \rightarrow \infty$:
\begin{cor}[Corollary \ref{cor:main} below]
\label{cor:main-intro}
Let $\mathbb{P}_{t,N} (\Delta \rightarrow \lambda^{[1]} \rightarrow \cdots \rightarrow \lambda^{[N]})$ denote the joint distribution of coloured compositions $\lambda^{[1]},\dots, \lambda^{[N]}$ generated by $N$ applications of the kernel $\mathbb{P}_{t,1}$ to the trivial state $\Delta$. In the asymptotic regime described by \eqref{coord-scal-intro}, we have the following weak convergence of measures:
\begin{multline}
\label{cor-joint}
\mathbb{P}_{t,N}
\left(\Delta \rightarrow \lambda^{[1]} \rightarrow \cdots \rightarrow \lambda^{[N]}\right)
\\
\rightarrow
\prod_{i=1}^{n}
\rho_{\rm GUE}
\left( (x^{[1]})_i \prec (x^{[2]})_i \prec \cdots \prec (x^{[N]})_i  \right)
dx^{[1,N]}
\cdot
\mathbb{P}_{\rm col}\left(c^{[1]} \prec c^{[2]} \prec \cdots \prec c^{[N]}\right)
\end{multline}
as $t \rightarrow \infty$, with $dx^{[1,N]} = \prod_{i=1}^{N} dx^{[i]}$ denoting the $nN(N+1)/2$-dimensional Lebesgue measure. Here we have introduced the shorthand
\begin{align*}
\left(x^{[k]}\right)_i = \left(x^{[k]}_{(i-1)k+1},\dots,x^{[k]}_{ik}\right),
\qquad
\forall\ 1 \leq i \leq n,\ \ 1 \leq k \leq N,
\end{align*}
and $\mathbb{P}_{\rm col}(c^{[1]} \prec c^{[2]} \prec \cdots \prec c^{[N]})$ is a joint distribution on colour sequences given explicitly by \eqref{joint-distr-col-intro} below.
\end{cor}

\subsection{Distribution on interlacing triangles}
\label{ssec:intro-triangle}

While Theorem \ref{thm:main-intro} and Corollary \ref{cor:main-intro} provide a complete description of the asymptotic behaviour of coordinates of the coloured compositions \eqref{chain-intro} as $t \rightarrow \infty$, we are left with the task of understanding the factors $\mathbb{P}_{\rm col}\left(c^{[m]} \rightarrow c^{[m+1]}\right)$ and $\mathbb{P}_{\rm col}\left(c^{[1]} \prec c^{[2]} \prec \cdots \prec c^{[N]}\right)$ that occur therein. These factors provide information about how colours distribute themselves within interlacing diagrams of the form \eqref{cor:main-intro}, as $t \rightarrow \infty$.

Let $i^{[m]} \in \{1,\dots,n\}^{nm}$ and $j^{[m+1]} \in \{1,\dots,n\}^{n(m+1)}$ be two sequences such that each colour $\{1,\dots,n\}$ is represented exactly $m$ times in $i^{[m]}$ and $m+1$ times in $j^{[m+1]}$. We say that these colour sequences interlace, and write $i^{[m]} \prec j^{[m+1]}$, provided they can be stacked to form an {\it admissible diagram}:
\begin{align}
\label{admiss-diag}
\tikz{1.3}{
\draw[lgray,line width=4pt,<-] (0.5,0) -- (12,0);
\foreach\x in {1,2,3,5,6,7,9,10,11}{
\draw[lgray,line width=1.5pt,->] (\x,0) -- (\x,0.5);
}
\node[above] at (1,0.5) {$j_1$};
\node[above] at (2,0.5) {$\cdots$};
\node[above] at (3,0.5) {$j_{m+1}$};
\node[above] at (5,0.5) {$\cdots$};
\node[above] at (6,0.5) {$\cdots$};
\node[above] at (7,0.5) {$\cdots$};
\node[above] at (9,0.5) {$j_{(n-1)(m+1)+1}$};
\node[above] at (10.2,0.5) {$\cdots$};
\node[above] at (11,0.5) {$j_{n(m+1)}$};
\foreach\x in {1.5,2.5,5.5,6.5,9.5,10.5}{
\draw[lgray,line width=1.5pt,->] (\x,-0.5) -- (\x,0);
}
\node[below] at (1.5,-0.5) {$i_1$};
\node[below] at (2,-0.5) {$\cdots$};
\node[below] at (2.5,-0.5) {$i_m$};
\node[below] at (5.5,-0.5) {$\cdots$};
\node[below] at (6.5,-0.5) {$\cdots$};
\node[below] at (9.2,-0.5) {$i_{(n-1)m+1}$};
\node[below] at (10.1,-0.5) {$\cdots$};
\node[below] at (10.7,-0.5) {$i_{nm}$};
\node[left] at (0.5,0) {$\emptyset$}; \node[right] at (12,0) {$[1,n]$};
}
\end{align}
In the above diagram the incoming/outgoing vertical arrows are grouped into a total of $n$ bundles, each of width $m$ or $m+1$, respectively. The colours $i^{[m]}=(i_1,\dots,i_{nm})$ enter sequentially via the arrows at the base, while colours $j^{[m+1]}=(j_1,\dots,j_{n(m+1)})$ exit sequentially via the arrows at the top.  A copy of all colours $[1,n] \equiv \{1,\dots,n\}$ enters via the right, and no colours exit via the left. The diagram is admissible provided that, after one draws the trajectories of all coloured paths, each colour $\{1,\dots,n\}$ never occurs more than once at any point along the thick horizontal line. 

Given two colour sequences $i^{[m]} \prec j^{[m+1]}$ we define a statistic $\xi\left( i^{[m]}; j^{[m+1]} \right)$ which enumerates the number of events of the form
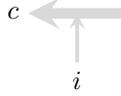
\begin{align*}
\tikz{1.3}{
\draw[lgray,line width=4pt,<-] (0.5,0) -- (1.5,0);
\draw[lgray,line width=1.5pt,->] (1,-0.5) -- (1,0);
\node[below] at (1,-0.5) {$i$};
\node[left] at (0.5,0) {$c$};
}
\end{align*}
in a path of colour $c$ passes over a colour $i$, with $c>i$, within the diagram \eqref{admiss-diag}.

\begin{thm}
\label{thm:discr-distr}
The factor $\mathbb{P}_{\rm col} \left( c^{[m]} \rightarrow c^{[m+1]} \right)$ appearing in \eqref{main-result-intro} is given by
\begin{multline}
\label{col-markov-intro}
\mathbb{P}_{\rm col} 
\left( c^{[m]} \rightarrow c^{[m+1]} \right)
\\
=
\bm{1}_{c^{[m]} \prec c^{[m+1]}}
(-1)^n
q^{\binom{nm+n+1}{2}-\binom{nm+1}{2}-\xi\left( c^{[m]}; c^{[m+1]} \right)}
\frac{(1-q)^{nm}}{(q;q)_n^{2m+1}}
\frac{
g^{c^{[m+1]}}_{\Delta}\left((m+1)^n;\vec{Q}^{[m+1]}\right)
}
{
g^{c^{[m]}}_{\Delta}\left(m^n;\vec{Q}^{[m]}\right)
},
\end{multline}
where $g^{c^{[m]}}_{\Delta}\left(m^n;\vec{Q}^{[m]}\right)$ denotes a partition function of the form \eqref{g-def-intro}, with $m \mapsto nm$, $\lambda = m^n$, $\mu=\Delta$, $s=0$,
\begin{align*}
(x_1,\dots,x_{nm})
=
\underbrace{(q^{n-1},\dots,q^{n-1})}_{m\ {\rm times}}
\cup
\cdots
\cup
\underbrace{(q,\dots,q)}_{m\ {\rm times}}
\cup
\underbrace{(1,\dots,1)}_{m\ {\rm times}}
\equiv
\vec{Q}^{[m]},
\end{align*}
and in which colour $c^{[m]}_i$ exits via the left edge of row $i$ within \eqref{g-def-intro} (rather than in totally ordered fashion). An analogous definition applies to $g^{c^{[m+1]}}_{\Delta}\left((m+1)^n;\vec{Q}^{[m+1]}\right)$. The expression \eqref{col-markov-intro} constitutes a valid discrete transition probability; namely, it satisfies the sum-to-unity property \eqref{discrete-sum-to1-intro}.
\end{thm}

The sum-to-unity property in Theorem \ref{thm:discr-distr} is not immediate and plays a substantial role in the proof of Theorem \ref{thm:main-intro} and Corollary \ref{cor:main-intro}. In particular, as briefly mentioned in the beginning of Section \ref{intro:asymptotics}, we only compute the asymptotics of the Markov kernel \eqref{skew-cauchy3-intro} under a certain ansatz for the behaviour of the coordinates $\ell^{[m]},\ell^{[m+1]}$ and colour sequences $c^{[m]},c^{[m+1]}$. To show that this ansatz asymptotically exhausts the full measure induced by the kernel \eqref{skew-cauchy3-intro} requires the above sum-to-unity property, whose proof hinges upon a rather unusual expansion property of the partition functions in question; see, in particular, Theorem \ref{thm:expand}. The main tool behind this proof are commutation relations between the row operators used to build our partition functions, which in turn are a consequence of the underlying Yang--Baxter integrability.  

More generally, one may consider collections of colour sequences $\emptyset \prec c^{[1]} \prec \cdots \prec c^{[N]}$ such that for all $1 \leq k \leq N$, $c^{[k]} \in \{1,\dots,n\}^{nk}$ and each colour $\{1,\dots,n\}$ is represented exactly $k$ times in $c^{[k]}$. We refer to such a collection of positive integers as an {\it interlacing triangular array} of {\it rank} $n$ and {\it height} $N$, and let $\mathcal{T}_N(n)$ denote the set of all such objects; see Definition \ref{def:interlace} for a more precise formulation.

As a direct consequence of Theorem \ref{thm:discr-distr} we obtain the following result:
\begin{cor}[Corollary \ref{cor:joint-col} below]
Let $\emptyset \prec c^{[1]} \prec \cdots \prec c^{[N]}$ be an interlacing triangular array generated by $N$ successive applications of the Markov kernel \eqref{col-markov-intro} on the empty sequence $\emptyset$. This array has joint distribution
\begin{align}
\label{joint-distr-col-intro}
\mathbb{P}_{\rm col}\left(c^{[1]} \prec \cdots \prec c^{[N]}\right)
=
\bm{1}_{c^{[1]} \prec \cdots \prec c^{[N]}}
(-1)^{nN} q^{\binom{nN+1}{2}}
\frac{(1-q)^{n \binom{N}{2}}}{(q;q)_{n}^{N^2}}
g^{c^{[N]}}_{\Delta}\left(N^n;\vec{Q}^{[N]}\right)
\prod_{i=1}^{N}
q^{-\xi \left( c^{[i-1]} ; c^{[i]} \right)}.
\end{align}
\end{cor}

\subsection{Positivity and enumeration conjectures}

A number of interesting observations arise concerning the measure \eqref{joint-distr-col-intro}, as well as the set $\mathcal{T}_N(n)$ of interlacing triangular arrays on which it is supported. The first is a positivity property that we noticed from explicit implementation of the Markov kernel \eqref{col-markov-intro} on a computer: 
\begin{conj}[Conjecture \ref{conj:pos} below]
\label{conj:pos-intro}
Fix integers $m,n \geq 1$ and a colour sequence $c^{[m]} \in \{1,\dots,n\}^{nm}$. Let $\mathbb{P}_{\rm col}(c^{[m]})$ denote the probability of arriving at the colour sequence $c^{[m]}$ after $m$ applications of the Markov kernel \eqref{col-markov-intro} to the trivial sequence $c^{[0]}=\emptyset$. Then one has that
\begin{align}
\label{pos-conj-intro}
\mathbb{P}_{\rm col}\left(c^{[m]}\right)
=
\mathcal{P}\left(c^{[m]}\right)
\cdot
\left( \prod_{i=1}^{n} \frac{1-q}{1-q^i} \right)^{m^2}
\quad
\text{where}\ \ 
\mathcal{P}\left(c^{[m]}\right) \in \mathbb{N}[q].
\end{align}
\end{conj}
In fact, one sees that \eqref{pos-conj-intro} expresses $\mathbb{P}_{\rm col}(c^{[m]})$ as a ratio of two positive polynomials in $q$; the denominator is nothing but the Poincar\'e polynomial associated to $\mathfrak{S}_n$ raised to the power $m^2$. An explicit illustration of this conjecture, for $n=2$, is given in Figure \ref{fig:n=2}. At this stage we do not know of any combinatorial interpretation of 
$\mathcal{P}\left(c^{[m]}\right)$, although it would be very interesting to find one.

There is also the purely combinatorial problem of enumerating the number of elements in the set $\mathcal{T}_N(n)$. It a trivial fact that $|\mathcal{T}_N(1)| = 1$,\footnote{In the case $n=1$, LLT measures degenerate to their Schur counterparts. In that situation, the asymptotic analysis carried through in this text leads to a single GUE corners process, which has a trivial interlacing $1$-colouring.} and one can easily show that $|\mathcal{T}_N(2)|=2^N$; see Proposition \ref{prop:n=2}. While for $n \geq 3$ we have no direct enumeration of $|\mathcal{T}_N(n)|$, we do present two conjectures relating to $(n+1)$-colourings of certain graphs:
\begin{conj}[Conjecture \ref{conj:triangle} below]
\label{conj:triangle-intro}
Let $G^{\triangle}_N$ denote the triangular graph
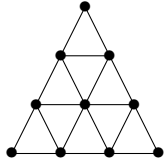
\begin{align*}
\tikz{1.3}{
\node at (0,0) {$\bullet$};
\node at (0.5,0) {$\bullet$};
\node at (1,0) {$\bullet$};
\node at (1.5,0) {$\bullet$};
\node at (0.25,0.5) {$\bullet$};
\node at (0.75,0.5) {$\bullet$};
\node at (1.25,0.5) {$\bullet$};
\node at (0.5,1) {$\bullet$};
\node at (1,1) {$\bullet$};
\node at (0.75,1.5) {$\bullet$};
%%%%%%%%%%
\draw (0,0) -- (1.5,0);
\draw (0.25,0.5) -- (1.25,0.5);
\draw (0.5,1) -- (1,1);
\draw (0.25,0.5) -- (0.5,0);
\draw (0.5,1) -- (1,0);
\draw (0.75,1.5) -- (1.5,0);
\draw (0,0) -- (0.75,1.5);
\draw (0.5,0) -- (1,1);
\draw (1,0) -- (1.25,0.5);
}
\end{align*}
where the number of vertices along one side of the triangle is equal to $N+1$. Let $\mathfrak{g}^{\triangle}_N(4)$ denote the number of $4$-colourings of $G^{\triangle}_N$ (adjacent vertices must have different colours). We conjecture that
\begin{align*}
4\cdot|\mathcal{T}_N(3)| = \mathfrak{g}^{\triangle}_N(4), \qquad \forall\ N \geq 1.
\end{align*} 
\end{conj}

\begin{conj}[Conjecture \ref{conj:a5} below]
\label{conj:a5-intro}
Let $G^{\varhexstar}_N$ denote the graph
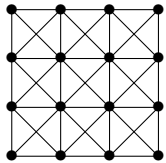
\begin{align*}
\tikz{1.3}{
\node at (0,0) {$\bullet$};
\node at (0.5,0) {$\bullet$};
\node at (1,0) {$\bullet$};
\node at (1.5,0) {$\bullet$};
%%%
\node at (0,0.5) {$\bullet$};
\node at (0.5,0.5) {$\bullet$};
\node at (1,0.5) {$\bullet$};
\node at (1.5,0.5) {$\bullet$};
%%%
\node at (0,1) {$\bullet$};
\node at (0.5,1) {$\bullet$};
\node at (1,1) {$\bullet$};
\node at (1.5,1) {$\bullet$};
%%%
\node at (0,1.5) {$\bullet$};
\node at (0.5,1.5) {$\bullet$};
\node at (1,1.5) {$\bullet$};
\node at (1.5,1.5) {$\bullet$};
%%%%%%%
\draw (0.5,0) -- (0,0.5);
\draw (1,0) -- (0,1);
\draw (1.5,0) -- (0,1.5);
\draw (1.5,0.5) -- (0.5,1.5);
\draw (1.5,1) -- (1,1.5);
\draw (1,0) -- (1.5,0.5);
\draw (0.5,0) -- (1.5,1);
\draw (0,0) -- (1.5,1.5);
\draw (0,0.5) -- (1,1.5);
\draw (0,1) -- (0.5,1.5);
\draw (0,0) -- (0,1.5);
\draw (0.5,0) -- (0.5,1.5);
\draw (1,0) -- (1,1.5);
\draw (1.5,0) -- (1.5,1.5);
\draw (0,0) -- (1.5,0);
\draw (0,0.5) -- (1.5,0.5);
\draw (0,1) -- (1.5,1);
\draw (0,1.5) -- (1.5,1.5);
}
\end{align*}
where two vertices share an edge if they are connected via a king move on the chessboard (that is, they a connected via a unit horizontal, vertical, or diagonal step), and the number of vertices along one side of the square is equal to $N+1$. Let $\mathfrak{g}^{\varhexstar}_N(5)$ denote the number of $5$-colourings of $G^{\varhexstar}_N$ (adjacent vertices must have different colours). We conjecture that
\begin{align*}
5 \cdot |\mathcal{T}_N(4)| = \mathfrak{g}^{\varhexstar}_N(5),
\qquad
\forall\ N \geq 1.
\end{align*}
\end{conj}
Conjecture \ref{conj:triangle-intro} has recently been proved bijectively \cite{GaetzGao}, but Conjecture \ref{conj:a5-intro} remains open. Based on these conjectures, it is tempting to speculate about the possibility of assigning meaningful probability measures to graph colourings, but this lies outside the scope of the current work.

\subsection*{Acknowledgments}
	
Amol Aggarwal was partially supported by a Clay Research Fellowship, a Packard Fellowship, and the IAS School of Mathematics. Alexei Borodin was partially supported by the NSF grants DMS-1664619, DMS-1853981, and the Simons Investigator program. Michael Wheeler was supported by an Australian Research Council Future Fellowship, grant FT200100981.

\section{Fermionic vertex models}

In this section we review the basic vertex models that will be used throughout the text; these are {\it fermionic vertex models}, as introduced in \cite{ABW21}. We give the explicit form of our vertex weights in Sections \ref{ssec:L}--\ref{ssec:M}, as well as the Yang--Baxter equations that they satisfy, in Section \ref{ssec:YB}. We conclude by introducing {\it row operators} and studying algebraic relations between them, in Sections \ref{ssec:row}--\ref{ssec:cr}; these results will be needed in the subsequent material on partition functions in Section \ref{sec:pf}.

\subsection{Notation}
\label{ssec:not}

For all pairs of positive integers $i,j$ such that $i \leq j$ let $[i,j] \subset \mathbb{N}$ denote the interval $\{i,i+1,\dots,j\}$. Similarly, we define $(i,j] = [i+1,j]$ when $i<j$, and $(i,j] = \emptyset$ when $i=j$. For all $1\leq i \leq n$, let $\bm{e}_i \in \mathbb{R}^n$ denote the $i$-th Euclidean unit vector.  Let $\bm{e}_0 \in \mathbb{R}^n$ denote the zero vector. Define $\bm{e}_{[i,j]} = \sum_{i \leq k \leq j} \bm{e}_k$; more generally, for any non-empty set $I \subset \mathbb{N}$ we write $\bm{e}_I = \sum_{i \in I} \bm{e}_i$. For any vector $\A = (A_1,\dots,A_n) \in (\mathbb{Z}_{\geq 0})^n$ and indices $i,j \in \{1,\dots,n\}$ we define 
\begin{align*}
\A^{+}_{i}
=
\A + \bm{e}_i,
\quad
\A^{-}_{i}
=
\A - \bm{e}_i,
\quad
\A^{+-}_{ij}
=
\A + \bm{e}_i - \bm{e}_j,
\quad
\As{i}{j}
=
\sum_{k=i}^{j} A_k,
\quad
|\A|
=
\As{1}{n}
=
\sum_{k=1}^{n} A_k,
\end{align*}
where in the second last case it is assumed that $i \leq j$. By agreement, we choose $\As{i}{j} = 0$ for $i>j$.

Let $\mathfrak{S}_m$ denote the symmetric group of degree $m$. For any set $I \subset \mathbb{N}$ we define $\mathfrak{S}_I$ to be the set of all permutations of the elements in $I$; in particular, we then have $\mathfrak{S}_{[1,m]}\equiv \mathfrak{S}_m$.

\subsection{$L$-weights}
\label{ssec:L}

Our partition functions will be expressed in terms of two families of vertex weights. The first of these were introduced in \cite[Example 8.1.2 and Figure 8.2]{ABW21} and we call them {\it $L$-weights}; they are denoted by\footnote{We use a tilde when writing our weights for consistency with the work of \cite{BorodinW}. In that earlier work, which dealt with models based on $U_q(\widehat{\mathfrak{sl}}(n+1))$ rather than $U_q(\widehat{\mathfrak{sl}}(1|n))$ of the current text, the notation $\tilde{L}^{(s)}_{z,q}(\A,b;\C,d)$ was reserved for vertex weights in the {\it stochastic gauge} (that is, with a sum-to-unity property). While the weights \eqref{fund-weights} no longer satisfy a sum-to-unity property, it is easily seen that they have a completely analogous structure to their tilde analogues in \cite[Chapter 2]{BorodinW}.}
\begin{align}
\label{generic-L}
\tilde{L}^{(s)}_{z,q}(\A,b;\C,d)
\equiv
\tilde{L}_z(\A,b;\C,d)
=
\tikz{0.7}{
\node[left] at (-1.5,0) {$z \rightarrow$};
\draw[lgray,line width=1pt,->] (-1,0) -- (1,0);
\draw[lgray,line width=4pt,->] (0,-1) -- (0,1);
\node[left] at (-1,0) {\tiny $b$};\node[right] at (1,0) {\tiny $d$};
\node[below] at (0,-1) {\tiny $\A$};\node[above] at (0,1) {\tiny $\C$};
\node[below] at (0,-1.4) {$(s)$};
}
\quad\quad
b,d \in \{0,1,\dots,n\},
\quad
\A,\C \in \{0,1\}^n.
\end{align}
Labels assigned to the left and right horizontal edges take values in $\{0,1,\dots,n\}$, while labels assigned to the bottom and top vertical edges are $n$-dimensional binary strings. We define
\begin{align}
\label{conserve-L}
\tilde{L}^{(s)}_{z,q}(\A,b;\C,d)
=
0,
\qquad
\text{unless}
\qquad
\A + \bm{e}_b = \C + \bm{e}_d.
\end{align}
The property \eqref{conserve-L} expresses conservation of particles as one traverses through the vertex in the SW $\rightarrow$ NE direction. For the cases where the constraint $\A + \bm{e}_b = \C + \bm{e}_d$ is obeyed, we have the following table of weights:
\begin{align}
\label{fund-weights}
\begin{tabular}{|c|c|c|}
\hline
\quad
\tikz{0.7}{
\draw[lgray,line width=1pt,->] (-1,0) -- (1,0);
\draw[lgray,line width=4pt,->] (0,-1) -- (0,1);
\node[left] at (-1,0) {\tiny $0$};\node[right] at (1,0) {\tiny $0$};
\node[below] at (0,-1) {\tiny $\A$};\node[above] at (0,1) {\tiny $\A$};
}
\quad
&
\quad
\tikz{0.7}{
\draw[lgray,line width=1pt,->] (-1,0) -- (1,0);
\draw[lgray,line width=4pt,->] (0,-1) -- (0,1);
\node[left] at (-1,0) {\tiny $i$};\node[right] at (1,0) {\tiny $i$};
\node[below] at (0,-1) {\tiny $\A$};\node[above] at (0,1) {\tiny $\A$};
}
\quad
&
\quad
\tikz{0.7}{
\draw[lgray,line width=1pt,->] (-1,0) -- (1,0);
\draw[lgray,line width=4pt,->] (0,-1) -- (0,1);
\node[left] at (-1,0) {\tiny $0$};\node[right] at (1,0) {\tiny $i$};
\node[below] at (0,-1) {\tiny $\A$};\node[above] at (0,1) {\tiny $\A^{-}_i$};
}
\quad
\\[1.3cm]
\quad
$\dfrac{1-q^{\As{1}{n}}sz}{1-s z}$
\quad
& 
\quad
$\dfrac{(-1)^{A_i}(s-q^{A_i}z)q^{\Ap{i}{n}}s}{1-s z}$
\quad
& 
\quad
$\dfrac{(q^{A_i}-1) q^{\Ap{i}{n}}sz}{1-s z}$
\quad
\\[0.7cm]
\hline
\quad
\tikz{0.7}{
\draw[lgray,line width=1pt,->] (-1,0) -- (1,0);
\draw[lgray,line width=4pt,->] (0,-1) -- (0,1);
\node[left] at (-1,0) {\tiny $i$};\node[right] at (1,0) {\tiny $0$};
\node[below] at (0,-1) {\tiny $\A$};\node[above] at (0,1) {\tiny $\A^{+}_i$};
}
\quad
&
\quad
\tikz{0.7}{
\draw[lgray,line width=1pt,->] (-1,0) -- (1,0);
\draw[lgray,line width=4pt,->] (0,-1) -- (0,1);
\node[left] at (-1,0) {\tiny $i$};\node[right] at (1,0) {\tiny $j$};
\node[below] at (0,-1) {\tiny $\A$};\node[above] at (0,1) 
{\tiny $\A^{+-}_{ij}$};
}
\quad
&
\quad
\tikz{0.7}{
\draw[lgray,line width=1pt,->] (-1,0) -- (1,0);
\draw[lgray,line width=4pt,->] (0,-1) -- (0,1);
\node[left] at (-1,0) {\tiny $j$};\node[right] at (1,0) {\tiny $i$};
\node[below] at (0,-1) {\tiny $\A$};\node[above] at (0,1) {\tiny $\A^{+-}_{ji}$};
}
\quad
\\[1.3cm] 
\quad
$\dfrac{1-s^2 q^{\As{1}{n}}}{1-s z}$
\quad
& 
\quad
$\dfrac{(q^{A_j}-1) q^{\Ap{j}{n}}sz}{1-s z}$
\quad
&
\quad
$\dfrac{(q^{A_i}-1) q^{\Ap{i}{n}}s^2}{1-s z}$
\quad
\\[0.7cm]
\hline
\end{tabular} 
\end{align}
where it is assumed that $1 \leq i < j \leq n$.

The weights \eqref{fund-weights} take a very similar form to the weights $\tilde{L}_z(\A,b;\C,d)$ defined in \cite[Sections 2.2 and 2.5]{BorodinW}; in fact, the two sets of weights differ only with respect to two details. The first is that the weights \eqref{generic-L} are defined only for $\A,\C \in \{0,1\}^n$ (that is, for {\it fermionic} states), whereas in \cite[Section 2.2]{BorodinW} one has $\A,\C \in (\mathbb{Z}_{\geq 0})^n$ ({\it bosonic} states). The second is that the specific weight $\tilde{L}_z(\A,i;\A,i)$ is different across the two works\footnote{Indeed, in \cite[Sections 2.2 and 2.5]{BorodinW}, one has $$\tilde{L}_z(\A,i;\A,i)=\dfrac{(sq^{A_i}-z)q^{\Ap{i}{n}}s}{1-s z}$$}, when $i \in [1,n]$ and $A_i > 0$. 

In certain partition functions that we subsequently define, the boundary conditions inject into the lattice exactly one particle of each colour $\{1,\dots,n\}$. In such partition functions, each colour $\{1,\dots,n\}$ flows at most once through a vertex of the lattice; in this setting, both of the differences between the weights \eqref{generic-L} and those of \cite[Sections 2.2 and 2.5]{BorodinW}, pointed out above, are no longer apparent. This fact will allow us to deduce matchings between certain functions that we define in the present work and those of \cite{BorodinW}, in spite of the fact that the model used in the current text is {\it a priori} different.

\subsection{$M$-weights}
\label{ssec:M}

The second family of vertex weights we call {\it $M$-weights}; they are denoted by
\begin{align}
\label{generic-M}
\tilde{M}^{(s)}_{z,q}(\A,b;\C,d)
\equiv
\tilde{M}_z(\A,b;\C,d)
=
\tikz{0.7}{
\node[right] at (1.5,0) {$\leftarrow z$};
\draw[lgray,line width=1.5pt,<-] (-1,0) -- (1,0);
\draw[lgray,line width=4pt,->] (0,-1) -- (0,1);
\node[left] at (-1,0) {\tiny $d$};\node[right] at (1,0) {\tiny $b$};
\node[below] at (0,-1) {\tiny $\A$};\node[above] at (0,1) {\tiny $\C$};
\node[below] at (0,-1.4) {$(s)$};
}
\quad\quad
b,d \in \{0,1,\dots,n\},
\quad
\A,\C \in \{0,1\}^n.
\end{align}
As in the case of $L$-weights, labels assigned to the left and right horizontal edges take values in $\{0,1,\dots,n\}$, while labels assigned to the bottom and top vertical edges are $n$-dimensional binary strings. In contrast to $L$-weights, particle conservation for $M$-type vertices happens in the SE $\rightarrow$ NW direction, namely:
\begin{align}
\label{conserve-M}
\tilde{M}^{(s)}_{z,q}(\A,b;\C,d)
=
0,
\qquad
\text{unless}
\qquad
\A + \bm{e}_b = \C + \bm{e}_d.
\end{align}
For all $\A,\C \in \{0,1\}^n$ and $b,d \in \{0,1,\dots,n\}$, we define 
\begin{align}
\label{LM-sym}
\tilde{M}^{(s)}_{z,q}(\A,b;\C,d)
=
\tilde{L}^{(1/s)}_{1/z,1/q}(\A,b;\C,d),
\end{align}
expressing every $M$-weight in terms of a corresponding $L$-weight, under reflection about the thick vertical line of the vertex, and reciprocation of the parameters $z$, $q$, $s$.

\subsection{Yang--Baxter equations}
\label{ssec:YB}

We introduce one further set of vertex weights which arise from the fundamental $R$-matrix for the quantum affine superalgebra $U_q(\widehat{\mathfrak{sl}}(1|n))$ \cite{BazhanovShadrikov}; these we call {\it fundamental weights}. They are denoted by the crossing of two thin lines:
\begin{align}
\label{R-vert}
R_{z,q}(a,b;c,d)
\equiv
R_z(a,b;c,d)
=
\tikz{0.7}{
\draw[lgray,line width=1.5pt,->] (-1,0) -- (1,0);
\draw[lgray,line width=1.5pt,->] (0,-1) -- (0,1);
\node[left] at (-1,0) {\tiny $b$};\node[right] at (1,0) {\tiny $d$};
\node[below] at (0,-1) {\tiny $a$};\node[above] at (0,1) {\tiny $c$};
},
\quad
a,b,c,d \in \{0,1,\dots,n\}.
\end{align}
These vertices have the conservation property
\begin{align*}
R_{z,q}(a,b;c,d)
=
0,
\qquad
\text{unless}
\qquad
\bm{e}_a + \bm{e}_b = \bm{e}_c + \bm{e}_d.
\end{align*}
For the cases where the constraint $\bm{e}_a + \bm{e}_b = \bm{e}_c + \bm{e}_d$ is obeyed, we have the following table of weights:
\begin{align}
\label{fund-vert}
\begin{tabular}{|c|c|c|}
\hline
\quad
\tikz{0.6}{
	\draw[lgray,line width=1.5pt,->] (-1,0) -- (1,0);
	\draw[lgray,line width=1.5pt,->] (0,-1) -- (0,1);
	\node[left] at (-1,0) {\tiny $0$};\node[right] at (1,0) {\tiny $0$};
	\node[below] at (0,-1) {\tiny $0$};\node[above] at (0,1) {\tiny $0$};
}
\quad
&
\quad
\tikz{0.6}{
	\draw[lgray,line width=1.5pt,->] (-1,0) -- (1,0);
	\draw[lgray,line width=1.5pt,->] (0,-1) -- (0,1);
	\node[left] at (-1,0) {\tiny $a$};\node[right] at (1,0) {\tiny $a$};
	\node[below] at (0,-1) {\tiny $b$};\node[above] at (0,1) {\tiny $b$};
}
\quad
&
\quad
\tikz{0.6}{
	\draw[lgray,line width=1.5pt,->] (-1,0) -- (1,0);
	\draw[lgray,line width=1.5pt,->] (0,-1) -- (0,1);
	\node[left] at (-1,0) {\tiny $a$};\node[right] at (1,0) {\tiny $b$};
	\node[below] at (0,-1) {\tiny $b$};\node[above] at (0,1) {\tiny $a$};
}
\quad
\\[1.3cm]
\quad
$1$
\quad
& 
\quad
$\dfrac{q(1-z)}{1-qz}$
\quad
& 
\quad
$\dfrac{1-q}{1-qz}$
\quad
\\[0.7cm]
\hline
\quad
\tikz{0.6}{
	\draw[lgray,line width=1.5pt,->] (-1,0) -- (1,0);
	\draw[lgray,line width=1.5pt,->] (0,-1) -- (0,1);
	\node[left] at (-1,0) {\tiny $b$};\node[right] at (1,0) {\tiny $b$};
	\node[below] at (0,-1) {\tiny $b$};\node[above] at (0,1) {\tiny $b$};
}
\quad
&
\quad
\tikz{0.6}{
	\draw[lgray,line width=1.5pt,->] (-1,0) -- (1,0);
	\draw[lgray,line width=1.5pt,->] (0,-1) -- (0,1);
	\node[left] at (-1,0) {\tiny $b$};\node[right] at (1,0) {\tiny $b$};
	\node[below] at (0,-1) {\tiny $a$};\node[above] at (0,1) {\tiny $a$};
}
\quad
&
\quad
\tikz{0.6}{
	\draw[lgray,line width=1.5pt,->] (-1,0) -- (1,0);
	\draw[lgray,line width=1.5pt,->] (0,-1) -- (0,1);
	\node[left] at (-1,0) {\tiny $b$};\node[right] at (1,0) {\tiny $a$};
	\node[below] at (0,-1) {\tiny $a$};\node[above] at (0,1) {\tiny $b$};
}
\quad
\\[1.3cm]
\quad
$\dfrac{z-q}{1-qz}$
\quad
& 
\quad
$\dfrac{1-z}{1-qz}$
\quad
&
\quad
$\dfrac{(1-q)z}{1-qz}$
\quad 
\\[0.7cm]
\hline
\end{tabular}
\end{align}
where we assume that $0 \leq a < b \leq n$.

The $L$-weights, $M$-weights and fundamental weights satisfy a collection of Yang--Baxter equations, that we record as a single theorem below. These Yang--Baxter equations underpin the algebraic relations between the row operators that we define in Section \ref{ssec:row}.

\begin{thm}
For any fixed integers $a_1,a_2,a_3,b_1,b_2,b_3 \in \{0,1,\dots,n\}$ and vectors 
$\A,\B \in \{0,1\}^n$, the vertex weights \eqref{generic-L}, \eqref{generic-M}, \eqref{R-vert} satisfy the relations
\begin{multline}
\label{RLLa}
\sum_{0 \leq c_1,c_2 \leq n}
\
\sum_{\C \in \{0,1\}^n}
R_{y/x}(a_2,a_1;c_2,c_1)
\tilde{L}_x(\A,c_1;\C,b_1)
\tilde{L}_y(\C,c_2;\B,b_2)
\\
=
\sum_{0 \leq c_1,c_2 \leq n}
\
\sum_{\C \in \{0,1\}^n}
\tilde{L}_y(\A,a_2;\C,c_2)
\tilde{L}_x(\C,a_1;\B,c_1)
R_{y/x}(c_2,c_1;b_2,b_1),
\end{multline}

\begin{multline}
\label{RLLb}
\sum_{0 \leq c_1,c_3 \leq n}
\
\sum_{\C \in \{0,1\}^n}
\tilde{L}_x(\A,a_1;\C,c_1)
R_{1/(qxz)}(a_3,c_1;c_3,b_1)
\tilde{M}_z(\C,c_3;\B,b_3)
\\
=
\sum_{0 \leq c_1,c_3 \leq n}
\
\sum_{\C \in \{0,1\}^n}
\tilde{M}_z(\A,a_3;\C,c_3)
R_{1/(qxz)}(c_3,a_1;b_3,c_1)
\tilde{L}_x(\C,c_1;\B,b_1),
\end{multline}

\begin{multline}
\label{RLLc}
\sum_{0 \leq c_2,c_3 \leq n}
\
\sum_{\C \in \{0,1\}^n}
\tilde{M}_y(\A,a_2;\C,c_2)
\tilde{M}_z(\C,a_3;\B,c_3)
R_{y/z}(c_3,c_2;b_3,b_2)
\\
=
\sum_{0 \leq c_2,c_3 \leq n}
\
\sum_{\C \in \{0,1\}^n}
R_{y/z}(a_3,a_2;c_3,c_2)
\tilde{M}_z(\A,c_3;\C,b_3)
\tilde{M}_y(\C,c_2;\B,b_2).
\end{multline}
\end{thm}

\begin{proof}
All three equations may be recovered from the master Yang--Baxter equation \eqref{master}; we will comment briefly on this in Section \ref{ssec:master-yb}. The equations \eqref{RLLa}--\eqref{RLLc} are the fermionic cousins of equations (2.3.1)--(2.3.3) in \cite[Section 2.3]{BorodinW}; the latter being valid for the bosonic counterparts of the models \eqref{generic-L} and \eqref{generic-M}.
\end{proof}

\subsection{Row operators}
\label{ssec:row}

Let $V$ be the vector space obtained by taking the formal linear span of all $n$-dimensional binary strings:  
\begin{align*}
V= \bigoplus_{\A \in \{0,1\}^n} \mathbb{C} \ket{\A},
\end{align*}
and for any $N \geq 0$ consider the $(N+1)$-fold tensor product of this space:
\begin{align*}
\mathbb{V}(N) = \underbrace{V \otimes \cdots \otimes V}_{N+1\ {\rm times}}.
\end{align*}
For each $0 \leq i,j \leq n$ we introduce a linear operator 
$T^{\rightarrow}_{i,j}(x;N) \in {\rm End}(\mathbb{V}(N))$ with the action
\begin{align}
\label{C-row}
T^{\rightarrow}_{i,j}(x;N)
:
\bigotimes_{k=0}^{N}
\ket{\B(k)}
\mapsto
\sum_{\A(0),\ldots,\A(N) \in \{0,1\}^n}
\left(
\tikz{1.2}{
\draw[lgray,line width=1pt,->] (0.5,0) -- (5.5,0);
\foreach\x in {1,...,5}{
\draw[lgray,line width=4pt,->] (\x,-0.5) -- (\x,0.5);
}
\node at (-0.2,0) {$x \rightarrow$};
\node[left] at (0.5,0) {\fs $i$};\node[right] at (5.5,0) {\fs $j$};
\node[below] at (5,-0.5) {\fs $\A(N)$};\node[above] at (5,0.5) {\fs $\B(N)$};
\node[below] at (4,-0.5) {\fs $\cdots$};\node[above] at (4,0.5) {\fs $\cdots$};
\node[below] at (3,-0.5) {\fs $\cdots$};\node[above] at (3,0.5) {\fs $\cdots$};
\node[below] at (2,-0.5) {\fs $\cdots$};\node[above] at (2,0.5) {\fs $\cdots$};
\node[below] at (1,-0.5) {\fs $\A(0)$};\node[above] at (1,0.5) {\fs $\B(0)$};
}
\right)
\bigotimes_{k=0}^{N}
\ket{\A(k)}.
\end{align}
The quantity
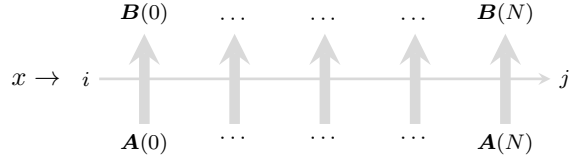
\begin{align*}
\tikz{1.2}{
\draw[lgray,line width=1pt,->] (0.5,0) -- (5.5,0);
\foreach\x in {1,...,5}{
\draw[lgray,line width=4pt,->] (\x,-0.5) -- (\x,0.5);
}
\node at (-0.2,0) {$x \rightarrow$};
\node[left] at (0.5,0) {\fs $i$};\node[right] at (5.5,0) {\fs $j$};
\node[below] at (5,-0.5) {\fs $\A(N)$};\node[above] at (5,0.5) {\fs $\B(N)$};
\node[below] at (4,-0.5) {\fs $\cdots$};\node[above] at (4,0.5) {\fs $\cdots$};
\node[below] at (3,-0.5) {\fs $\cdots$};\node[above] at (3,0.5) {\fs $\cdots$};
\node[below] at (2,-0.5) {\fs $\cdots$};\node[above] at (2,0.5) {\fs $\cdots$};
\node[below] at (1,-0.5) {\fs $\A(0)$};\node[above] at (1,0.5) {\fs $\B(0)$};
}
\end{align*}
is a one-row partition function in the model \eqref{generic-L}, and can be calculated by multiplying the weights of each vertex from left to right, noting that the integer values prescribed to all internal vertical edges are fixed by the local conservation property \eqref{conserve-L}.

In a similar vein, for each $0 \leq i,j \leq n$ we introduce a linear operator 
$T^{\leftarrow}_{i,j}(x;N) \in {\rm End}(\mathbb{V}(N))$ with the action
\begin{align}
\label{B-row}
T^{\leftarrow}_{i,j}(x;N)
:
\bigotimes_{k=0}^{N}
\ket{\B(k)}
\mapsto
\sum_{\A(0),\ldots,\A(N) \in \{0,1\}^n}
\left(
\tikz{1.2}{
\draw[lgray,line width=1pt,<-] (0.5,0) -- (5.5,0);
\foreach\x in {1,...,5}{
\draw[lgray,line width=4pt,->] (\x,-0.5) -- (\x,0.5);
}
\node at (6.2,0) {$\leftarrow x$};
\node[left] at (0.5,0) {\fs $i$};\node[right] at (5.5,0) {\fs $j$};
\node[below] at (5,-0.5) {\fs $\A(N)$};\node[above] at (5,0.5) {\fs $\B(N)$};
\node[below] at (4,-0.5) {\fs $\cdots$};\node[above] at (4,0.5) {\fs $\cdots$};
\node[below] at (3,-0.5) {\fs $\cdots$};\node[above] at (3,0.5) {\fs $\cdots$};
\node[below] at (2,-0.5) {\fs $\cdots$};\node[above] at (2,0.5) {\fs $\cdots$};
\node[below] at (1,-0.5) {\fs $\A(0)$};\node[above] at (1,0.5) {\fs $\B(0)$};
}
\right)
\bigotimes_{k=0}^{N}
\ket{\A(k)},
\end{align}
where the quantity
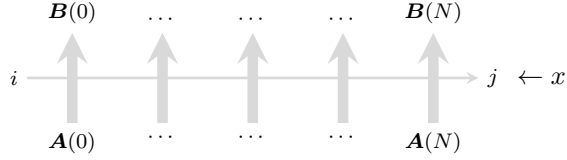
\begin{align*}
\tikz{1.2}{
\draw[lgray,line width=1pt,->] (0.5,0) -- (5.5,0);
\foreach\x in {1,...,5}{
\draw[lgray,line width=4pt,->] (\x,-0.5) -- (\x,0.5);
}
\node at (6.2,0) {$\leftarrow x$};
\node[left] at (0.5,0) {\fs $i$};\node[right] at (5.5,0) {\fs $j$};
\node[below] at (5,-0.5) {\fs $\A(N)$};\node[above] at (5,0.5) {\fs $\B(N)$};
\node[below] at (4,-0.5) {\fs $\cdots$};\node[above] at (4,0.5) {\fs $\cdots$};
\node[below] at (3,-0.5) {\fs $\cdots$};\node[above] at (3,0.5) {\fs $\cdots$};
\node[below] at (2,-0.5) {\fs $\cdots$};\node[above] at (2,0.5) {\fs $\cdots$};
\node[below] at (1,-0.5) {\fs $\A(0)$};\node[above] at (1,0.5) {\fs $\B(0)$};
}
\end{align*}
is a one-row partition function in the model \eqref{generic-M}.

\subsection{Commutation relations}
\label{ssec:cr}

We introduce a lift of $\mathbb{V}(N)$ to an infinite tensor product:
\begin{align*}
\mathbb{V}(\infty)
=
{\rm Span}_{\mathbb{C}}
\left\{
\bigotimes_{k=0}^{\infty}
\ket{\A(k)}
\right\}
\end{align*}
where the binary strings $\A(k) \in \{0,1\}^n$, $k \geq 0$ have the stability property
\begin{align*}
\exists\ M \in \mathbb{N} \ : \ \A(k) = \bm{e}_0,\ \forall\ k \geq M.
\end{align*}
Let $T^{\rightarrow}_{i,0}(x;\infty) = \mathcal{C}_i(x)$ and $T^{\leftarrow}_{i,0}(x;\infty) = \mathcal{B}_i(x)$ denote the corresponding lifts of the operators \eqref{C-row} and \eqref{B-row}, in the case where the right index $j$ is set to $0$. We shall only ever consider the case where $\mathcal{C}_i(x)$ and $\mathcal{B}_i(x)$ act on stable states in the infinite tensor product, \textit{i.e.}, on the elements of $\mathbb{V}(\infty)$.

\begin{thm}
Fix two nonnegative integers $i,j$ such that $1 \leq i<j \leq n$, and two arbitrary complex parameters $x,y$. The following exchange relations hold:
\begin{align}
\label{CC<}
\frac{x-qy}{x-y} \mathcal{C}_i(y) \mathcal{C}_j(x) 
&=
\frac{(1-q)y}{x-y} \mathcal{C}_i(x) \mathcal{C}_j(y)
+
\mathcal{C}_j(x) \mathcal{C}_i(y),
\\
\label{BB<}
\frac{y-qx}{q(y-x)} \mathcal{B}_j(y) \mathcal{B}_i(x) 
&=
\frac{(1-q)x}{q(y-x)} \mathcal{B}_j(x) \mathcal{B}_i(y)
+
\mathcal{B}_i(x) \mathcal{B}_j(y).
\end{align}
\end{thm}

\begin{proof}
The proof of \eqref{CC<} makes use of the first Yang--Baxter equation \eqref{RLLa}, applied successively to the two-row partition function that arises by joining the operators $\mathcal{C}_i(y)$ and $\mathcal{C}_j(x)$; the proof of \eqref{BB<} employs the third Yang--Baxter equation \eqref{RLLc}, applied to the two-row partition function that arises by joining operators $\mathcal{B}_j(y)$ and $\mathcal{B}_i(x)$. For full details, we refer the reader to \cite[Section 3.2, Theorems 3.2.1 and 3.2.5]{BorodinW}.
\end{proof}

\begin{thm}
Fix two nonnegative integers $i,j$ such that $0 \leq i<j \leq n$, and complex parameters $x,y$ such that
\begin{align}
\label{weight-condition}
\left|
\frac{x-s}{1-sx}
\cdot
\frac{y-s}{1-sy}
\right|
<
1.
\end{align}
The row operators $\mathcal{C}_i(x)$ and $\mathcal{B}_j(y)$ obey the following commutation relation:
\begin{align}
\label{CB}
\mathcal{C}_i(x) \mathcal{B}_j(y) &= 
\frac{1-qxy}{1-xy}
\mathcal{B}_j(y)\mathcal{C}_i(x).
\end{align}
\end{thm}

\begin{proof}
The proof makes use of the second Yang--Baxter equation \eqref{RLLb}, applied successively to the two-row partition function that arises by joining the operators $\mathcal{C}_i(x)$ and $\mathcal{B}_j(y)$. For the full details, we refer the reader to \cite[Section 3.2, Theorem 3.2.3]{BorodinW}.
\end{proof}

\section{Partition functions}
\label{sec:pf}

This section brings together a number of partition function definitions, as well as fundamental results related to them, for use throughout the remainder of the text. Most of the facts summarized here were first obtained in \cite[Chapters 3--5 and Chapter 8]{BorodinW}, and where a theorem is directly transcribed from there, we refer the reader to that earlier text for a full proof. We begin by defining {\it coloured compositions} in Section \ref{ssec:cc}; these are used to index many of the quantities that we subsequently define. Sections \ref{ssec:G}--\ref{ssec:perm} introduce the partition functions required; we then state a number of properties of these partition functions in Sections \ref{ssec:hecke}--\ref{ssec:int}.

\subsection{Coloured compositions}
\label{ssec:cc}

\begin{defn}
\label{def:cc}
Let $\lambda = (\lambda_1,\dots,\lambda_n)$ be a composition of length $n$ such that 
$|\lambda| = \sum_{i=1}^{n} \lambda_i = m$; $m$ is called its {\it weight}. We introduce the set $\mathcal{S}_{\lambda}$ of (strict, nonnegative) $\lambda$-coloured compositions as follows:
\begin{align}
\label{lambda-col}
\mathcal{S}_{\lambda}
=
\Big\{ 
\mu 
= 
\Big(
0 \leq \mu^{(1)}_1 < \cdots < \mu^{(1)}_{\lambda_1} \Big|
0 \leq \mu^{(2)}_1 < \cdots < \mu^{(2)}_{\lambda_2} \Big|
\cdots \Big|
0 \leq \mu^{(n)}_1 < \cdots < \mu^{(n)}_{\lambda_n}\Big)
\Big\}.
\end{align}
The elements of $\mathcal{S}_{\lambda}$ are vectors of length $n$ whose $i$-th component $\mu^{(i)}$ is a strict\footnote{That is, with strict inequalities in \eqref{lambda-col}; this corresponds to the fermionicity of our model.}, nonnegative signature of length $\lambda_i$, for all $1 \leq i \leq n$. These components, or blocks, demarcate the colouring of $\mu$; the colour of each block is indicated by the superscript attached to it. We refer to $\lambda$ as the {\it colour profile} of $\mu$.
\end{defn}

\begin{defn}
\label{defn:padding}
With the same assumptions as in Definition \ref{def:cc}, we also define the set $\mathcal{S}^{+}_{\lambda} \subset \mathcal{S}_{\lambda}$ as follows:
\begin{align}
\mathcal{S}^{+}_{\lambda}
=
\{\mu \in \mathcal{S}_{\lambda} : \mu_1^{(j)} \geq 1,\ \forall\ 1 \leq j \leq n \}.
\end{align}
This is the restriction to coloured compositions that have positive parts only. For any coloured composition $\mu \in \mathcal{S}^{+}_{\lambda}$ we define its {\it padding} $0\cup\mu \in \mathcal{S}_{\lambda+1^n}$ by prepending a part of size $0$ in each of the $n$ blocks of $\mu$.
\end{defn}

Let $\mu \in \mathcal{S}_{\lambda}$ be a $\lambda$-coloured composition. We associate to $\mu$ a vector $\ket{\mu}_{\lambda} \in \mathbb{V}(\infty)$, defined as follows:
\begin{align}
\label{A(k)}
\ket{\mu}_{\lambda}
=
\bigotimes_{k=0}^{\infty}
\ket{\bm{A}(k)},
\qquad
\bm{A}(k) = \sum_{j=1}^{n} A_j(k) \bm{e}_j,
\qquad
A_j(k)
=
\left\{
\begin{array}{ll}
1,
&
\quad
k \in \mu^{(j)},
\\ \\
0,
&
\quad
{\rm otherwise}.
\end{array}
\right.
\end{align}
In other words, the component $A_j(k)$ is equal to $1$ if the integer $k$ is present in the strict signature $\mu^{(j)}$, and equal to $0$ if not. We shall also make use of dual vectors 
$\bra{\mu}_{\lambda} \in \mathbb{V}(\infty)^{*}$, defined to act linearly on elements of the form \eqref{A(k)} via the relation $\bra{\mu}_{\lambda} \cdot \ket{\nu}_{\lambda}  = \delta_{\mu,\nu}$ for all 
$\mu,\nu \in \mathcal{S}_{\lambda}$.

\begin{defn}[Rainbow compositions]
The elements of $\mathcal{S}_{1^n}$ are called {\it rainbow compositions}; we have
\begin{align*}
\mathcal{S}_{1^n}
=
\Big\{ 
\mu 
= 
(\mu_1 | \mu_2 | \cdots | \mu_n)
\Big\}.
\end{align*}
That is, a rainbow composition consists of $n$ blocks, each of unit length; no constraint is imposed on the relative ordering of the parts.
\end{defn}

\subsection{Functions $G_{\mu/\nu}$}
\label{ssec:G}

Fix a $\lambda$-coloured composition $\nu \in \mathcal{S}_{\lambda}$ with component signatures $\nu^{(i)}$, $1 \leq i \leq n$, and define, similarly to \eqref{A(k)}, a vector 
$\ket{\nu}_{\lambda} \in \mathbb{V}(\infty)$:
\begin{align}
\label{B(k)}
\ket{\nu}_{\lambda}
=
\bigotimes_{k=0}^{\infty}
\ket{\bm{B}(k)},
\qquad
\bm{B}(k) = \sum_{j=1}^{n} B_j(k) \bm{e}_j,
\qquad
B_j(k)
=
\left\{
\begin{array}{ll}
1,
&
\quad
k \in \nu^{(j)},
\\ \\
0,
&
\quad
{\rm otherwise}.
\end{array}
\right.
\end{align}

\begin{defn}
Let $\lambda = (\lambda_1,\dots,\lambda_n)$ be a composition, and fix two $\lambda$-coloured compositions $\mu \in \mathcal{S}_{\lambda}$ and $\nu \in \mathcal{S}_{\lambda}$. Let the corresponding vectors in $\mathbb{V}(\infty)$, $\ket{\mu}_{\lambda}$ and $\ket{\nu}_{\lambda}$, be given by \eqref{A(k)} and \eqref{B(k)} respectively. For any integer $p \geq 1$ we define the following family of symmetric rational functions:
\begin{align}
\label{G-def}
(-s)^{|\mu|-|\nu|}
\cdot
G_{\mu/\nu}(\lambda;x_1,\dots,x_p)
=
\bra{\nu}_{\lambda}
\prod_{i=1}^{p}
\mathcal{C}_0(x_i)
\ket{\mu}_{\lambda}.
\end{align}
In the case $\lambda = (1,\dots,1) = 1^n$, we drop the notational dependence on $\lambda$, and write
\begin{align*}
G_{\mu/\nu}(1^n;x_1,\dots,x_p)
\equiv
G_{\mu/\nu}(x_1,\dots,x_p).
\end{align*}
The symmetry in $(x_1,\dots,x_p)$ follows from the commutativity of the $\mathcal{C}_0(x_i)$ operators; for a proof of the latter fact, see \cite[Theorem 3.2.1]{BorodinW}.
\end{defn}

Translating the row operators in \eqref{G-def} into their graphical form, we obtain the following partition function representation of $G_{\mu/\nu}$:
\begin{align}
\label{G-pf}
(-s)^{|\mu|-|\nu|}
\cdot
G_{\mu/\nu}(\lambda;x_1,\dots,x_p)
&=
\tikz{0.75}{
\foreach\y in {1,...,5}{
\draw[lgray,line width=1.5pt,->] (1,\y) -- (8,\y);
}
\foreach\x in {2,...,7}{
\draw[lgray,line width=4pt,->] (\x,0) -- (\x,6);
}
%spectral parameters
\node[left] at (0.5,1) {$x_1 \rightarrow$};
\node[left] at (0.5,2) {$x_2 \rightarrow$};
\node[left] at (0.5,3) {$\vdots$};
\node[left] at (0.5,4) {$\vdots$};
\node[left] at (0.5,5) {$x_p \rightarrow$};
%top labels
\node[above] at (7,6) {$\cdots$};
\node[above] at (6,6) {$\cdots$};
\node[above] at (5,6) {$\cdots$};
\node[above] at (4,6) {\footnotesize$\bm{A}(2)$};
\node[above] at (3,6) {\footnotesize$\bm{A}(1)$};
\node[above] at (2,6) {\footnotesize$\bm{A}(0)$};
%bottom labels
\node[below] at (7,0) {$\cdots$};
\node[below] at (6,0) {$\cdots$};
\node[below] at (5,0) {$\cdots$};
\node[below] at (4,0) {\footnotesize$\bm{B}(2)$};
\node[below] at (3,0) {\footnotesize$\bm{B}(1)$};
\node[below] at (2,0) {\footnotesize$\bm{B}(0)$};
%right labels
\node[right] at (8,1) {$0$};
\node[right] at (8,2) {$0$};
\node[right] at (8,3) {$\vdots$};
\node[right] at (8,4) {$\vdots$};
\node[right] at (8,5) {$0$};
%left labels
\node[left] at (1,1) {$0$};
\node[left] at (1,2) {$0$};
\node[left] at (1,3) {$\vdots$};
\node[left] at (1,4) {$\vdots$};
\node[left] at (1,5) {$0$};
}
\end{align}
The factor of $(-s)^{|\mu|-|\nu|}$ incorporated into the definition \eqref{G-def} is due to the fact that, in the partition function representation \eqref{G-pf}, each right edge of the lattice that is occupied by a nonzero colour produces a residual factor of $-s$; this can be seen by analysing the final two columns of the table \eqref{fund-weights}. Since we ultimately intend to take the limit $s \rightarrow 0$, it is important to excise this overall factor from our partition function; we do this by multiplying the left hand side of \eqref{G-def} by $(-s)^{|\mu|-|\nu|}$, since the total number of occupied right edges in the lattice \eqref{G-pf} is equal to $|\mu|-|\nu|$.

\subsection{Functions $f_{\mu}$ and $g_{\mu}$}
\label{ssec:fg}

\begin{defn}
\label{defn:gen-sector}
Let $\lambda = (\lambda_1,\dots,\lambda_n)$ be a composition of weight $m$, and fix a $\lambda$-coloured composition $\mu \in \mathcal{S}_{\lambda}$. Write $\ell_k = \sum_{i=1}^{k} \lambda_i$ for the $k$-th partial sum of $\lambda$. We define the following family of non-symmetric rational functions:
\begin{align}
\label{generic-f}
(-s)^{|\mu|}
\cdot
f_{\mu}(\lambda; x_1,\dots,x_m)
=
\bra{\emptyset}
\prod_{j\in [1,\ell_1]}
\mathcal{C}_1(x_j)
\prod_{j\in (\ell_1,\ell_2]}
\mathcal{C}_2(x_j)
\cdots
\prod_{j\in(\ell_{n-1},\ell_n]}
\mathcal{C}_n(x_j)
\ket{\mu}_{\lambda},
\end{align}
where $\ket{\mu}_{\lambda} \in \mathbb{V}(\infty)$ is given by \eqref{A(k)} and $\bra{\emptyset} \in \mathbb{V}(\infty)^{*}$ denotes the (dual) vacuum state
\begin{align}
\label{dual-vac}
\bra{\emptyset}
=
\bigotimes_{k=0}^{\infty}
\bra{\bm{e}_0},
\end{align}
which is completely devoid of particles.
\end{defn}

Translating the row operators in \eqref{generic-f} into their graphical form, we obtain the following partition function representation of $f_{\mu}$:
\begin{align}
\label{f-def}
(-s)^{|\mu|}
\cdot
f_{\mu}(\lambda;x_1,\dots,x_m)
=
\tikz{0.75}{
\foreach\y in {0,...,5}{
\draw[lgray,line width=1.5pt,->] (1,\y) -- (8,\y);
}
\foreach\x in {2,...,7}{
\draw[lgray,line width=4pt,->] (\x,-1) -- (\x,6);
}
%spectral parameters
\node[left] at (-0.5,0) {$x_1 \rightarrow$};
\node[left] at (-0.5,2) {$\vdots$};
\node[left] at (-0.5,3) {$\vdots$};
\node[left] at (-0.5,5) {$x_m \rightarrow$};
%bottom labels
\node[below] at (7,-1) {$\cdots$};
\node[below] at (6,-1) {$\cdots$};
\node[below] at (5,-1) {$\cdots$};
\node[below] at (4,-1) {\footnotesize$\bm{e}_0$};
\node[below] at (3,-1) {\footnotesize$\bm{e}_0$};
\node[below] at (2,-1) {\footnotesize$\bm{e}_0$};
%top labels
\node[above] at (7,6) {$\cdots$};
\node[above] at (6,6) {$\cdots$};
\node[above] at (5,6) {$\cdots$};
\node[above] at (4,6) {\footnotesize$\bm{A}(2)$};
\node[above] at (3,6) {\footnotesize$\bm{A}(1)$};
\node[above] at (2,6) {\footnotesize$\bm{A}(0)$};
%right labels
\node[right] at (8,0) {$0$};
\node[right] at (8,1) {$0$};
\node[right] at (8,2) {$0$};
\node[right] at (8,3) {$0$};
\node[right] at (8,4) {$0$};
\node[right] at (8,5) {$0$};
%left labels
\node[left] at (1,0) {$1$};
\node[left] at (1.5,0.6) {$\vdots$};
\node[left] at (1,1) {$1$};
\node[left] at (1,2) {$\vdots$};
%\node[left] at (1.5,2.6) {$\vdots$};
\node[left] at (1,3) {$\vdots$};
%\node at (2.5,3.6) {$\vdots$};
%\node at (6.5,3.6) {$\vdots$};
\node[left] at (1,4) {$n$};
\node[left] at (1.5,4.6) {$\vdots$};
\node[left] at (1,5) {$n$};
}
\end{align}
The factor of $(-s)^{|\mu|}$ introduced into the definition \eqref{generic-f} has analogous origins to the factor $(-s)^{|\mu|-|\nu|}$ in \eqref{G-pf}; see the explanation in the paragraph immediately following \eqref{G-pf}.

\begin{rmk}
\label{rainbow-rmk}
In the case $\lambda = (1,\dots,1) = 1^n$, we drop the notational dependence on $\lambda$, and write
\begin{align*}
f_{\mu}(1^n; x_1,\dots,x_n)
\equiv
f_{\mu}(x_1,\dots,x_n).
\end{align*}
The function $f_{\mu}(x_1,\dots,x_n)$ then matches identically with the family of {\it non-symmetric spin Hall--Littlewood functions} defined in \cite[Section 3.4]{BorodinW}; see Definition 3.4.3 therein. The reason for the match is the fact that when $\lambda = 1^n$, each colour $\{1,\dots,n\}$ enters the partition function \eqref{f-def} exactly once, which is precisely the regime when the weights \eqref{fund-weights} and those of \cite[Sections 2.2 and 2.5]{BorodinW} agree (see the discussion below equation \eqref{fund-weights}).
\end{rmk}

\begin{rmk}
\label{factorize-rmk}
In the case $\lambda = (1,\dots,1) = 1^n$, and assuming a {\it weakly increasing}\/ rainbow composition $\mu = (\mu_1 \leq \cdots \leq \mu_n)$, one has the factorization
\begin{align}
\label{factorize}
f_{\mu}(1^n;x_1,\dots,x_n)
=
\frac{\prod_{j \geq 0} (s^2;q)_{\#_j(\mu)}}{\prod_{i=1}^{n}(1-s x_i)}
\prod_{i=1}^{n}
\left( \frac{x_i-s}{1-sx_i} \right)^{\mu_i},
\end{align}
where $\#_j(\mu)$ denotes the number of parts in $\mu$ which are equal to $j$. This is proved by a simple freezing argument applied to the partition function \eqref{f-def}; we refer the reader to \cite[Section 5.1]{BorodinW}.
\end{rmk}

\begin{defn}
\label{defn:generic-sector-dual}
Let $\lambda = (\lambda_1,\dots,\lambda_n)$ be a composition of weight $m$, and fix a $\lambda$-coloured composition $\mu \in \mathcal{S}_{\lambda}$. Write $\ell_k = \sum_{i=1}^{k} \lambda_i$ for the $k$-th partial sum of $\lambda$. Define a further family of non-symmetric rational functions:
\begin{align}
\label{generic-g}
(-s)^{-|\mu|}
\cdot
g_{\mu}(\lambda; x_1,\dots,x_m)
=
\bra{\mu}_{\lambda}
\prod_{j\in [1,\ell_1]}
\mathcal{B}_1(x_j)
\prod_{j\in (\ell_1,\ell_2]}
\mathcal{B}_2(x_j)
\cdots
\prod_{j\in(\ell_{n-1},\ell_n]}
\mathcal{B}_n(x_j)
\ket{\emptyset},
\end{align}
where $\bra{\mu}_{\lambda} \in \mathbb{V}(\infty)^{*}$ is the dual of the vector \eqref{A(k)}, and $\ket{\emptyset} \in \mathbb{V}(\infty)$ denotes the vacuum state
\begin{align}
\label{vac}
\ket{\emptyset}
=
\bigotimes_{k=0}^{\infty}
\ket{\bm{e}_0}.
\end{align}
\end{defn}

Translating the row operators in \eqref{generic-g} into their graphical form, we obtain the following partition function representation of $g_{\mu}$:
\begin{align}
\label{g-def}
(-s)^{-|\mu|}
\cdot
g_{\mu}(\lambda;x_1,\dots,x_m)
=
\tikz{0.75}{
\foreach\y in {0,...,5}{
\draw[lgray,line width=1.5pt,<-] (1,\y) -- (8,\y);
}
\foreach\x in {2,...,7}{
\draw[lgray,line width=4pt,->] (\x,-1) -- (\x,6);
}
%spectral parameters
\node[right] at (8.5,0) {$\leftarrow x_1$};
\node[right] at (8.5,2) {$\vdots$};
\node[right] at (8.5,3) {$\vdots$};
\node[right] at (8.5,5) {$\leftarrow x_m$};
%top labels
\node[above] at (7,6) {$\cdots$};
\node[above] at (6,6) {$\cdots$};
\node[above] at (5,6) {$\cdots$};
\node[above] at (4,6) {\footnotesize$\bm{e}_0$};
\node[above] at (3,6) {\footnotesize$\bm{e}_0$};
\node[above] at (2,6) {\footnotesize$\bm{e}_0$};
%bottom labels
\node[below] at (7,-1) {$\cdots$};
\node[below] at (6,-1) {$\cdots$};
\node[below] at (5,-1) {$\cdots$};
\node[below] at (4,-1) {\footnotesize$\bm{A}(2)$};
\node[below] at (3,-1) {\footnotesize$\bm{A}(1)$};
\node[below] at (2,-1) {\footnotesize$\bm{A}(0)$};
%right labels
\node[right] at (8,0) {$0$};
\node[right] at (8,1) {$0$};
\node[right] at (8,2) {$0$};
\node[right] at (8,3) {$0$};
\node[right] at (8,4) {$0$};
\node[right] at (8,5) {$0$};
%left labels
\node[left] at (1,0) {$1$};
\node[left] at (1.5,0.6) {$\vdots$};
\node[left] at (1,1) {$1$};
\node[left] at (1,2) {$\vdots$};
%\node[left] at (1.5,2.6) {$\vdots$};
\node[left] at (1,3) {$\vdots$};
%\node at (2.5,3.6) {$\vdots$};
%\node at (6.5,3.6) {$\vdots$};
\node[left] at (1,4) {$n$};
\node[left] at (1.5,4.6) {$\vdots$};
\node[left] at (1,5) {$n$};
}
\end{align}

\begin{rmk}
In the case $\lambda = (1,\dots,1) = 1^n$, we drop the notational dependence on $\lambda$, and write
\begin{align*}
g_{\mu}(1^n; x_1,\dots,x_n)
\equiv
g_{\mu}(x_1,\dots,x_n).
\end{align*}
The function $g_{\mu}(x_1,\dots,x_n)$ matches identically with the family of {\it dual non-symmetric spin Hall--Littlewood functions} defined in \cite[Section 3.4]{BorodinW}; see Definition 3.4.6 therein. This match may be deduced by the same reasoning as in Remark \ref{rainbow-rmk}, above. 
\end{rmk}

\subsection{Permuted boundary conditions}
\label{ssec:perm}

\begin{defn}
\label{defn:gen-sector-sigma}
Let $\lambda = (\lambda_1,\dots,\lambda_n)$ be a composition of weight $m$, and fix a $\lambda$-coloured composition $\mu \in \mathcal{S}_{\lambda}$. Fix also a vector $\sigma = (\sigma_1,\dots,\sigma_m)$ such that $|\{k : \sigma_k = i\}| = \lambda_i$ for all $1 \leq i \leq n$. We define the following families of non-symmetric rational functions:
\begin{align}
\label{generic-f-sigma}
(-s)^{|\mu|}
\cdot
f_{\mu}^{\sigma}(\lambda; x_1,\dots,x_m)
&=
\bra{\emptyset}
\prod_{j=1}^{m}
\mathcal{C}_{\sigma_j}(x_j)
\ket{\mu}_{\lambda},
\\
\label{generic-g-sigma}
(-s)^{-|\mu|}
\cdot
g_{\mu}^{\sigma}(\lambda; x_1,\dots,x_m)
&=
\bra{\mu}_{\lambda}
\prod_{j=1}^{m}
\mathcal{B}_{\sigma_j}(x_j)
\ket{\emptyset}.
\end{align}
The first family \eqref{generic-f-sigma} matches with that of Definition \ref{defn:gen-sector}, and the second family \eqref{generic-g-sigma} matches with that of Definition \ref{defn:generic-sector-dual}, when $\sigma = (1^{\lambda_1},2^{\lambda_2},\dots,n^{\lambda_n})$.
\end{defn}

\subsection{Hecke generators and recursion relations}
\label{ssec:hecke}

Recall the definition of the Hecke algebra of type $A_{n-1}$. It is the algebra generated by a family $T_1,\dots,T_{n-1}$, modulo the relations
\begin{align}
\label{hecke1}
(T_i - q)(T_i + 1) = 0,
\quad 
1 \leq i \leq n-1,
\qquad
T_i T_{i+1} T_i = T_{i+1} T_i T_{i+1},
\quad
1 \leq i \leq n-2,
\end{align}
as well as the commutativity property
\begin{align}
\label{hecke2}
[T_i,T_j] = 0,
\quad \forall\ i,j\ \ \text{such that}\ \ |i-j| > 1.
\end{align}
Introduce the simple transpositions $\mathfrak{s}_i$, acting on arbitrary functions $h$ of the alphabet $(x_1,\dots,x_n)$:
\begin{align*}
\mathfrak{s}_i \cdot h(x_1,\dots,x_n)
=
h(x_1,\dots,x_{i+1},x_i,\dots,x_n),
\quad
1 \leq i \leq n-1.
\end{align*}
Making use of these, we define the {\it Demazure--Lusztig operators}
\begin{align}
\label{hecke-poly}
T_i = q - \frac{x_i-q x_{i+1}}{x_i-x_{i+1}} (1-\mathfrak{s}_i),
\quad
1 \leq i \leq n-1,
\end{align}
which provide a faithful representation of the Hecke algebra on the field of rational functions 
$\mathbb{Q}(x_1,\dots,x_n)$.\footnote{Normally, one takes the operators $T_i$ to act on polynomials in the alphabet $(x_1,\dots,x_n)$, since they preserve polynomiality. In this work our partition functions are {\it a priori}\/ rational, which poses no problem, since the action \eqref{hecke-poly} is still faithful on $\mathbb{Q}(x_1,\dots,x_n)$.} From the quadratic identity $(T_i-q)(T_i+1) = 0$, multiplied by $T_i^{-1}$, one gets an explicit formula for inverse Hecke generators:
\begin{align*}
T_i^{-1} = q^{-1}(T_i-q+1) = 
q^{-1}\left(1 - \frac{x_i-q x_{i+1}}{x_i-x_{i+1}} (1-\mathfrak{s}_i) \right),
\quad
1 \leq i \leq n-1.
\end{align*}
In what follows, we will need another version of the Demazure--Lusztig operators \eqref{hecke-poly} in which the variables $(x_i,x_{i+1})$ get reciprocated. We reserve a special notation for this:
\begin{align}
\label{hecke-tilde}
\tilde{T}_i = q - \frac{x_{i+1}-q x_i}{x_{i+1}-x_i} (1-\mathfrak{s}_i),
\quad
\tilde{T}^{-1}_i = q^{-1} \left(1 - \frac{x_{i+1}-q x_i}{x_{i+1}-x_i} (1-\mathfrak{s}_i)\right),
\quad
1 \leq i \leq n-1,
\end{align}
Clearly, the generators $\tilde{T}_i$ also satisfy the basic relations \eqref{hecke1}--\eqref{hecke2} of the Hecke algebra.

\begin{thm}
Fix an integer $1 \leq i \leq n$ and a composition $\mu = (\mu_1|\mu_2|\dots|\mu_n) \in \mathcal{S}_{1^n}$ such that $\mu_i < \mu_{i+1}$. The functions $f_{\mu}(1^n;x_1,\dots,x_n) \equiv f_{\mu}(x_1,\dots,x_n)$ and $g_{\mu}(1^n;x_1,\dots,x_n) \equiv g_{\mu}(x_1,\dots,x_n)$ transform under the action of \eqref{hecke-poly}--\eqref{hecke-tilde} in the following way:
\begin{align}
\label{T-f}
T_i \cdot f_{\mu}(x_1,\dots,x_n) &= f_{\mathfrak{s}_i \cdot \mu}(x_1,\dots,x_n),
\\
\label{T-g}
\tilde{T}_i \cdot g_{\mu}(x_1,\dots,x_n) &= q \cdot g_{\mathfrak{s}_i \cdot \mu}(x_1,\dots,x_n),
\end{align}
where $\mathfrak{s}_i \cdot \mu$ denotes the composition obtained by switching $\mu_i $ and $\mu_{i+1}$.
\end{thm}

\begin{proof}
Both statements \eqref{T-f} and \eqref{T-g} are proved in \cite{BorodinW}; see equations (5.3.1) and (8.2.24) therein, respectively.
\end{proof}

\begin{thm}
Fix a coloured composition $\nu \in \mathcal{S}_{\lambda}$, where $\lambda = (\lambda_1,\dots,\lambda_n)$ is a composition such that $|\lambda| = m$, as well as a vector $\sigma = (\sigma_1,\dots,\sigma_m)$ such that $|\{k:\sigma_k = i\}| = \lambda_i$ for all $1 \leq i \leq n$. Assuming that $\sigma_j < \sigma_{j+1}$ for some $1 \leq j \leq m-1$, there holds
\begin{align}
\label{invT-f}
T_j \cdot f_{\nu}^{\mathfrak{s}_j \cdot \sigma}(\lambda;x_1,\dots,x_m)
&=
q \cdot f_{\nu}^{\sigma}(\lambda;x_1,\dots,x_m),
\\
\label{invT-g}
\tilde{T}_j \cdot
g_{\nu}^{\sigma}(\lambda;x_1,\dots,x_m)
&=
g_{\nu}^{\mathfrak{s}_j \cdot \sigma}(\lambda;x_1,\dots,x_m),
\end{align}
where $\mathfrak{s}_j \cdot \sigma$ denotes the vector obtained by switching $\sigma_j$ and $\sigma_{j+1}$.
\end{thm}

\begin{proof}
The proof of \eqref{invT-f} is by isolating the action of $T_j^{-1}$ on the pair of operators 
$\mathcal{C}_{\sigma_j}(x_j) \mathcal{C}_{\sigma_{j+1}}(x_{j+1})$, which is the only place that $f_{\nu}^{\sigma}(\lambda;x_1,\dots,x_m)$ depends on $(x_j,x_{j+1})$. Using the explicit form of $T_j^{-1}$, we have
\begin{align}
\label{invT-pf1}
q\cdot T_j^{-1}
\cdot
\mathcal{C}_{\sigma_j}(x_j) \mathcal{C}_{\sigma_{j+1}}(x_{j+1})
=
\frac{(q-1) x_{j+1}}{x_j-x_{j+1}}
\mathcal{C}_{\sigma_j}(x_j) \mathcal{C}_{\sigma_{j+1}}(x_{j+1})
+
\frac{x_j-qx_{j+1}}{x_j-x_{j+1}}
\mathcal{C}_{\sigma_j}(x_{j+1}) \mathcal{C}_{\sigma_{j+1}}(x_j).
\end{align}
In view of the fact that $\sigma_j < \sigma_{j+1}$, we may use the commutation relation \eqref{CC<} to combine the right hand side of \eqref{invT-pf1} into a single term:
\begin{align*}
q\cdot T_j^{-1}
\cdot
\mathcal{C}_{\sigma_j}(x_j) \mathcal{C}_{\sigma_{j+1}}(x_{j+1})
=
\mathcal{C}_{\sigma_{j+1}}(x_j) \mathcal{C}_{\sigma_j}(x_{j+1}).
\end{align*}
Substitution of this identity into \eqref{generic-f-sigma} immediately proves \eqref{invT-f}.

In a similar vein, one proves \eqref{invT-g} by isolating the action of $\tilde{T}_j^{-1}$ on the pair $\mathcal{B}_{\sigma_{j+1}}(x_j) \mathcal{B}_{\sigma_j}(x_{j+1})$, which is the only place that $g_{\nu}^{\mathfrak{s}_j \cdot \sigma}(\lambda;x_1,\dots,x_m)$ depends on $(x_j,x_{j+1})$. Using the explicit form of $\tilde{T}_j^{-1}$, we have
\begin{align}
\label{invT-pf2}
\tilde{T}_j^{-1}
\cdot
\mathcal{B}_{\sigma_{j+1}}(x_j) \mathcal{B}_{\sigma_j}(x_{j+1})
=
\frac{(q-1) x_j}{q(x_{j+1}-x_j)}
\mathcal{B}_{\sigma_{j+1}}(x_j) \mathcal{B}_{\sigma_j}(x_{j+1})
+
\frac{x_{j+1}-qx_j}{q(x_{j+1}-x_j)}
\mathcal{B}_{\sigma_{j+1}}(x_{j+1}) \mathcal{B}_{\sigma_j}(x_j).
\end{align}
Since $\sigma_j < \sigma_{j+1}$, we use the commutation relation \eqref{BB<} to combine the right hand side of \eqref{invT-pf2} into a single term:
\begin{align*}
\tilde{T}_j^{-1}
\cdot
\mathcal{B}_{\sigma_{j+1}}(x_j) \mathcal{B}_{\sigma_j}(x_{j+1})
=
\mathcal{B}_{\sigma_j}(x_j) \mathcal{B}_{\sigma_{j+1}}(x_{j+1}).
\end{align*}
Substitution of this identity into \eqref{generic-g-sigma} proves \eqref{invT-g}.

\end{proof}

\subsection{Antisymmetrization}
\label{ssec:anti}

A key property of the vertex models \eqref{generic-L} and \eqref{generic-M} is that of {\it colour-merging}; this is the combinatorial statement that partition functions in the models \eqref{generic-L} and \eqref{generic-M}, with $n$ colours, become equal to partition functions with $m < n$ colours under a certain antisymmetrization procedure applied to the boundary conditions. The most general colour-merging statement is given and proved as \cite[Theorem 5.2.2]{ABW21}; here we will reproduce this statement only at the level that we need, namely, for two of the families of rational functions that we have defined. 

To state our antisymmetrization results, we require some definitions. 

\begin{defn}[Rainbow recolouring]
\label{def:rainbow-rec}
Let  $\lambda = (\lambda_1,\dots,\lambda_n)$ be a composition such that $|\lambda|=m$, and fix a coloured composition $\mu \in \mathcal{S}_{\lambda}$. Denoting 
\begin{align*}
\mu
=
\Big(\mu^{(1)}_1 < \cdots < \mu^{(1)}_{\lambda_1} \Big|
\mu^{(2)}_1 < \cdots < \mu^{(2)}_{\lambda_2} \Big|
\cdots \Big|
\mu^{(n)}_1 < \cdots < \mu^{(n)}_{\lambda_n}\Big),
\end{align*}
we associate to this a rainbow composition $\breve\mu = \left( \breve{\mu}_1| \breve{\mu}_2 | \cdots | \breve{\mu}_m \right) \in \mathcal{S}_{1^m}$ such that for each $1 \leq i \leq m$ we have
\begin{align*}
\breve{\mu}_i = \mu^{(k)}_j,
\end{align*} 
where $1 \leq k \leq n$, $1 \leq j \leq \lambda_k$ are the unique integers such that
\begin{align*}
i = j+\sum_{a=1}^{k-1} \lambda_a.
\end{align*}
In simpler terms, $\breve{\mu}$ is the composition obtained from recolouring the parts of $\mu$ sequentially from $1$ to $m$ into pairwise distinct colours, while keeping the magnitude of all parts fixed.
\end{defn}

\begin{defn}
Fix a positive integer $m$ and let $\lambda = (\lambda_1,\dots,\lambda_n)$ be a composition such that $|\lambda| = m$, with partial sums $\ell_k = \sum_{i=1}^{k} \lambda_i$. We say that $\sigma \in \mathfrak{S}_{\lambda} \subset \mathfrak{S}_m$ provided that $\sigma$ fixes $(\ell_{k-1},\ell_k]$ for each integer $1 \leq k \leq n$, that is,
\begin{align*}
\sigma \in \mathfrak{S}_{[1,\ell_1]} \times \mathfrak{S}_{(\ell_1,\ell_2]} \times \cdots \times \mathfrak{S}_{(\ell_{n-1},\ell_n]}.
\end{align*}
\end{defn}

\begin{prop}
\label{prop:f-ant}
Fix a coloured composition $\nu \in \mathcal{S}_{\lambda}$ and let $\breve{\nu}$ denote its rainbow recolouring, as in Definition \ref{def:rainbow-rec}. We then have the following result, relating the functions \eqref{generic-g} for rainbow colour profiles with those of non-rainbow type:
\begin{align*}
\sum_{\sigma \in \mathfrak{S}_{\lambda}}
(-1)^{{\rm inv}(\sigma)}
g_{\sigma(\breve{\nu})}(1^m;x_1,\dots,x_m)
=
g_{\nu}(\lambda;x_1,\dots,x_m),
\end{align*}
where the sum is taken over all elements in $\mathfrak{S}_{\lambda}$. Here we have defined ${\rm inv}(\sigma) = {\rm card}\{ (i,j) : i<j,\ \sigma_i > \sigma_j\}$ and 
$\sigma(\breve{\nu}) = 
\left(\breve{\nu}_{\sigma(1)} | \breve{\nu}_{\sigma(2)}| \cdots | \breve{\nu}_{\sigma(m)}\right)$.
\end{prop}

\begin{prop}
\label{prop:G-ant}
Fix two coloured compositions $\mu,\nu \in \mathcal{S}_{\lambda}$ and let $\breve{\mu},\breve{\nu}$ denote their respective rainbow recolouring, as in Definition \ref{def:rainbow-rec}. The functions \eqref{G-def} have the following sum property:
\begin{align*}
\sum_{\sigma \in \mathfrak{S}_{\lambda}}
(-1)^{{\rm inv}(\sigma)}
G_{\breve{\mu}/\sigma(\breve{\nu})}(1^m;x_1,\dots,x_p)
=
G_{\mu/\nu}(\lambda;x_1,\dots,x_p)
\end{align*}
where the sum is taken over all elements in $\mathfrak{S}_{\lambda}$.
\end{prop}
A similar antisymmetrization result can be stated for the functions \eqref{generic-f}, but we omit it from this section since we shall not require it in what follows.

\begin{rmk}
Propositions \ref{prop:f-ant} and \ref{prop:G-ant} are both statements about partition functions constructed from $M$-weights, as defined in Section \ref{ssec:M}. In order to recover them as corollaries of \cite[Theorem 5.2.2]{ABW21} one should first apply the symmetry \eqref{LM-sym}, which converts them to statements about partition functions built from $L$-weights, and the matching with \cite{ABW21} then goes through in a straightforward way.
\end{rmk}

\subsection{Orthogonality}
\label{ssec:orthog}

In this section we directly transcribe an orthogonality result for non-symmetric spin Hall--Littlewood functions, from \cite[Chapter 8]{BorodinW}. Throughout, we denote the imaginary unit by ${\tt i} = \sqrt{-1}$. Let $\{C_1,\dots,C_n\}$ be a collection of contours in the complex plane, and fix two complex parameters $q,s \in \mathbb{C}$. We say that the set $\{C_1,\dots,C_n\}$ is admissible with respect to $(q,s)$ if the following conditions are met:
\begin{itemize}
\item The contours $\{C_1,\dots,C_n\}$ are closed, positively oriented and pairwise non-intersecting;
\item The contours $C_i$ and $q \cdot C_i$ are both contained within contour $C_{i+1}$ for all $1 \leq i \leq n-1$, where $q \cdot C_i$ denotes the image of $C_i$ under multiplication by $q$;
\item All contours surround the point $s$.
\end{itemize}

\begin{thm}
Fix two rainbow compositions $\mu,\nu \in \mathcal{S}_{1^n}$, and let $\{C_1,\dots,C_n\}$ be contours admissible with respect to $(q,s)$. We then have
\begin{align}
\label{f-g-orthog}
\left( \frac{1}{2\pi{\tt i}} \right)^n
\oint_{C_1}
\frac{dy_1}{y_1}
\cdots 
\oint_{C_n}
\frac{dy_n}{y_n}
\prod_{1 \leq i<j \leq n}
\left(
\frac{y_j-y_i}{y_j-q y_i}
\right)
f_{\mu}(y_1^{-1},\dots,y_n^{-1})
g_{\nu}(y_1,\dots,y_n)
=
\frac{{\bm 1}_{\mu = \nu}\cdot(q-1)^n}{q^{n(n+1)/2}}.
\end{align}
\end{thm}

\begin{proof}
This is Theorem 8.2.1 of \cite[Chapter 8]{BorodinW}.
\end{proof}

Closely related to the orthogonality statement \eqref{f-g-orthog}, and in fact instrumental in its proof, is the following property of the Hecke generators \eqref{hecke-poly}, \eqref{hecke-tilde} with respect to such integrals:

\begin{prop}
Fix an integer $1 \leq k \leq n-1$, and three functions $a(y_1,\dots,y_n)$, $b(y_1,\dots,y_n)$ and $c(y_1,\dots,y_n)$, the last of which is symmetric in its alphabet $(y_1,\dots,y_n)$. We have the following equality of integrals:
\begin{multline}
\label{adjoint}
\oint_{C_1}
\frac{dy_1}{y_1}
\cdots 
\oint_{C_n}
\frac{dy_n}{y_n}
\prod_{1 \leq i<j \leq n}
\left(
\frac{y_j-y_i}{y_j-q y_i}
\right)
(T_k \cdot a)(y_1^{-1},\dots,y_n^{-1})
b(y_1,\dots,y_n)
c(y_1,\dots,y_n)
\\
=
\oint_{C_1}
\frac{dy_1}{y_1}
\cdots 
\oint_{C_n}
\frac{dy_n}{y_n}
\prod_{1 \leq i<j \leq n}
\left(
\frac{y_j-y_i}{y_j-q y_i}
\right)
a(y_1^{-1},\dots,y_n^{-1})
(\tilde{T}_k \cdot b)(y_1,\dots,y_n)
c(y_1,\dots,y_n),
\end{multline}
with $T_k$, $\tilde{T}_k$ given by \eqref{hecke-poly}, \eqref{hecke-tilde}, respectively.
\end{prop}

\begin{proof}
The proof of this result, for $c(y_1,\dots,y_n) \equiv 1$, is given in Proposition 8.1.3 in \cite[Chapter 8]{BorodinW}. The extension of the result to generic symmetric functions $c(y_1,\dots,y_n)$ follows immediately, in view of the fact that acting with Hecke generators $T_k$, $\tilde{T}_k$ commutes with multiplication by functions which are symmetric in $(y_k,y_{k+1})$.
\end{proof}

\subsection{Cauchy identity}
\label{ssec:cauchy}

It is possible to derive a number of summation identities of Cauchy-type for the non-symmetric spin Hall--Littlewood functions; see \cite[Chapter 4]{BorodinW}. In this section we state a Cauchy identity that did not previously appear in that text, although it is similar in flavour to \cite[Proposition 4.5.1]{BorodinW}, and proved in precisely the same fashion.

\begin{thm}
Let $\lambda = (\lambda_1,\dots,\lambda_n)$ be a composition such that $|\lambda| = m$. Fix a coloured composition $\nu \in \mathcal{S}_{\lambda}$ and two alphabets 
$(x_1,\dots,x_p)$, $(y_1,\dots,y_m)$ of complex parameters satisfying the constraint
\begin{align}
\label{converge}
\left| 
\frac{x_i-s}{1-s x_i}
\cdot
\frac{y_j-s}{1-s y_j}
\right|
<
1,
\qquad
\forall\ 1 \leq i \leq p,\ 1 \leq j \leq m.
\end{align}
The following summation identity holds:
\begin{align}
\label{Gg-cauchy}
\sum_{\kappa \in \mathcal{S}_{\lambda}}
G_{\kappa/\nu}(\lambda;x_1,\dots,x_p)
g_{\kappa}(\lambda;y_1,\dots,y_m)
=
\prod_{i=1}^{p}
\prod_{j=1}^{m}
\frac{1-q x_i y_j}{1-x_i y_j}
\cdot
g_{\nu}(\lambda;y_1,\dots,y_m).
\end{align}

\end{thm}

\begin{proof}
The left hand side of \eqref{Gg-cauchy} may be represented algebraically as
\begin{multline*}
\sum_{\kappa \in \mathcal{S}_{\lambda}}
G_{\kappa/\nu}(\lambda;x_1,\dots,x_p)
g_{\kappa}(\lambda;y_1,\dots,y_m)
\\
=
\bra{\nu}_{\lambda}
\prod_{i=1}^{p}
\mathcal{C}_0(x_i)
\prod_{j\in [1,\ell_1]}
\mathcal{B}_1(x_j)
\prod_{j\in (\ell_1,\ell_2]}
\mathcal{B}_2(x_j)
\cdots
\prod_{j\in(\ell_{n-1},\ell_n]}
\mathcal{B}_n(x_j)
\ket{\emptyset}.
\end{multline*}
We use the commutation relation \eqref{CB} in the case $i=0$, $j \geq 1$ to transfer all $\mathcal{B}$-operators to the left of the product; this results in the equation
\begin{multline}
\label{gG-proof1}
\sum_{\kappa \in \mathcal{S}_{\lambda}}
G_{\kappa/\nu}(\lambda;x_1,\dots,x_p)
g_{\kappa}(\lambda;y_1,\dots,y_m)
\\
=
\prod_{i=1}^{p}
\prod_{j=1}^{m}
\frac{1-q x_i y_j}{1-x_i y_j}
\cdot
\bra{\nu}_{\lambda}
\prod_{j\in [1,\ell_1]}
\mathcal{B}_1(x_j)
\prod_{j\in (\ell_1,\ell_2]}
\mathcal{B}_2(x_j)
\cdots
\prod_{j\in(\ell_{n-1},\ell_n]}
\mathcal{B}_n(x_j)
\ket{\emptyset},
\end{multline}
where we have used the fact that
\begin{align*}
\prod_{i=1}^{p}
\mathcal{C}_0(x_i)
\ket{\emptyset}
=
\ket{\emptyset}.
\end{align*}
The expression obtained, \eqref{gG-proof1}, matches with the right hand side of 
\eqref{Gg-cauchy}.

\end{proof}

\subsection{Integral formula for $G_{\mu/\nu}$}
\label{ssec:int}

Combining the results of Sections \ref{ssec:orthog}--\ref{ssec:cauchy}, we now obtain an integral formula\footnote{A more general version of this integral formula appears in \cite[Proposition 11.3.1]{ABW21}.} for the rational symmetric functions \eqref{G-pf}: 
\begin{thm}
We have the following integral formula for the function $G_{\mu/\nu}(\lambda;x_1,\dots,x_p)$:
\begin{multline}
\label{G-int}
G_{\mu/\nu}(\lambda;x_1,\dots,x_p)
=
\frac{q^{m(m+1)/2}}{(q-1)^m}
\cdot
\left( \frac{1}{2\pi{\tt i}} \right)^m
\oint_{C_1}
\frac{dy_1}{y_1}
\cdots 
\oint_{C_m}
\frac{dy_m}{y_m}
\\
\times
\prod_{1 \leq i<j \leq m}
\left(
\frac{y_j-y_i}{y_j-q y_i}
\right)
f_{\breve\mu}(1^m;y_1^{-1},\dots,y_m^{-1})
g_{\nu}(\lambda;y_1,\dots,y_m)
\prod_{i=1}^{p}
\prod_{j=1}^{m}
\frac{1-q x_i y_j}{1-x_i y_j},
\end{multline}
where $\{C_1,\dots,C_m\}$ are contours admissible with respect to $(q,s)$.
\end{thm}

\begin{proof}
This result is essentially given by \cite[Corollary 11.3.2]{ABW21}, though we reproduce its proof here for the reader's convenience. 

We begin by proving \eqref{G-int} in the case where $\mu,\nu$ are rainbow compositions. Start from the Cauchy identity \eqref{Gg-cauchy} with $\lambda = 1^m$, multiply it by $f_{\mu}(y_1^{-1},\dots,y_m^{-1})$, prior to integrating as in the left hand side of \eqref{f-g-orthog}.\footnote{The summation convergence in \eqref{Gg-cauchy} is uniform on compact sets as long as the inequalities \eqref{converge} are satisfied.} In view of the orthogonality property \eqref{f-g-orthog}, this filters the $\kappa = \mu$ term from the sum and we read off the identity
\begin{multline}
\label{G-int-rainbow}
G_{\mu/\nu}(1^m;x_1,\dots,x_p)
=
\frac{q^{m(m+1)/2}}{(q-1)^m}
\cdot
\left( \frac{1}{2\pi{\tt i}} \right)^m
\oint_{C_1}
\frac{dy_1}{y_1}
\cdots 
\oint_{C_m}
\frac{dy_m}{y_m}
\\
\times
\prod_{1 \leq i<j \leq m}
\left(
\frac{y_j-y_i}{y_j-q y_i}
\right)
f_{\mu}(1^m;y_1^{-1},\dots,y_m^{-1})
g_{\nu}(1^m;y_1,\dots,y_m)
\prod_{i=1}^{p}
\prod_{j=1}^{m}
\frac{1-q x_i y_j}{1-x_i y_j}.
\end{multline}
This proves \eqref{G-int} in the case $\mu,\nu \in \mathcal{S}_{1^m}$.

The general case \eqref{G-int} then follows by antisymmetrization of \eqref{G-int-rainbow}; the left hand side antisymmetrization is obtained using Proposition \ref{prop:f-ant}, while that of the right hand side is carried out using Proposition \ref{prop:G-ant}. 
\end{proof}

\section{Fusion}

In this section we briefly recall some of the basics regarding the fusion procedure, when applied to the model \eqref{generic-L}. For full details, we refer the reader to \cite[Chapter 3]{ABW21} and \cite[Appendices B and C]{BorodinW}.

\subsection{Definition of fused vertices}

To define fused vertices we require some additional notation; introduce {\it column vertices} by taking towers of height $N$ of the $L$-weights \eqref{generic-L}. In particular, for all $\A, \C \in \{0,1\}^n$ and $b_1,\dots,b_N,d_1,\dots,d_N \in [0,n]$ we define
\begin{align}
\label{tower}
\tilde{L}^{(s)}_z
\Big(
\A,(b_1,\dots,b_N);\C,(d_1,\dots,d_N)
\Big)
=
\tikz{0.9}{
\draw[lgray,line width=1.5pt,->] (-1,5) -- (1,5);
%\draw[lgray,line width=1.5pt,->] (-1,4) -- (1,4);
\draw[lgray,line width=1.5pt,->] (-1,3) -- (1,3);
%\draw[lgray,line width=1.5pt,->] (-1,2) -- (1,2);
\draw[lgray,line width=1.5pt,->] (-1,1) -- (1,1);
\draw[lgray,line width=1.5pt,->] (-1,0) -- (1,0);
\draw[lgray,line width=4pt,->] (0,-1) -- (0,6);
\node[left] at (-1,0) {\tiny $b_1$};\node[right] at (1,0) {\tiny $d_1$};
\node[left] at (-1,1) {\tiny $b_2$};\node[right] at (1,1) {\tiny $d_2$};
\node[left] at (-1,2.5) {$\vdots$};\node[right] at (1,2.5) {$\vdots$};
\node[left] at (-1,3.5) {$\vdots$};\node[right] at (1,3.5) {$\vdots$};
\node[left] at (-1,5) {\tiny $b_N$};\node[right] at (1,5) {\tiny $d_N$};
\node[below] at (0,-1) {\tiny $\A$};\node[above] at (0,6) {\tiny $\C$};
\node[left] at (-2,0) {$z \rightarrow$};
\node[left] at (-2,1) {$q z \rightarrow$};
\node[left] at (-2,5) {$q^{N-1} z \rightarrow$};
}
\end{align}
where the spectral parameters associated to horizontal lines, read from bottom to top, form the geometric progression $(z,qz,\dots,q^{N-1}z)$.

\begin{defn}
Fix four binary strings $\A = (A_1,\dots,A_n)$, $\B = (B_1,\dots,B_n)$, $\C = (C_1,\dots,C_n)$ and $\D = (D_1,\dots,D_n)$ in $\{0,1\}^n$. Choose an integer $N \geq 1$ and introduce the notation $r=q^{-N/2}$. We define {\it fused vertex weights} as follows:
\begin{align}
\label{fused-def}
\tilde{L}^{(r,s)}_z(\A,\B; \C,\D)
=
\frac{1}{Z_q(N;\B)}
\sum_{\substack{
(b_1,\dots,b_N)
\\
(d_1,\dots,d_N)
}}
q^{{\rm inv}(b_1,\dots,b_N)}
\tilde{L}^{(s)}_z
\Big(
\A,(b_1,\dots,b_N);\C,(d_1,\dots,d_N)
\Big),
\end{align}
where the sum is taken over vectors $(b_1,\dots,b_N)$ and $(d_1,\dots,d_N)$ such that 
$\sum_{i=1}^{N} \bm{e}_{b_i} = \B$ and $\sum_{i=1}^{N} \bm{e}_{d_i} = \D$, we recall that 
${\rm inv}(b_1,\dots,b_N) = {\rm card}\{ (i,j) : i<j,\ b_i > b_j\}$, and where the normalization takes the form
\begin{align*}
Z_q(N;\B)
=
\frac{(q;q)_N}{(q;q)_{B_0} (q;q)_{B_1} \dots (q;q)_{B_n}},
\quad
B_0 = N - \sum_{i=1}^{n} B_i.
\end{align*}
We represent the fused vertices \eqref{fused-def} graphically as follows:
\begin{align*}
\tilde{L}^{(r,s)}_z(\A,\B; \C,\D)
=
\tikz{0.7}{
\node[left] at (-1.5,0) {$(z;r) \rightarrow $};
\draw[lgray,line width=4pt,->] (-1,0) -- (1,0);
\draw[lgray,line width=4pt,->] (0,-1) -- (0,1);
\node[left] at (-1,0) {\tiny $\B$};\node[right] at (1,0) {\tiny $\D$};
\node[below] at (0,-1) {\tiny $\A$};\node[above] at (0,1) {\tiny $\C$};
%\node[below] at (0,-1.4) {$\uparrow$};
\node[below] at (0,-1.4) {$(s)$};
}
\qquad
\A,\B,\C,\D \in \{0,1\}^n.
\end{align*}
\end{defn}

\begin{prop}
For all integers $b,d \in [0,n]$ and binary strings $\A,\C \in \{0,1\}^n$, one has
\begin{align*}
\tilde{L}^{(r,s)}_z(\A,\bm{e}_b; \C,\bm{e}_d)\Big|_{r=q^{-1/2}}
=
\tilde{L}^{(s)}_z(\A,b;\C,d).
\end{align*}
\end{prop}

\begin{proof}
Setting $r=q^{-1/2}$ is equivalent to taking $N=1$. For $N=1$ the sum on the right hand side of \eqref{fused-def} trivializes and $Z_q(1;\bm{e}_b)=1$ for all $b \in [0,n]$; the claimed equality is then manifest.
\end{proof}

\subsection{Fused vertex weights}

The fused vertex weights \eqref{fused-def} were explicitly evaluated as \cite[Theorem 4.3.2]{ABW21}; we recall this explicit formula here. 

For any pair of vectors $\A,\B \in \mathbb{Z}^n$, define the function 
\begin{align}
\label{phi-def}
\varphi(\A,\B) = \sum_{1 \leq i<j \leq n} A_i B_j.
\end{align}
Fix binary strings $\A = (A_1,\dots,A_n)$, $\B = (B_1,\dots,B_n)$, $\C = (C_1,\dots,C_n)$, $\D = (D_1,\dots,D_n)$. Construct another vector $\V = (V_1,\dots,V_n)$, where $V_i = {\rm min}\{A_i,B_i,C_i,D_i\}$ for $i \in [1,n]$. The fused weights \eqref{fused-def} are then given by\footnote{We state this formula for spectral parameter $sz$, rather than $z$, in order to match with \cite[Theorem 4.3.2]{ABW21}.}
\begin{multline}
\label{fused-weights}
\tilde{L}^{(r,s)}_{sz}(\A,\B; \C,\D)
= 
\bm{1}_{\A+\B=\C+\D}
\\
\times
(-1)^{|\V|}
z^{|\D|-|\B|}
r^{-2|\A|}
s^{2|\D|}
q^{-\varphi(\A,\V)-|\V|}
\frac{(r^{-2}q^{-|\V|+1}z;q)_{|\V|}}{(s^2 r^{-2} q^{-|\V|}z;q)_{|\V|}}
W^{(r,s)}_z(\A,\B-\V;\C,\D-\V)
\end{multline}
where we have defined
\begin{multline}
\label{w-weight}
W^{(r,s)}_z(\A,\B-\V;\C,\D-\V)
=
\\
\sum_{\P}
\Phi(\C-\P,\C+\D-\V-\P;s^2 r^{-2} q^{-|\V|} z,s^2 z)
\Phi(\P,\B-\V;r^2q^{|\V|}z^{-1},r^2q^{|\V|}),
\end{multline}
with the sum over all $\P = (P_1,\dots,P_n)$ such that $P_i \leq {\rm min}\{C_i,B_i-V_i\}$ for all $i \in [1,n]$. The functions appearing in the summand of \eqref{w-weight} are defined for any two vectors $\bm{S} = (S_1,\dots,S_n)$, $\bm{T} = (T_1,\dots,T_n)$ such that $S_i \leq T_i$ for all $i \in [1,n]$:
\begin{align*}
\Phi(\bm{S},\bm{T};u,v)
=
\frac{(u;q)_{|\bm{S}|} (v/u;q)_{|\bm{T}|-|\bm{S}|}}{(v;q)_{|\bm{T}|}}
(v/u)^{|\bm{S}|}
q^{\varphi(\bm{T}-\bm{S},\bm{S})}
\prod_{i=1}^{n} \binom{T_i}{S_i}_q,
\end{align*}
where we have used the standard $q$-binomial coefficient $$\binom{b}{a}_q = \dfrac{(q;q)_b}{(q;q)_a(q;q)_{b-a}}, \quad a \leq b.$$ The weights \eqref{fused-weights} provide an explicit evaluation of the right hand side of \eqref{fused-def}, under the identification $r=q^{-N/2}$; however, the formula \eqref{fused-weights} makes sense for arbitrary values of $r \in \mathbb{C}$ (that is, as a rational function in $r$), and we tacitly assume this in what follows.

\subsection{Master Yang--Baxter equation}
\label{ssec:master-yb}

The fused vertex weights satisfy a master Yang--Baxter equation, that contains the previous three Yang--Baxter relations \eqref{RLLa}--\eqref{RLLc} as special cases.

\begin{thm}
Fix a collection of binary strings $\A(1),\A(2),\A(3),\B(1),\B(2),\B(3) \in \{0,1\}^n$ and arbitrary parameters $x,y,r,s,t \in \mathbb{C}$. The weights \eqref{fused-weights} satisfy the equation
\begin{multline}
\label{master}
\sum_{\C(1),\C(2),\C(3)}
\tilde{L}^{(r,s)}_{sx/y}\Big(\A(2),\A(1);\C(2),\C(1)\Big)
\tilde{L}^{(r,t)}_{x}\Big(\A(3),\C(1);\C(3),\B(1)\Big)
\tilde{L}^{(s,t)}_{y}\Big(\C(3),\C(2);\B(3),\B(2)\Big)
\\
=
\sum_{\C(1),\C(2),\C(3)}
\tilde{L}^{(s,t)}_{y}\Big(\A(3),\A(2);\C(3),\C(2)\Big)
\tilde{L}^{(r,t)}_{x}\Big(\C(3),\A(1);\B(3),\C(1)\Big)
\tilde{L}^{(r,s)}_{sx/y}\Big(\C(2),\C(1);\B(2),\B(1)\Big),
\end{multline}
where $\C(1),\C(2),\C(3)$ are summed over all binary strings in $\{0,1\}^n$. 

\end{thm}

\begin{proof}
See \cite[Proposition 5.1.4]{ABW21} for full details.
\end{proof}

The master Yang--Baxter equation \eqref{master} reduces to the three given earlier, namely \eqref{RLLa}--\eqref{RLLc}, by choosing any two of $r,s,t$ to be equal to $q^{-1/2}$, keeping the remaining parameter arbitrary (and up to further relabelling of the spectral parameters $x$ and $y$). Details of these reductions, for the bosonic counterpart of the models discussed in the current text, may be found in \cite[Appendix C]{BorodinW}. In what follows, we will make use of yet another reduction:
\begin{cor}
Fix two integers $a,b \in [0,n]$ and binary strings $\A(2),\A(3),\B(2),\B(3) \in \{0,1\}^n$. The weights \eqref{fund-weights} and \eqref{fused-weights} satisfy the equation
\begin{multline*}
\sum_{c,\C(2),\C(3)}
\tilde{L}^{(r)}_{rx/y}\Big(\A(2),a;\C(2),c\Big)
\tilde{L}^{(s)}_{x}\Big(\A(3),c;\C(3),b\Big)
\tilde{L}^{(r,s)}_{y}\Big(\C(3),\C(2);\B(3),\B(2)\Big)
\\
=
\sum_{c,\C(2),\C(3)}
\tilde{L}^{(r,s)}_{y}\Big(\A(3),\A(2);\C(3),\C(2)\Big)
\tilde{L}^{(s)}_{x}\Big(\C(3),a;\B(3),c\Big)
\tilde{L}^{(r)}_{rx/y}\Big(\C(2),c;\B(2),b\Big),
\end{multline*}
where $c$ is summed over all integers in $[0,n]$ and $\C(2),\C(3)$ are summed over all binary strings in $\{0,1\}^n$. This equation has the following graphical version:
\begin{align}
\label{yb-reduce1}
\sum_{c,\C(2),\C(3)}
\tikz{0.9}{
\draw[lgray,line width=1.5pt,->]
(-2,1) node[above,scale=0.6] {\color{black} $a$} -- (-1,0) node[below,scale=0.6] {\color{black} $c$} -- (1,0) node[right,scale=0.6] {\color{black} $b$};
\draw[lgray,line width=4pt,->] 
(-2,0) node[below,scale=0.6] {\color{black} $\A(2)$} -- (-1,1) node[above,scale=0.6] {\color{black} $\C(2)$} -- (1,1) node[right,scale=0.6] {\color{black} $\B(2)$};
\draw[lgray,line width=4pt,->] 
(0,-1) node[below,scale=0.6] {\color{black} $\A(3)$} -- (0,0.5) node[scale=0.6] {\color{black} $\C(3)$} -- (0,2) node[above,scale=0.6] {\color{black} $\B(3)$};
\node[left] at (-2.2,1) {$x \rightarrow$};
\node[left] at (-2.2,0) {$(y;r) \rightarrow$};
%\node[below] at (0,-1.4) {$\uparrow$};
\node[below] at (0,-1.4) {$(s)$};
}
\quad
=
\quad
\sum_{c,\C(2),\C(3)}
\tikz{0.9}{
\draw[lgray,line width=1.5pt,->] 
(-1,1) node[left,scale=0.6] {\color{black} $a$} -- (1,1) node[above,scale=0.6] {\color{black} $c$} -- (2,0) node[below,scale=0.6] {\color{black} $b$};
\draw[lgray,line width=4pt,->] 
(-1,0) node[left,scale=0.6] {\color{black} $\A(2)$} -- (1,0) node[below,scale=0.6] {\color{black} $\C(2)$} -- (2,1) node[above,scale=0.6] {\color{black} $\B(2)$};
\draw[lgray,line width=4pt,->] 
(0,-1) node[below,scale=0.6] {\color{black} $\A(3)$} -- (0,0.5) node[scale=0.6] {\color{black} $\C(3)$} -- (0,2) node[above,scale=0.6] {\color{black} $\B(3)$};
\node[left] at (-1.5,1) {$x \rightarrow$};
\node[left] at (-1.5,0) {$(y;r) \rightarrow$};
%\node[below] at (0,-1.4) {$\uparrow$};
\node[below] at (0,-1.4) {$(s)$};
}
\end{align}
\end{cor}

\begin{proof}
This is the reduction $r=q^{-1/2}$, $\A(1) = \bm{e}_a$, $\B(1) = \bm{e}_b$ of equation \eqref{master}, followed by the relabelling $s \mapsto r$, $t \mapsto s$.
\end{proof}

\subsection{Fused row operators}
\label{sec:fused-row}

For any integer $N \geq 0$ and non-empty set $I \subset [0,n]$, define the following analogue of the row operators \eqref{C-row}:
\begin{multline}
\label{Cfused-row}
\mathcal{D}_{I}(x;r)
:
\bigotimes_{k=0}^{N}
\ket{\B(k)}
\\
\mapsto
\sum_{\A(0),\ldots,\A(N) \in \{0,1\}^n}
\left(
\tikz{1.2}{
\draw[lgray,line width=4pt,->] (0.5,0) -- (5.5,0);
\foreach\x in {1,...,5}{
\draw[lgray,line width=4pt,->] (\x,-0.5) -- (\x,0.5);
}
\node at (-0.5,0) {$(x;r) \rightarrow$};
\node[left] at (0.5,0) {\fs $\bm{e}_0$};\node[right] at (5.5,0) {\fs $\bm{e}_I$};
\node[below] at (5,-0.5) {\fs $\A(N)$};\node[above] at (5,0.5) {\fs $\B(N)$};
\node[below] at (4,-0.5) {\fs $\cdots$};\node[above] at (4,0.5) {\fs $\cdots$};
\node[below] at (3,-0.5) {\fs $\cdots$};\node[above] at (3,0.5) {\fs $\cdots$};
\node[below] at (2,-0.5) {\fs $\cdots$};\node[above] at (2,0.5) {\fs $\cdots$};
\node[below] at (1,-0.5) {\fs $\A(0)$};\node[above] at (1,0.5) {\fs $\B(0)$};
}
\right)
\bigotimes_{k=0}^{N}
\ket{\A(k)},
\end{multline}
where the quantity
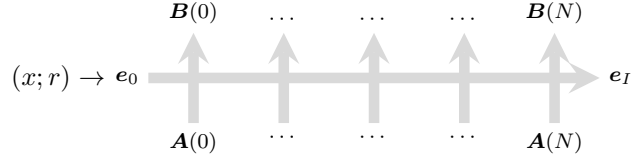
\begin{align*}
\tikz{1.2}{
\draw[lgray,line width=4pt,->] (0.5,0) -- (5.5,0);
\foreach\x in {1,...,5}{
\draw[lgray,line width=4pt,->] (\x,-0.5) -- (\x,0.5);
}
\node at (-0.5,0) {$(x;r) \rightarrow$};
\node[left] at (0.5,0) {\fs $\bm{e}_0$};\node[right] at (5.5,0) {\fs $\bm{e}_I$};
\node[below] at (5,-0.5) {\fs $\A(N)$};\node[above] at (5,0.5) {\fs $\B(N)$};
\node[below] at (4,-0.5) {\fs $\cdots$};\node[above] at (4,0.5) {\fs $\cdots$};
\node[below] at (3,-0.5) {\fs $\cdots$};\node[above] at (3,0.5) {\fs $\cdots$};
\node[below] at (2,-0.5) {\fs $\cdots$};\node[above] at (2,0.5) {\fs $\cdots$};
\node[below] at (1,-0.5) {\fs $\A(0)$};\node[above] at (1,0.5) {\fs $\B(0)$};
}
\end{align*}
is a one-row partition function using the vertex weights \eqref{fused-weights}.\footnote{Note that unless $I = \{0\}$, it is now essential for $N$ to be finite, unlike in the definitions of $\mathcal{C}_i(x)$ and $\mathcal{B}_i(x)$ where $N$ is taken to $\infty$.}

When $r$ is specialized to $r = q^{-p/2}$, with $p \in [1,n]$, we refer to $\mathcal{D}_I(x;r)$ as a row operator of {\it width} $p$; this is in reference to the fact that the horizontal line of the row operator can now carry at most $p$ paths. The case $r=q^{-1/2}$ has a particular significance in what follows; in this case the capacity constraint of the horizontal line imposes that $|I| = 1$, and we have
\begin{align}
\label{unfused}
\mathcal{D}_{\{i\}}(x;q^{-1/2}) \equiv \mathcal{D}_i(x) = T^{\rightarrow}_{0,i}(x;N),
\end{align}
for all $i \in [0,n]$, where the operator on the right hand side is given by \eqref{C-row}. 

\subsection{Commutation relations}

This subsection documents several types of commutation relations between the fused row operators \eqref{Cfused-row} of varying widths. The majority of these results will not be needed until Section \ref{sec:discrete-dist} of the text, where they are used to compute a certain class of partition functions that play a role in our subsequent probability distributions. The reader may prefer to skip this subsection and return to it, as needed, in Section \ref{sec:discrete-dist}.    

\begin{prop}
\label{prop:row-com1}
Fix an integer $i \in [1,n]$ and a set $J \subset [1,n]$ such that $i \in J$. We have the following exchange relation between fused row operators \eqref{Cfused-row} and their unfused counterparts \eqref{unfused}:
\begin{align}
\label{Cfused-com}
\mathcal{D}_i(x)
\mathcal{D}_{J}(y;r)
=
\tilde{L}_{rx/y}^{(r)}(\bm{e}_J,i;\bm{e}_J,i)
\cdot
\mathcal{D}_{J}(y;r)
\mathcal{D}_i(x),
\end{align}
where the coefficient appearing on the right hand side is given by the top-middle entry of the table \eqref{fund-weights}.
\end{prop}

\begin{proof}
We give the proof in the case of row operators of unit length, namely, for $N=0$; however, for generic $N$ the proof follows in exactly the same way. Starting from the relation \eqref{yb-reduce1}, we set $a=0$, $\A(2) = \bm{e}_0$, $b=i$, $\B(2) = \bm{e}_J$, keeping $\A(3)$ and $\B(3)$ arbitrary. The diagonally-oriented vertex on the left hand side freezes; it is given by $\tilde{L}^{(r)}_{rx/y}(\bm{e}_0,0;\bm{e}_0,0) = 1$. Due to the fact that $i \in J$, the diagonally-oriented vertex on the right hand side also freezes; colour $i$ is present in both of the outgoing edges of this vertex, meaning that it must be present in both of the incoming edges (since $\C(2)$ is a binary string). The weight of this frozen vertex is $\tilde{L}_{rx/y}^{(r)}(\bm{e}_J,i;\bm{e}_J,i)$, completing the proof.
\end{proof}

\begin{prop}
\label{prop:row-com2}
Fix an integer $i \in [1,n]$ and a set $J \subset [1,n]$ such that $i \not\in J$. We have the following exchange relation between fused row operators \eqref{Cfused-row} and their unfused counterparts \eqref{unfused}:
\begin{align}
\label{Cfused-com2}
\mathcal{D}_i(x)
\mathcal{D}_J(y;r)
=
\sum_{j \in \{0,i\} \cup J}
\tilde{L}_{rx/y}^{(r)}(\bm{e}_J+\bm{e}_i-\bm{e}_j,j;\bm{e}_J,i)
\cdot
\mathcal{D}_{\{i\}\cup J\backslash \{j\}}(y;r)
\mathcal{D}_j(x),
\end{align}
where the coefficients appearing in the sum are given by the bottom-middle and bottom-right entries of the table \eqref{fund-weights}.
\end{prop}

\begin{proof}
Similarly to the proof of Proposition \ref{prop:row-com2}, one starts from the relation \eqref{yb-reduce1} and sets $a=0$, $\A(2) = \bm{e}_0$, $b=i$, $\B(2) = \bm{e}_J$, keeping $\A(3)$ and $\B(3)$ arbitrary. The diagonally-oriented vertex on the left hand side again freezes with weight $1$. This time, however, the diagonally-oriented vertex on the right hand side is not frozen; this is due to the fact that $i \not\in J$, meaning that colour $i$ is only present in one of the outgoing edges of this vertex. The weight of the diagonally-oriented vertex is seen to be $\tilde{L}_{rx/y}^{(r)}(\bm{e}_J+\bm{e}_i-\bm{e}_j,j;\bm{e}_J,i)$, and the result follows by summing over all possible values of $c=j$.
\end{proof}

Combining Propositions \ref{prop:row-com1} and \ref{prop:row-com2} we obtain the following important result:

\begin{prop}
Fix two integers $p,i \in [1,n]$ and a set $J \subset [1,n]$ of cardinality $|J|=p$. We then have the commutation relation
\begin{align}
\label{important-relation}
\mathcal{D}_i(x)
\mathcal{D}_{J}(x;q^{-p/2})
=
\frac{1-q}{1-q^p}
\cdot
\left\{
\begin{array}{rl}
q^{\alpha_i(J)}
\mathcal{D}_J(x;q^{-p/2})
\mathcal{D}_i(x),
&
i \in J,
\\
\\
\displaystyle{\sum_{j \in J}}
q^{\alpha_i\left(J^{+-}_{ij}\right)}
\mathcal{D}_{J^{+-}_{ij}}(x;q^{-p/2})
\mathcal{D}_j(x),
&
i \not\in J,
\end{array}
\right.
\end{align}
between row operators of width 1 and width $p$ respectively. Here we have defined $J^{+-}_{ij} = \{i\} \cup J \backslash \{j\}$ and $\alpha_i(K)$ denotes the number of elements in the set $K \subset \mathbb{N}$ which exceed $i$; namely, 
$\alpha_i(K) = |\{k \in K : k>i\}|$.
\end{prop}

\begin{proof}
We analyse the cases $i \in J$ and $i \not\in J$ separately. For $i \in J$ we use 
\eqref{Cfused-com} with $x=y$, $r=q^{-p/2}$; under this choice of parameters the coefficient on the right hand side reads
\begin{align*}
\tilde{L}_{rx/y}^{(r)}(\bm{e}_J,i;\bm{e}_J,i)
\Big|_{x=y}
\Big|_{r=q^{-p/2}}
=
\frac{r^2(qx/y-1)q^{\alpha_i(J)}}{1-r^2 x/y}
\Big|_{x=y}
\Big|_{r=q^{-p/2}}
=
\frac{1-q}{1-q^p}
\cdot
q^{\alpha_i(J)},
\end{align*}
and we recover the first line of \eqref{important-relation}. 

For $i \not\in J$ we use \eqref{Cfused-com2} with $x=y$, $r=q^{-p/2}$; this allows two of the terms on the right hand of \eqref{Cfused-com2} to be eliminated. First, we may eliminate the $j=0$ term from the summation; this follows from the fact that $\mathcal{D}_{\{i\}\cup J}(y;r) = 0$ at $r=q^{-p/2}$, since $|\bm{e}_J+\bm{e}_i| = p+1$. Second, we may eliminate the $j=i$ term from the summation, since for $i \not\in J$ one has
\begin{align*}
\tilde{L}_{rx/y}^{(r)}(\bm{e}_J,i;\bm{e}_J,i)
\Big|_{x=y}
\Big|_{r=q^{-p/2}}
=
\frac{r^2(1-x/y)q^{\alpha_i(J)}}{1-r^2 x/y}
\Big|_{x=y}
\Big|_{r=q^{-p/2}}
=
0.
\end{align*}
The remaining terms in the summation are those for which $j \in J$; for those we obtain
\begin{align*}
\tilde{L}_{rx/y}^{(r)}(\bm{e}_J+\bm{e}_i-\bm{e}_j,j;\bm{e}_J,i)
\Big|_{x=y}
\Big|_{r=q^{-p/2}}
=
\frac{1-q}{1-q^p}
\cdot
q^{\alpha_i\left(J^{+-}_{ij}\right)},
\end{align*}
and the second line of \eqref{important-relation} holds.

\end{proof}

\begin{prop}
Fix an integer $p \in [1,n]$ and a set $I = \{i_1,\dots,i_p\} \subset [1,n]$ of cardinality $|I|=p$. We have the following ``peeling'' property between the fused row operators \eqref{Cfused-row} and their unfused counterparts \eqref{unfused}:
\begin{align}
\label{row-relation2}
\mathcal{D}_I(x;q^{-p/2})
=
\sum_{j \in I}
\mathcal{D}_{I \backslash\{j\}}\left(qx;q^{-(p-1)/2}\right)
\mathcal{D}_j(x),
\end{align}
allowing us to extract a row operator of width 1 from a row operator of width $p$.
\end{prop}

\begin{proof}
Extending the definition \eqref{fused-def} of fused vertex weights to row operators ({\it cf.} \cite[Appendix B]{BorodinW}), one finds that
\begin{align}
\label{fused-to-sigma0}
\mathcal{D}_I(x;q^{-p/2})
=
\sum_{\sigma \in \mathfrak{S}_p}
\mathcal{D}_{i_{\sigma(1)}}(x)
\mathcal{D}_{i_{\sigma(2)}}(qx)
\cdots
\mathcal{D}_{i_{\sigma(p)}}(q^{p-1}x),
\end{align}
where the objects appearing on the right hand side are unfused row operators \eqref{unfused}. In particular, for row operators of unit length (namely, for $N=0$), the relation \eqref{fused-to-sigma0} matches precisely with the definition \eqref{fused-def} for $r=q^{-p/2}$, $\B = \bm{e}_0$ and $\D=\bm{e}_I$. Converting the sum over $\mathfrak{S}_p$ into summation over $\mathfrak{S}_{p-1}$ subgroups, and re-fusing the final $p-1$ operators in the resulting summand, we may rewrite \eqref{fused-to-sigma0} as 
\begin{align}
\label{fused-to-sigma}
\mathcal{D}_I(x;q^{-p/2})
=
\sum_{i \in I}
\mathcal{D}_i(x)
\mathcal{D}_{I \backslash\{i\}}\left(qx;q^{-(p-1)/2}\right).
\end{align}
Now from \eqref{Cfused-com2} with $J = I \backslash \{i\}$ and $r=q^{-(p-1)/2}$, one has that
\begin{align}
\label{fused-to-sigma2}
\mathcal{D}_i(x) \mathcal{D}_{I\backslash\{i\}}\left(qx;q^{-(p-1)/2}\right)
=
\sum_{j \in I}
\left(
\left.
\tilde{L}_{rq^{-1}}^{(r)}(\bm{e}_I-\bm{e}_j,j;\bm{e}_I-\bm{e}_i,i)
\right|_{r=q^{-(p-1)/2}}
\right)
\mathcal{D}_{I\backslash\{j\}}\left(qx;q^{-(p-1)/2}\right)
\mathcal{D}_j(x);
\end{align}
the $j=0$ term was dropped from the above sum because 
$\mathcal{D}_{I}(qx;q^{-(p-1)/2})=0$, in view of the fact that this row operator has width $p-1$ while $I$ has cardinality $p$. Summing both sides of \eqref{fused-to-sigma2} over $i \in I$ yields
\begin{align}
\label{reverse-order}
\sum_{i \in I}
\mathcal{D}_i(x) \mathcal{D}_{I\backslash\{i\}}\left(qx;q^{-(p-1)/2}\right)
=
\sum_{j \in I}
\mathcal{D}_{I\backslash\{j\}}\left(qx;q^{-(p-1)/2}\right) \mathcal{D}_j(x),
\end{align}
in view of the stochasticity property
$
\sum_{i \in I}
\tilde{L}_{rq^{-1}}^{(r)}(\bm{e}_I-\bm{e}_j,j;\bm{e}_I-\bm{e}_i,i)
=
1
$ (see \cite[Proposition 2.5.1]{BorodinW}). Combining \eqref{fused-to-sigma} and \eqref{reverse-order}, we have completed the proof.
\end{proof}

\section{LLT measures and Plancherel specialization}

In this section we introduce the probability measures that will be central to this text; they are based on the Lascoux--Leclerc--Thibon polynomials \cite{LLT,ABW21} and their associated Cauchy identity \cite{Lam,ABW21}, and accordingly we refer to them as {\it LLT measures}. In analogy with the Schur and Macdonald processes \cite{OR03,BorodinC}, one may introduce a class of Markov kernels that preserve the form of the LLT measure when they act upon it. Acting consecutively with these Markov kernels, we obtain $n$-tuples of random Gelfand--Tsetlin patterns; one Gelfand--Tsetlin pattern is produced for each of the $n$ colours in our partition functions. The main result of this paper is a complete description of the behaviour of these patterns under a certain asymptotic regime of the underlying measure; this is carried out in Section \ref{sec:asymp}.

The layout of this section is as follows. In Sections \ref{ssec:LLTdef}--\ref{ssec:intLLT} we recall a partition function representation for the LLT polynomials, recently obtained in \cite{CGKM22,ABW21}, and use it to present an integral formula for the latter. In Section \ref{ssec:planch} we apply the Plancherel specialization of the ring of symmetric functions to the integral obtained in Section \ref{ssec:intLLT}, yielding an integral formula for the Plancherel-specialized LLT polynomials. In Sections \ref{ssec:cauchyLLT}--\ref{ssec:markov} we recall the (skew) Cauchy identity for LLT polynomials, and use it to define our LLT measures and associated Markov kernels.

\subsection{Functions $\mathbb{G}_{\lambda/\mu}$ and reduction to LLT polynomials}
\label{ssec:LLTdef}

In Section \ref{ssec:int} we introduced the symmetric rational functions $G_{\lambda/\mu}$ as matrix elements of products of the row operators \eqref{C-row}; we now generalize these, by replacing the row operators in the algebraic construction with their fused analogues \eqref{Cfused-row}.

\begin{defn}
Let $\lambda = (\lambda_1,\dots,\lambda_n)$ be a composition of weight $m$, and fix two $\lambda$-coloured compositions $\mu \in \mathcal{S}_{\lambda}$ and $\nu \in \mathcal{S}_{\lambda}$. Let the corresponding vectors in $\mathbb{V}(\infty)$, $\ket{\mu}_{\lambda}$ and $\ket{\nu}_{\lambda}$, be given by \eqref{A(k)} and \eqref{B(k)}, respectively. For any integer $p \geq 1$ we define the following family of symmetric rational functions:
\begin{align}
\label{G-def-fused}
(-s)^{|\mu|-|\nu|}
\cdot
\mathbb{G}_{\mu/\nu}(\lambda;x_1,\dots,x_p;r_1,\dots,r_p)
=
\bra{\nu}_{\lambda}
\prod_{i=1}^{p}
\mathcal{D}_{\{0\}}(x_i;r_i)
\ket{\mu}_{\lambda},
\end{align}
where the operators $\mathcal{D}_{\{0\}}(x_i;r_i)$ are given by \eqref{Cfused-row}.
\end{defn}

In graphical notation, the definition \eqref{G-def-fused} reads
\begin{align}
\label{G-pf-fused}
(-s)^{|\mu|-|\nu|}
\cdot
\mathbb{G}_{\mu/\nu}(\lambda;x_1,\dots,x_p;r_1,\dots,r_p)
&=
\tikz{0.8}{
\foreach\y in {1,...,5}{
\draw[lgray,line width=4pt,->] (1,\y) -- (8,\y);
}
\foreach\x in {2,...,7}{
\draw[lgray,line width=4pt,->] (\x,0) -- (\x,6);
}
%spectral parameters
\node[left] at (0.5,1) {$(x_1;r_1) \rightarrow$};
\node[left] at (0.5,2) {$(x_2;r_2) \rightarrow$};
\node[left] at (0.5,3) {$\vdots$};
\node[left] at (0.5,4) {$\vdots$};
\node[left] at (0.5,5) {$(x_p;r_p) \rightarrow$};
%top labels
\node[above] at (7,6) {$\cdots$};
\node[above] at (6,6) {$\cdots$};
\node[above] at (5,6) {$\cdots$};
\node[above] at (4,6) {\footnotesize$\bm{A}(2)$};
\node[above] at (3,6) {\footnotesize$\bm{A}(1)$};
\node[above] at (2,6) {\footnotesize$\bm{A}(0)$};
%bottom labels
\node[below] at (7,0) {$\cdots$};
\node[below] at (6,0) {$\cdots$};
\node[below] at (5,0) {$\cdots$};
\node[below] at (4,0) {\footnotesize$\bm{B}(2)$};
\node[below] at (3,0) {\footnotesize$\bm{B}(1)$};
\node[below] at (2,0) {\footnotesize$\bm{B}(0)$};
%right labels
\node[right] at (8,1) {$\bm{e}_0$};
\node[right] at (8,2) {$\bm{e}_0$};
\node[right] at (8,3) {$\vdots$};
\node[right] at (8,4) {$\vdots$};
\node[right] at (8,5) {$\bm{e}_0$};
%left labels
\node[left] at (1,1) {$\bm{e}_0$};
\node[left] at (1,2) {$\bm{e}_0$};
\node[left] at (1,3) {$\vdots$};
\node[left] at (1,4) {$\vdots$};
\node[left] at (1,5) {$\bm{e}_0$};
}
\end{align}
where the vectors $\bm{A}(k), \bm{B}(k)$, $k \geq 0$ are given by \eqref{A(k)}--\eqref{B(k)}.

Two reductions of \eqref{G-def-fused} are of interest. The first is obtained by setting $r_i = q^{-1/2}$ for all $1 \leq i \leq p$; in this case, each fused row operator reduces to its unfused analogue, as described in equation \eqref{unfused}, and we find that
\begin{align*}
\mathbb{G}_{\mu/\nu}(\lambda;x_1,\dots,x_p;q^{-1/2},\dots,q^{-1/2})
=
G_{\mu/\nu}(\lambda;x_1,\dots,x_p).
\end{align*}
The second we record as a theorem, below.

\begin{thm}
The function $\mathbb{G}_{\mu/\nu}(\lambda;x_1,\dots,x_p;r_1,\dots,r_p)$ has well-defined $s \rightarrow 0$ and $r_1,\dots,r_p \rightarrow \infty$ limits. Under these limits, it becomes a polynomial in $(x_1,\dots,x_p)$ with monomial coefficients living in $\mathbb{N}[q]$.
\end{thm}

\begin{proof}
To compute the limit $s \rightarrow 0$, we divide both sides of \eqref{G-def-fused} by $(-s)^{|\mu|-|\nu|}$; on the right hand side, we may distribute the resulting $(-s)^{|\nu|-|\nu|}$ factor within the partition function by assigning a factor of $(-s)^{-1}$ to each horizontal unit step by a path. By \cite[Corollary 8.3.6]{ABW21},
\begin{align}
\label{llt-weights}
\lim_{r \rightarrow\infty}
\lim_{s \rightarrow 0}
(-s)^{-|\D|}
\tilde{L}^{(r,s)}_x(\A,\B;\C,\D)
=
\bm{1}_{\bm{C}+\bm{D} \in \{0,1\}^n}
\cdot
x^{|\D|}
q^{\varphi(\D,\C)+\varphi(\D,\D)},
\qquad
\forall\
\A,\B,\C,\D \in \{0,1\}^n.
\end{align}
Since this limit exists at the level of the individual vertices, it follows that the $s \rightarrow 0$ and $r_1,\dots,r_p \rightarrow \infty$ limits exist when applied to the whole partition function. The fact that the resulting function is a polynomial in $(x_1,\dots,x_p)$, with nonnegative polynomial coefficients in $q$, is manifest from the right hand side of \eqref{llt-weights}.
\end{proof}

Throughout the remainder of the text, we refer to
\begin{align}
\label{LLTdef}
\mathbb{G}_{\mu/\nu}(\lambda;x_1,\dots,x_p;\infty,\dots,\infty)\Big|_{s \rightarrow 0} \equiv \mathbb{G}_{\mu/\nu}(x_1,\dots,x_p)
\end{align}
as a Lascoux--Leclerc--Thibon (LLT) polynomial\footnote{The LLT polynomials have two combinatorial definitions; either in terms of ribbon tilings of a Young diagram, or in terms of $n$-tuples of semi-standard Young tableaux. For both of these definitions, we refer to \cite[Sections 9.1 and 9.2]{ABW21}; for the matching of $\mathbb{G}_{\mu/\nu}(\lambda;x_1,\dots,x_p;\infty,\dots,\infty)\Big|_{s \rightarrow 0}$ with the resulting polynomials we refer to \cite[Theorem 9.3.2 (1)]{ABW21}.}, and tacitly assume that the $s \rightarrow 0$ and $r_1,\dots,r_p \rightarrow \infty$ limits have been taken, unless it is specifically stated otherwise.

\subsection{Padding and shifting LLT polynomials}

To this point, LLT polynomials were indexed by coloured compositions, as given by Definiton \ref{def:cc}. There is a natural way to extend their definition to allow indexing by {\it coloured signatures} (the extension of the set \eqref{lambda-col} that allows parts to take any integer values, including negative ones) which will be convenient when we come to stating the Cauchy identity for LLT polynomials (see Section \ref{ssec:cauchyLLT}).

One may consider the effect of appending an extra column to the left of the partition function \eqref{G-pf-fused}, with the boundary conditions at the top and bottom of this column prescribed as $\bm{A}(-1) = 1^n$ and $\bm{B}(-1) = 1^n$, respectively. One sees that the appended column freezes with weight $1$ (assuming the limit where the weights \eqref{llt-weights} are used), and therefore does not contribute to the overall evaluation of the partition function. This invariance property may be expressed as
\begin{align}
\label{pad}
\mathbb{G}_{-1\cup\mu/-1\cup\nu}(x_1,\dots,x_p)
=
\mathbb{G}_{\mu/\nu}(x_1,\dots,x_p),
\end{align}
for any $\mu,\nu \in \mathcal{S}_{\lambda}$, where $-1 \cup \mu$ and $-1 \cup \nu$ mean prepending a part of size $-1$ in each of the $n$ blocks of $\mu$ and $\nu$, respectively (similarly to Definition \ref{defn:padding}). The procedure \eqref{pad} may clearly be iterated, allowing us to prepend arbitrarily many negative parts to the coloured compositions in question. One also notes that, on the resulting coloured signatures, there holds
\begin{align}
\label{shift}
\mathbb{G}_{(\mu-1)/(\nu-1)}(x_1,\dots,x_p)
=
\mathbb{G}_{\mu/\nu}(x_1,\dots,x_p),
\end{align}
where $(\mu-1)$ and $(\nu-1)$ mean subtracting $1$ from every part of $\mu$ and $\nu$, respectively. It is then easy to see that \eqref{pad} and \eqref{shift} completely determine the value of $\mathbb{G}_{\mu/\nu}(x_1,\dots,x_p)$ for any coloured signatures $\mu$ and $\nu$ (possibly containing infinitely many negative parts\footnote{In situations where infinitely many negative parts occur, only finitely many negative integers will be omitted from our coloured signatures; for more information, see equation \eqref{inf-sig} and the sentence that follows.}).

\subsection{Integral formula for LLT polynomials}
\label{ssec:intLLT}

Applying the fusion procedure to the integral formula obtained in Section \ref{ssec:int}, one may obtain an integral formula for the LLT polynomials; we reproduce that result here, in essentially the same form as it appeared in \cite[Corollary 11.5.3]{ABW21}. 
\begin{thm}
\label{thm:LLT-int}
Fix a composition $\lambda = (\lambda_1,\dots,\lambda_n)$ such that $|\lambda| = m$, and choose two coloured compositions $\mu,\nu \in \mathcal{S}_{\lambda}$. The LLT polynomials \eqref{LLTdef} are given by the following integral expression:
\begin{multline}
\label{LLTint}
\mathbb{G}_{\mu/\nu}(x_1,\dots,x_p)
=
\frac{q^{m(m+1)/2}}{(q-1)^{m}}
\cdot
\left( \frac{1}{2\pi{\tt i}} \right)^{m}
\oint_{C_1}
\frac{dy_1}{y_1}
\cdots 
\oint_{C_m}
\frac{dy_m}{y_m}
\\
\times
\prod_{1 \leq i<j \leq m}
\left(
\frac{y_j-y_i}{y_j-q y_i}
\right)
f_{\breve\mu}(1^m;y_1^{-1},\dots,y_m^{-1})
g_{\nu}(\lambda;y_1,\dots,y_m)
\prod_{i=1}^{p}
\prod_{j=1}^{m}
\frac{1}{1-x_i y_j},
\end{multline}
where the contours $\{C_1,\dots,C_m\}$ are admissible with respect to $(q,0)$. It is implicit that $s=0$ in the functions $f_{\breve\mu}$ and $g_{\nu}$.
\end{thm}

\begin{proof}
We focus on the proof for $p=1$, as this captures the essence of the proof for generic $p$. Fix an integer $P \geq 1$. Using \eqref{G-def-fused}, the one-variable function $\mathbb{G}_{\mu/\nu}(\lambda; x; q^{-P/2})$ is given by
\begin{align*}
(-s)^{|\mu|-|\nu|}
\cdot
\mathbb{G}_{\mu/\nu}(\lambda;x;q^{-P/2})
&=
\bra{\nu}_{\lambda}
\mathcal{D}_{\{0\}}(x;q^{-P/2})
\ket{\mu}_{\lambda}
\\
&=
\bra{\nu}_{\lambda}
\mathcal{D}_0(x)
\mathcal{D}_0(q x)
\ldots
\mathcal{D}_0(q^{P-1} x)
\ket{\mu}_{\lambda}
=
(-s)^{|\mu|-|\nu|}
G_{\mu/\nu}(\lambda;x,qx,\dots,q^{P-1}x)
\end{align*}
where we have replaced the fused row operator $\mathcal{D}_{\{0\}}(x;q^{-P/2})$ by the bundle of unfused row operators \eqref{C-row} which comprise it. From the integral formula \eqref{G-int}, we then have that
\begin{multline*}
\mathbb{G}_{\mu/\nu}(\lambda;x;q^{-P/2})
=
G_{\mu/\nu}(\lambda;x,qx,\dots,q^{P-1}x)
=
\frac{q^{m(m+1)/2}}{(q-1)^m}
\cdot
\left( \frac{1}{2\pi{\tt i}} \right)^m
\oint_{C_1}
\frac{dy_1}{y_1}
\cdots 
\oint_{C_m}
\frac{dy_m}{y_m}
\\
\times
\prod_{1 \leq i<j \leq m}
\left(
\frac{y_j-y_i}{y_j-q y_i}
\right)
f_{\breve\mu}(1^m;y_1^{-1},\dots,y_m^{-1})
g_{\nu}(\lambda;y_1,\dots,y_m)
\prod_{j=1}^{m}
\frac{1-q^P x y_j}{1-x y_j}.
\end{multline*}
This yields an expression for $\mathbb{G}_{\mu/\nu}(\lambda;x;r)$ by performing the analytic continuation $q^P \mapsto r^{-2}$; the $p=1$ case of \eqref{LLTint} then follows by sending $r \rightarrow \infty$. The generic $p$ version of \eqref{LLTint} may be proved along similar lines; namely, we split each of the $p$ fused operators $\mathcal{D}_{\{0\}}(x_i;q^{-P/2})$ in \eqref{G-def-fused} into a bundle of $P$ unfused row operators, and carry out the analysis above on each of the bundles.
\end{proof}

\subsection{Plancherel specialization}
\label{ssec:planch}

Let $\Lambda$ denote the ring of symmetric functions in the (infinite) alphabet $x := (x_1,x_2,\dots)$. As described in \cite[Chapter I]{Macdonald}, the power sum basis of $\Lambda$ is the set of functions
\begin{align*}
p_{\lambda}(x) := \prod_{i \geq 1} p_{\lambda_i}(x),
\qquad
p_k(x) := \sum_{i \geq 1} x_i^k,
\ \
\forall\ k \geq 0,
\end{align*}
where $\lambda$ ranges over all partitions. Any function in $\Lambda$ is expressed as a unique linear combination of the functions $p_{\lambda}(x)$.

Fix an indeterminate $t \in \mathbb{C}$. The Plancherel specialization of $\Lambda$ is the map ${\rm Pl}_t : \Lambda \rightarrow \mathbb{C}$ under which the power sums transform as follows:
\begin{align*}
p_k(x) \mapsto \left\{
\begin{array}{ll}
t, & \quad k=1,
\\
\\
0, & \quad k \geq 2.
\end{array}
\right.
\end{align*}
Following standard notational practice for specializations of the ring of symmetric functions, we denote the image of a function $f \in \Lambda$ under ${\rm Pl}_t$ by $f({\rm Pl}_t)$.

The LLT polynomials \eqref{LLTdef} admit a natural lift to $\Lambda$, obtained by replacing the finite alphabet $(x_1,\dots,x_p)$ by the infinite one $x = (x_1,x_2,\dots)$. Making this replacement in \eqref{LLTint}, we have
\begin{multline}
\label{LLTint-inf}
\mathbb{G}_{\mu/\nu}(x_1,x_2,\dots)
=
\frac{q^{m(m+1)/2}}{(q-1)^{m}}
\cdot
\left( \frac{1}{2\pi{\tt i}} \right)^{m}
\oint_{C_1}
\frac{dy_1}{y_1}
\cdots 
\oint_{C_m}
\frac{dy_m}{y_m}
\\
\times
\prod_{1 \leq i<j \leq m}
\left(
\frac{y_j-y_i}{y_j-q y_i}
\right)
f_{\breve\mu}(1^m;y_1^{-1},\dots,y_m^{-1})
g_{\nu}(\lambda;y_1,\dots,y_m)
\prod_{j=1}^{m}
\exp\left(\sum_{k \geq 1} \frac{p_k(x) y_j^k}{k} \right)
\end{multline}
where we have used the fact that (as a formal power series) there holds
\begin{align*}
\prod_{i \geq 1} \frac{1}{1-x_i y}
=
\exp\left(-\sum_{i \geq 1} \log(1-x_i y) \right)
=
\exp\left( \sum_{i \geq 1} \sum_{k=1}^{\infty} \frac{x_i^k y^k}{k} \right)
=
\exp\left(\sum_{k \geq 1} \frac{p_k(x) y^k}{k} \right).
\end{align*}
We then read off the Plancherel specialization of $\mathbb{G}_{\mu/\nu}$:
\begin{multline}
\label{LLTplanch}
\mathbb{G}_{\mu/\nu}({\rm Pl}_t)
=
\frac{q^{m(m+1)/2}}{(q-1)^{m}}
\cdot
\left( \frac{1}{2\pi{\tt i}} \right)^{m}
\oint_{C_1}
\frac{dy_1}{y_1}
\cdots 
\oint_{C_m}
\frac{dy_m}{y_m}
\\
\times
\prod_{1 \leq i<j \leq m}
\left(
\frac{y_j-y_i}{y_j-q y_i}
\right)
f_{\breve\mu}(1^m;y_1^{-1},\dots,y_m^{-1})
g_{\nu}(\lambda;y_1,\dots,y_m)
\prod_{j=1}^{m}
e^{t y_j}.
\end{multline}
In what follows, we shall further restrict to $t \in \mathbb{R}_{\geq 0}$, where $t$ could be viewed as playing the role of continuous time.

\subsection{Skew Cauchy identity for LLT polynomials}
\label{ssec:cauchyLLT}

Up until now we dealt with coloured compositions of arbitrary colour profile 
$\lambda = (\lambda_1,\dots,\lambda_n)$. Throughout the rest of the paper we shall restrict our attention to the case $\lambda_i = N$ for all $1 \leq i \leq n$, where $N$ is some given positive integer; this means that each colour within a coloured composition $\mu$ is represented exactly $N$ times. We denote the corresponding set of coloured compositions as follows:
\begin{align*}
\mathcal{S}_{N^n}
=
\Big\{ 
\mu 
= 
\Big(\mu^{(1)}_1 < \cdots < \mu^{(1)}_{N} \Big|
\mu^{(2)}_1 < \cdots < \mu^{(2)}_{N} \Big|
\cdots \Big|
\mu^{(n)}_1 < \cdots < \mu^{(n)}_{N}\Big)
\Big\}.
\end{align*}
One element of $\mathcal{S}_{N^n}$ plays a special role; this is the element in which all parts of a coloured composition are as small as they can be. We assign this element the notation $\Delta$:
\begin{align}
\label{nothing}
\Delta
=
(0,1,\dots,N-1 | 0,1,\dots,N-1 | \cdots | 0,1,\dots,N-1)
\in 
\mathcal{S}_{N^n}.
\end{align}
Whenever the lower coloured composition $\nu$ in an LLT polynomial $\mathbb{G}_{\mu/\nu}$ is set equal to $\Delta$, we employ the lighter notation
\begin{align*}
\mathbb{G}_{\mu/\Delta}(x_1,\dots,x_p)
\equiv
\mathbb{G}_{\mu}(x_1,\dots,x_p).
\end{align*}

\begin{defn}
For any coloured composition 
$\mu = \Big( \mu^{(1)}_1 < \cdots < \mu^{(1)}_{N} \Big|
\cdots \Big| \mu^{(n)}_1 < \cdots < \mu^{(n)}_{N} \Big) \in \mathcal{S}_{N^n}$ we define the statistic
\begin{align*}
\psi(\mu)
=
\frac{1}{2}
\sum_{1 \leq i<j \leq n}\
\sum_{a \in \mu^{(i)}}\
\sum_{b \in \mu^{(j)}}
\bm{1}_{a>b}.
\end{align*}
\end{defn}

\begin{thm}
\label{thm:cauchy}
Fix two positive integers $p$ and $N$, and two alphabets $(x_1,\dots,x_p)$ and $(y_1,\dots,y_N)$. Let $\nu \in \mathcal{S}_{N^n}$ be a coloured composition. The LLT polynomials \eqref{LLTdef} satisfy the Cauchy summation identity
\begin{align}
\label{skew-cauchy}
\sum_{\mu \in \mathcal{S}_{N^n}}
q^{-2\psi(\mu)}
\mathbb{G}_{\mu/\nu}(x_1,\dots,x_p)
\mathbb{G}_{\mu}(y_1,\dots,y_N)
=
\prod_{i=1}^{p}
\prod_{j=1}^{N}
\frac{1}{(x_i y_j;q)_n}
\cdot
q^{-2\psi(\nu)}
\mathbb{G}_{\nu}(y_1,\dots,y_N),
\end{align}
where $\mathbb{G}_{\mu}(y_1,\dots,y_N) \equiv \mathbb{G}_{\mu/\Delta}(y_1,\dots,y_N)$ and $\mathbb{G}_{\nu}(y_1,\dots,y_N) \equiv \mathbb{G}_{\nu/\Delta}(y_1,\dots,y_N)$. This holds either as a formal power series, or as a numeric equality as long as $|q| <1$ and $|x_i y_j| <1$ for all $i,j$.
\end{thm}

\begin{proof}
This was originally obtained in \cite[Theorem 35]{Lam}. For a formulation in terms of the vertex model setup of the current text, we refer to \cite[Corollary 9.4.1]{ABW21} and \cite[Proposition 6.12]{CGKM22}.
\end{proof}

We make a small but important adjustment to the Cauchy identity \eqref{skew-cauchy}. For any $N \geq 1$, introduce the set of coloured signatures
\begin{align}
\label{col-signat}
\mathcal{S}(N)
=
\left\{
\mu
=
\left(\mu^{(1)} | \mu^{(2)} | \cdots | \mu^{(n)}\right)
\right\},
\end{align}
where for all $1 \leq i \leq n$ the components
\begin{align}
\label{inf-sig}
\mu^{(i)} = \left(\cdots < \mu^{(i)}_{-1} < \mu^{(i)}_0 < \mu^{(i)}_1 < \cdots < \mu^{(i)}_N \right)
\end{align}
are left-infinite strict signatures such that $\mu^{(i)}_k \not= k-1$ for only finitely many $k \in \mathbb{Z}$. We continue to denote by $\Delta$ the unique element in $\mathcal{S}(N)$ in which all signature parts are minimal. For any fixed $\nu \in \mathcal{S}(N)$, one then has that
\begin{align}
\label{skew-cauchy-pad}
\sum_{\mu \in \mathcal{S}(N)}
q^{-2\psi(\mu,\nu)}
\mathbb{G}_{\mu/\nu}(x_1,\dots,x_p)
\mathbb{G}_{\mu}(y_1,y_2,\dots)
=
\prod_{i=1}^{p}
\prod_{j\geq 1}
\frac{1}{(x_i y_j;q)_n}
\mathbb{G}_{\nu}(y_1,y_2,\dots),
\end{align}
in which the size of the alphabet $(y_1,y_2,\dots)$ is now infinite.\footnote{In particular, this will allow us to Plancherel-specialize this alphabet.} Here $\mathbb{G}_{\mu}(y_1,y_2,\dots) \equiv \mathbb{G}_{\mu/\Delta}(y_1,y_2,\dots)$ and $\mathbb{G}_{\nu}(y_1,y_2,\dots) \equiv \mathbb{G}_{\nu/\Delta}(y_1,y_2,\dots)$ as previously, and 
\begin{align}
\label{new-psi}
\psi(\mu,\nu)
=
\frac{1}{2}
\sum_{1 \leq i<j \leq n}\
\sum_{a \in \mu^{(i)}}\
\sum_{b \in \mu^{(j)}}
\bm{1}_{a>b>m}
-
\frac{1}{2}
\sum_{1 \leq i<j \leq n}\
\sum_{a \in \nu^{(i)}}\
\sum_{b \in \nu^{(j)}}
\bm{1}_{a>b>m},
\end{align}
with $m$ chosen to be any integer such that $\mu^{(i)}_k =\nu^{(i)}_k = k-1$ for all $k \leq m$ and $1 \leq i \leq n$. Equation \eqref{skew-cauchy-pad} holds as an identity of formal power series, which converges if $|q| <1$ and $|x_i y_j| <1$ for all $i,j$.

The claim \eqref{skew-cauchy-pad} is established by taking \eqref{skew-cauchy} with $N$ becoming arbitrarily large, and applying \eqref{pad} and \eqref{shift} appropriately to convert the indices of all functions to members of the set \eqref{col-signat}. It is easily verified that the quantity $\psi(\mu)-\psi(\nu)$ is invariant under such paddings and shifts, and may be written in the form \eqref{new-psi}.

\subsection{Markov kernels}
\label{ssec:markov}

We proceed to introduce probability measures from the skew Cauchy identity \eqref{skew-cauchy}. Normalizing so that the right hand side of \eqref{skew-cauchy} is equal to $1$, we have
\begin{align}
\label{skew-cauchy2}
\sum_{\mu \in \mathcal{S}(N)}
\prod_{i=1}^{p}
\prod_{j \geq 1}
(x_i y_j;q)_n
\cdot
q^{-2\psi(\mu,\nu)}
\mathbb{G}_{\mu/\nu}(x_1,\dots,x_p)
\frac{\mathbb{G}_{\mu}(y_1,y_2,\dots)}{\mathbb{G}_{\nu}(y_1,y_2,\dots)}
=
1.
\end{align}
In view of this sum-to-unity property, the summands in \eqref{skew-cauchy2} may be viewed as probabilities of transitioning from an initial coloured signature $\nu \in \mathcal{S}(N)$ to a final one $\mu \in \mathcal{S}(N)$. Many choices of the parameters $(x_1,\dots,x_p)$ and $(y_1,y_2,\dots)$ are possible, leading to a variety of interesting distributions, but in this work we focus on one particular choice; namely, we let $(x_1,\dots,x_p) = (1,\dots,1) \equiv 1^p$ and take the ${\rm Pl}_t$ specialization of the alphabet $(y_1,y_2,\dots)$. Under this choice, \eqref{skew-cauchy2} becomes
\begin{align*}
\sum_{\mu \in \mathcal{S}(N)}
\exp\left( -\frac{p(1-q^n)}{1-q} t \right)
\mathbb{G}_{\mu/\nu}(1^p)
q^{-2\psi(\mu,\nu)}
\frac{\mathbb{G}_{\mu}({\rm Pl}_t)}{\mathbb{G}_{\nu}({\rm Pl}_t)}
=
1.
\end{align*}
From this we introduce the Markov kernel $\mathbb{P}_{t,p} : V(\mathcal{S}(N)) \rightarrow V(\mathcal{S}(N))$ with matrix elements given by
\begin{align}
\label{skew-cauchy3}
\mathbb{P}_{t,p}(\nu \rightarrow \mu)
=
q^{-2\psi(\mu,\nu)}
\exp\left( -\frac{p(1-q^n)}{1-q}t \right)
\mathbb{G}_{\mu/\nu}(1^p)
\dfrac{\mathbb{G}_{\mu}({\rm Pl}_t)}
{\mathbb{G}_{\nu}({\rm Pl}_t)},
\end{align}
where $V(\mathcal{S}(N))$ denotes the complex linear span of the elements of $\mathcal{S}(N)$. Abusing notation slightly, whenever we write $\mathbb{P}_{t,p}(\nu)$  for some $\nu \in \mathcal{S}(N)$, this means a random coloured signature $\mu \in \mathcal{S}(N)$ sampled from the distribution \eqref{skew-cauchy3}.

\begin{rmk}
Throughout the rest of the text, we shall only be concerned with evaluating the kernel \eqref{skew-cauchy3} on coloured signatures $\mu,\nu$ such that $\mu^{(i)}_k = \nu^{(i)}_k = k-1$ for all $k \leq 0$ and $1 \leq i \leq n$. When coloured signatures have such a property, by slight abuse of notation we continue to write $\mu,\nu \in \mathcal{S}_{N^n} \subset \mathcal{S}(N)$ and shall still refer to these objects as coloured compositions.
\end{rmk}

Below we collect some elementary facts about the Markov kernel \eqref{skew-cauchy3}.

\begin{prop}
\label{prop:non-skew}
For any integer $p \geq 1$, real parameter $t \in \mathbb{R}_{\geq 0}$ and coloured composition $\mu \in \mathcal{S}_{N^n}$, we have
\begin{align}
\label{mu-meas}
\mathbb{P}_{t,p}(\Delta \rightarrow \mu)
=
q^{-2\psi(\mu)+\binom{n}{2}\binom{N}{2}}
\exp\left( -\frac{p(1-q^n)}{1-q} t \right)
\mathbb{G}_{\mu}(1^p)
\mathbb{G}_{\mu}({\rm Pl}_t).
\end{align}
\end{prop}

\begin{proof}
This is just the $\nu = \Delta$ case of \eqref{skew-cauchy3}, noting that $\mathbb{G}_{\Delta} = 1$ and $\psi(\Delta)=\frac{1}{2}\binom{n}{2}\binom{N}{2}$.
\end{proof}

\begin{prop}
For any two integers $p_1,p_2 \geq 1$ and real parameter $t \in \mathbb{R}_{\geq 0}$, the maps $\mathbb{P}_{t,p_1}$ and $\mathbb{P}_{t,p_2}$ compose according to the rule
\begin{align}
\label{compose}
\mathbb{P}_{t,p_1} \circ \mathbb{P}_{t,p_2} = \mathbb{P}_{t,p_1+p_2}.
\end{align}
\end{prop}

\begin{proof}
For fixed $\lambda, \nu \in \mathcal{S}_{N^n}$ one computes
\begin{align*}
&
\sum_{\mu \in \mathcal{S}_{N^n}}
\mathbb{P}_{t,p_1}(\nu \rightarrow \mu)
\mathbb{P}_{t,p_2}(\mu \rightarrow \lambda)
\\
\\
&=
q^{-2(\psi(\lambda)-2\psi(\nu))}
\exp\left( -\frac{(p_1+p_2)(1-q^n)}{1-q}t \right)
\frac{\mathbb{G}_{\lambda}({\rm Pl}_t)}{\mathbb{G}_{\nu}({\rm Pl}_t)}
\sum_{\mu \in \mathcal{S}_{N^n}}
\mathbb{G}_{\mu/\nu}(1^{p_1})
\mathbb{G}_{\lambda/\mu}(1^{p_2})
\\
\\
&=
q^{-2(\psi(\lambda)-2\psi(\nu))}
\exp\left( -\frac{(p_1+p_2)(1-q^n)}{1-q}t \right)
\frac{\mathbb{G}_{\lambda}({\rm Pl}_t)}{\mathbb{G}_{\nu}({\rm Pl}_t)}
\mathbb{G}_{\lambda/\nu}(1^{p_1+p_2})
=
\mathbb{P}_{t,p_1+p_2}(\nu \rightarrow \lambda),
\end{align*}
where we have used the branching rule for LLT polynomials (see \cite[Remark 9.1.1]{ABW21}) to produce the second equality. 
\end{proof}

In view of the property \eqref{compose}, we may view the Markov kernel $\mathbb{P}_{t,p}$ as the composition of $p$ kernels $\mathbb{P}_{t,1}$. Starting from the trivial state 
$\Delta \in \mathcal{S}_{N^n}$, we may either act directly with $\mathbb{P}_{t,p}$ to obtain a random coloured composition $\lambda^{[p]} = \mathbb{P}_{t,p}(\Delta)$, distributed according to \eqref{mu-meas}, or we may act $p$ times with $\mathbb{P}_{t,1}$, producing a chain of $p$ random coloured compositions
\begin{align}
\label{chain}
\Delta
\xrightarrow{\mathbb{P}_{t,1}}
\lambda^{[1]}
\xrightarrow{\mathbb{P}_{t,1}}
\lambda^{[2]}
\xrightarrow{\mathbb{P}_{t,1}}
\cdots
\xrightarrow{\mathbb{P}_{t,1}}
\lambda^{[p]}.
\end{align}
Our goal in the following section is to study the asymptotic behaviour of the distribution of the whole sequence $(\lambda^{[1]},\dots,\lambda^{[p]})$, as $t \rightarrow \infty$, with $p$ kept finite.

\section{Asymptotics}
\label{sec:asymp}

In this section we carry out an asymptotic analysis of the Markov kernel \eqref{skew-cauchy3} with $p=1$, as $t \rightarrow \infty$; this analysis proceeds in several steps. We begin by rewriting coloured compositions in terms of a pair of vectors $\vec{\ell}$ and $\vec{c}$ in Section \ref{ssec:coord-not}; $\vec{\ell}$ encodes the coordinates of the parts in a coloured composition, while $\vec{c}$ encodes the colour sequencing of its parts. In Section \ref{ssec:scaling}, we specify a particular time-dependent scaling of the coordinates associated to $\mu$ and $\nu$ within the function $\mathbb{P}_{t,1}$. We also impose certain interlacing constraints on the coordinates of $\mu$ and $\nu$; for finite $t$, these constraints prohibit certain coloured compositions on which the measure is non-zero, but it later transpires that as $t \rightarrow \infty$ these forbidden compositions naturally occur with vanishingly small probability, allowing us to omit them from our considerations.

Having fixed our choice of scaling and our interlacing assumptions, we proceed to the analysis of the individual factors in the measure \eqref{skew-cauchy3}. Section \ref{ssec:analysis1} deals with the factor $\mathbb{G}_{\mu/\nu}(1)$, whose analysis can be accessed by direct combinatorial means, while Section \ref{ssec:analysis2} deals with the factors $\mathbb{G}_{\mu}({\rm Pl}_t)$ and $\mathbb{G}_{\nu}({\rm Pl_t})$, which are analysed by steepest descent method applied to the integral formula \eqref{LLTplanch}. Our final formula is presented in Section \ref{ssec:formula}; we show that in the $t \rightarrow \infty$ limit being studied, the measure \eqref{skew-cauchy3} degenerates into the product of transition densities for $n$ independent GUE corners processes, multiplied by a discrete measure that is valued on colour sequences.

\subsection{Coordinate and colour sequence notation}
\label{ssec:coord-not}

\begin{defn}
\label{def:coord}
To every coloured composition $\mu = \left(\mu^{(1)}_1 < \cdots < \mu^{(1)}_N | \cdots | \mu^{(n)}_1 < \cdots < \mu^{(n)}_N\right) \in \mathcal{S}_{N^n}$ we associate three vectors $\vec{\ell} = (\ell_1,\dots,\ell_{nN}) \in \mathbb{N}^{nN}$, $\vec{c} = (c_1,\dots,c_{nN}) \in [1,n]^{nN}$, $\vec{b} = (b_1,\dots,b_{nN}) \in [1,N]^{nN}$ satisfying the relation
\begin{align*}
\ell_i
=
\mu^{(c_i)}_{b_i},
\qquad
\forall\ 1 \leq i \leq nN,
\end{align*}
and satisfying the properties {\bf (a)} $\ell_i \leq \ell_{i+1}$ for all $1 \leq i < nN$; {\bf (b)} $c_i < c_{i+1}$ if $\ell_i = \ell_{i+1}$, for all $1 \leq i < nN$; {\bf (c)} $b_i \not= b_j$ if $c_i = c_j$, for all $1 \leq i < j < nN$. 

More informally, $\vec{\ell}$ is the unique vector obtained by sorting the parts of $\mu$ in increasing order; we refer to it as the {\it coordinate vector} of $\mu$. The vector $\vec{c}$ records the colours of the parts of $\mu$ once it has been sorted in increasing order, with an increasing criterion imposed on these colours in the case of ties; we refer to it as the {\it colour sequence} of $\mu$. The vector $\vec{b}$ has been introduced only for the purpose of the making our definitions unambiguous, and plays no role in the rest of the paper.
\end{defn}

\subsection{Starting assumptions and scaling}
\label{ssec:scaling}

Throughout the rest of this section we will be concerned with the analysis of the Markov kernel $\mathbb{P}_{t,1}\left(\nu \rightarrow \mu \right)$ as given by \eqref{skew-cauchy3}, with $\nu = 0\cup\lambda^{[m]}$ and $\mu = \lambda^{[m+1]}$, where we have chosen
$\lambda^{[m]} \in \mathcal{S}^{+}_{m^n}$ and $\lambda^{[m+1]} \in \mathcal{S}^{+}_{(m+1)^n}$ (we remind the reader that the meaning of these notations is given by Definition \ref{defn:padding}). Under such choices, the kernel \eqref{skew-cauchy3} becomes 
\begin{align*}
\mathbb{P}_{t,1}\left(0 \cup \lambda^{[m]} \rightarrow \lambda^{[m+1]}\right)
=
q^{-2\left(\psi\left(\lambda^{[m+1]}\right)-\psi\left(0\cup\lambda^{[m]}\right)\right)}
\exp\left( -\frac{1-q^n}{1-q}t \right)
\mathbb{G}_{\lambda^{[m+1]}/0\cup\lambda^{[m]}}(1)
\dfrac{\mathbb{G}_{\lambda^{[m+1]}}({\rm Pl}_t)}
{\mathbb{G}_{0\cup\lambda^{[m]}}({\rm Pl}_t)}.
\end{align*}
Noting that
\begin{align*}
\mathbb{G}_{0\cup\lambda^{[m]}}({\rm Pl}_t) \equiv 
\mathbb{G}_{0\cup\lambda^{[m]}/\Delta}({\rm Pl}_t) = \mathbb{G}_{\lambda^{[m]}-1}({\rm Pl}_t),
\end{align*}
where $\lambda^{[m]}-1$ means subtraction of $1$ from every part of $\lambda^{[m]}$, we have that
\begin{align}
\label{kernel}
\mathbb{P}_{t,1}\left(0\cup\lambda^{[m]} \rightarrow \lambda^{[m+1]}\right)
=
q^{-2\left(\psi\left(\lambda^{[m+1]}\right)-\psi\left(0\cup\lambda^{[m]}\right)\right)}
\exp\left( -\frac{1-q^n}{1-q}t \right)
\mathbb{G}_{\lambda^{[m+1]}/0\cup\lambda^{[m]}}(1)
\dfrac{\mathbb{G}_{\lambda^{[m+1]}}({\rm Pl}_t)}
{\mathbb{G}_{\lambda^{[m]}-1}({\rm Pl}_t)}.
\end{align}
We shall make some assumptions concerning the coloured compositions $\lambda^{[m]} \in \mathcal{S}^{+}_{m^n}$ and $\lambda^{[m+1]} \in \mathcal{S}^{+}_{(m+1)^n}$ appearing within this formula. Following Definition \ref{def:coord} we represent them in terms of their corresponding coordinate vectors and colour sequences:
\begin{align}
\label{identify}
\lambda^{[m]}
\leftrightarrow 
\left( 
\ell^{[m]}_1,\dots,\ell^{[m]}_{nm} 
\Big| 
c^{[m]}_1,\dots,c^{[m]}_{nm} 
\right),
\qquad
\lambda^{[m+1]}
\leftrightarrow 
\left( 
\ell^{[m+1]}_1,\dots,\ell^{[m+1]}_{n(m+1)} 
\Big| 
c^{[m+1]}_1,\dots,c^{[m+1]}_{n(m+1)} 
\right),
\end{align} 
and we work directly with these vectors in what follows. Our first assumption is that the coordinates $\{\ell_i^{[m]}\}_{1 \leq i \leq nm}$ and $\{\ell_j^{[m+1]}\}_{1 \leq j \leq n(m+1)}$ are strictly increasing and obey the interlacing constraints
\begin{align}
\label{interlace-disc}
\ell^{[m+1]}_{j(m+1)+i} < \ell^{[m]}_{jm+i} < \ell^{[m+1]}_{j(m+1)+i+1},
\qquad
\forall\ i \in [1,m],\quad j \in [0,n-1].
\end{align}
Informally, this means that the coordinates $\{\ell_i^{[m]}\}_{1 \leq i \leq nm}$ and $\{\ell_j^{[m+1]}\}_{1 \leq j \leq n(m+1)}$ are each grouped into $n$ bundles of equal size, and coordinates within those bundles interlace; see Figure \ref{fig:gue}.

We will subsequently see that \eqref{kernel} depends on the coordinates 
$\{\ell_i^{[m]}\}_{1 \leq i \leq nm}$ and $\{\ell_j^{[m+1]}\}_{1 \leq j \leq n(m+1)}$ analytically. Our second assumption will be that these coordinates are analytically continued to real values, by setting
\begin{align}
\label{coord-scal}
\ell^{[k]}_i
\mapsto
q^{n-\lceil i/k \rceil} t
+
(q^{n-\lceil i/k \rceil} t)^\frac{1}{2}
x^{[k]}_i,
\qquad
1 \leq i \leq nk,
\qquad
k \in \{m,m+1\},
\end{align}
with $\lceil i/k \rceil$ denoting the ceiling function, and where
\begin{align}
\label{real-seq}
\left( x^{[m]}_1 < \cdots < x^{[m]}_{nm} \right) \in \mathbb{R}^{nm},
\qquad
\left( x^{[m+1]}_1 < \cdots < x^{[m+1]}_{n(m+1)} \right) \in \mathbb{R}^{n(m+1)}
\end{align}
are sequences of reals that obey the interlacing constraints 
\begin{align}
\label{interlace-real}
x^{[m+1]}_{j(m+1)+i} < x^{[m]}_{jm+i} < x^{[m+1]}_{j(m+1)+i+1},
\qquad
\forall\ i \in [1,m],\quad j \in [0,n-1].
\end{align}
Note that \eqref{interlace-real} is simply the translation of the earlier interlacing constraint \eqref{interlace-disc} to the real variables that now parametrize our coordinates.

We note that there exist choices of the coordinates $\{\ell_j^{[m+1]}\}_{1 \leq j \leq n(m+1)}$ which violate the constraints \eqref{interlace-disc} and yet have non-zero weight in the measure \eqref{kernel}. We refer to such choices as {\it unfavourable coordinates}. Our main result will be to show that under the scaling \eqref{coord-scal}, unfavourable coordinates do not occur with probability converging to $1$ as $t \rightarrow \infty$. We do this by showing that as $t \rightarrow \infty$ the quantity \eqref{kernel} weakly converges to the product of a continuous transition density $\rho_{\rm GUE}\left(x^{[m]} \rightarrow x^{[m+1]}\right)$ valued on interlacing real sequences \eqref{real-seq} and a discrete transition probability $\mathbb{P}_{\rm col}\left(c^{[m]} \rightarrow c^{[m+1]}\right)$ valued on colour sequences \eqref{identify}. In demonstrating that the resulting quantity integrates to unity, we prove that \eqref{coord-scal} captures the correct law of large numbers of the coordinates, with 
$\rho_{\rm GUE}\left(x^{[m]} \rightarrow x^{[m+1]}\right)$ providing the fluctuations.

\begin{figure}
\begin{tikzpicture}[scale=0.8]
%level 1
\draw[red,line width=0.7pt,->] (-0.05,1) -- (-0.05,2.05) -- (16,2.05) -- (16,3.05) -- (17.3,3.05) -- (17.3,4.05) -- (18.1,4.05) -- (18.1,5);
\draw[green,line width=0.7pt,->] (0,1) -- (0,2) -- (4,2) -- (4,3) -- (15.1,3) -- (15.1,4) -- (16.2,4) -- (16.2,5);
\draw[blue,line width=0.7pt,->] (0.05,1) -- (0.05,1.95) -- (8,1.95) -- (8,2.95) -- (8.9,2.95) -- (8.9,3.95) -- (14.6,3.95) -- (14.6,5);
%%level 2
\draw[red,line width=0.7pt,->] (-0.25,1) -- (-0.25,3.05) -- (4.3,3.05) -- (4.3,4.05) -- (9.5,4.05) -- (9.5,5);
\draw[green,line width=0.7pt,->] (-0.2,1) -- (-0.2,3) -- (3.4,3) -- (3.4,4) -- (4.8,4) -- (4.8,5);
\draw[blue,line width=0.7pt,->] (-0.15,1) -- (-0.15,2.95) -- (7.3,2.95) -- (7.3,3.95) -- (8.1,3.95) -- (8.1,5);
%%level 3
\draw[red,line width=0.7pt,->] (-0.45,1) -- (-0.45,4.05) -- (3.9,4.05) -- (3.9,5);
\draw[green,line width=0.7pt,->] (-0.4,1) -- (-0.4,4) -- (3,4) -- (3,5);
\draw[blue,line width=0.7pt,->] (-0.35,1) -- (-0.35,3.95) -- (7,3.95) -- (7,5);
%nodes
\node at (4,2) {$\bullet$};
\node[below] at (4,2) {$\ell_1^{[1]}$};
\node at (8,2) {$\bullet$};
\node[below] at (8,2) {$\ell_2^{[1]}$};
\node at (16,2) {$\bullet$};
\node[below] at (16,2) {$\ell_3^{[1]}$};
\node at (3.4,3) {$\bullet$};
\node[below] at (3.4,3) {$\ell_1^{[2]}$};
\node at (4.3,3) {$\bullet$};
\node[below] at (4.3,3) {$\ell_2^{[2]}$};
\node at (7.3,3) {$\bullet$};
\node[below] at (7.3,3) {$\ell_3^{[2]}$};
\node at (8.9,3) {$\bullet$};
\node[below] at (8.9,3) {$\ell_4^{[2]}$};
\node at (15.1,3) {$\bullet$};
\node[below] at (15.1,3) {$\ell_5^{[2]}$};
\node at (17.3,3) {$\bullet$};
\node[below] at (17.3,3) {$\ell_6^{[2]}$};
\node at (3,4) {$\bullet$};
\node[below] at (3,4) {$\ell_1^{[3]}$};
\node at (3.9,4) {$\bullet$};
\node[below] at (3.9,4) {$\ell_2^{[3]}$};
\node at (4.8,4) {$\bullet$};
\node[below] at (4.8,4) {$\ell_3^{[3]}$};
\node at (7,4) {$\bullet$};
\node[below] at (7,4) {$\ell_4^{[3]}$};
\node at (8.1,4) {$\bullet$};
\node[below] at (8.1,4) {$\ell_5^{[3]}$};
\node at (9.5,4) {$\bullet$};
\node[below] at (9.5,4) {$\ell_6^{[3]}$};
\node at (14.6,4) {$\bullet$};
\node[below] at (14.6,4) {$\ell_7^{[3]}$};
\node at (16.2,4) {$\bullet$};
\node[below] at (16.2,4) {$\ell_8^{[3]}$};
\node at (18.1,4) {$\bullet$};
\node[below] at (18.1,4) {$\ell_9^{[3]}$};
\end{tikzpicture}
\caption{A schematic representation of successive application of the Markov kernel $\mathbb{P}^{+}_{t,1}$ to the empty state $\Delta$, in the case $n=3$; each vertical unit step corresponds to such an application. As $t \rightarrow \infty$, the paths drift into bundles located a distance $q^{n-i}t$ from the origin, for $1 \leq i \leq n$. The $i$-th bundle tends to a GUE corners process centred at $q^{n-i}t$, with fluctuations on the order of $(q^{n-i}t)^{1/2}$.}
\label{fig:gue}
\end{figure}
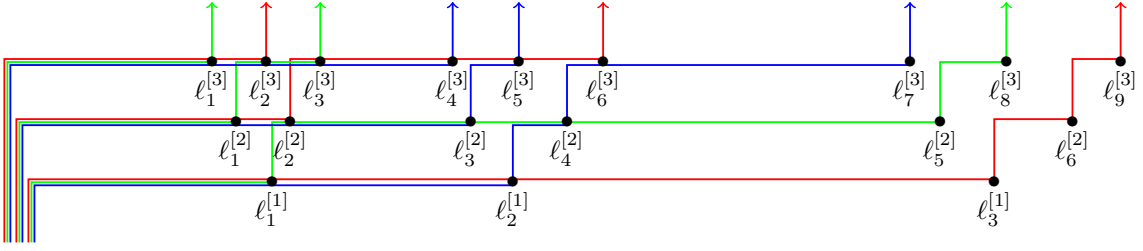

\subsection{Main result}
\label{ssec:main}

\begin{defn}[GUE corners process]
The Gaussian Unitary Ensemble (GUE) of rank $m$ is the collection of $m \times m$ Hermitian matrices $M = (M_{ij})_{i,j=1}^{m}$, where $M=(X+X^{*})/2$ and $X = (X_{ij})_{i,j=1}^{m}$ denotes an $m \times m$ matrix of i.i.d. complex Gaussian random variables $X_{ij} \sim \mathcal{N}(0,1) + {\tt i} \mathcal{N}(0,1)$. For all $1 \leq k \leq m$, write the eigenvalues of the $k \times k$ top-left sub-matrix of $M$ as $\theta_1^{[k]} \leq \cdots \leq \theta_k^{[k]}$. The joint law of the eigenvalues $\theta_i^{[j]}$, $1 \leq i \leq j$, $j \in [1,m]$ is called the {\it GUE corners process} of rank $m$.
\end{defn}

Following \cite[Theorem 20.1]{gorin-notes}, one has the following explicit formula for the density of the GUE corners process:
\begin{prop}
\label{prop:gue-dens}
The array $\theta_i^{[j]}$, $1 \leq i \leq j$, $j \in [1,m]$ has joint density
\begin{align}
\rho\left(\theta_i^{[j]} = x_i^{[j]},1 \leq i \leq j \leq m\right)
=
\bm{1}_{x^{[1]} \prec \cdots \prec x^{[m]}}
\left( \frac{1}{2\pi} \right)^{m/2}
\prod_{1 \leq i<j \leq m} (x^{[m]}_j-x^{[m]}_i)
\prod_{i=1}^{m} e^{-\frac{1}{2} \left( x^{[m]}_i \right)^2}
\end{align}
with respect to the $m(m+1)/2$-dimensional Lebesgue measure.
\end{prop}

Proposition \ref{prop:gue-dens} implies (see also \cite[Equation (20.2)]{gorin-notes}) the conditional probability density for the eigenvalues $\theta_i^{[m+1]}$, $1 \leq i \leq m+1$ of the $(m+1) \times (m+1)$ top-left sub-matrix, given those of the $m \times m$ one:
\begin{multline*}
\rho\left(
\theta_i^{[m+1]} = x_i^{[m+1]}, 1 \leq i \leq m+1 
\Big| 
\theta_i^{[m]} = x_i^{[m]}, 1 \leq i \leq m
\right)
\\
=
\bm{1}_{x^{[m+1]} \succ x^{[m]}}
\frac{1}{(2\pi)^{1/2}}
\frac{\prod_{1 \leq i<j \leq m+1} \left(x^{[m+1]}_j-x^{[m+1]}_i\right) 
\cdot
\prod_{i=1}^{m+1} e^{-\frac{1}{2} \left( x^{[m+1]}_i \right)^2}}
{\prod_{1 \leq i<j \leq m} \left(x^{[m]}_j-x^{[m]}_i\right) 
\cdot 
\prod_{i=1}^{m} e^{-\frac{1}{2} \left( x^{[m]}_i \right)^2}},
\end{multline*}
and for notational compactness, we shall write
\begin{align*}
\rho_{\rm GUE}\left(
x^{[1]}\prec \cdots \prec x^{[m]} \right)
&
:=
\rho\left(\theta_i^{[j]} = x_i^{[j]},1 \leq i \leq j \leq m\right),
\\
\rho_{\rm GUE}\left(
x^{[m]}
\rightarrow
x^{[m+1]}
\right)
&
:=
\rho\left(
\theta_i^{[m+1]} = x_i^{[m+1]}, 1 \leq i \leq m+1 
\Big| 
\theta_i^{[m]} = x_i^{[m]}, 1 \leq i \leq m
\right).
\end{align*}

\begin{thm}
\label{thm:main}
In the asymptotic regime described by \eqref{coord-scal}, the Markov kernel \eqref{kernel} weakly converges to a product of $n$ independent probability measures with densities in the GUE corners process, multiplied by a factor that depends only on the colour sequences \eqref{identify}:
\begin{multline}
\label{main-result}
\mathbb{P}_{t,1}\left(0\cup\lambda^{[m]} \rightarrow \lambda^{[m+1]}\right)
\\
\rightarrow
\prod_{i=1}^{n}
\rho_{\rm GUE}\left(
x_{(i-1)m+1}^{[m]},\dots,x_{im}^{[m]}
\rightarrow
x_{(i-1)(m+1)+1}^{[m+1]},\dots,x_{i(m+1)}^{[m+1]}
\right)
dx^{[m+1]}
\cdot
\mathbb{P}_{\rm col}\left(c^{[m]} \rightarrow c^{[m+1]}\right)
\end{multline}
as $t \rightarrow \infty$, where $dx^{[m+1]}$ denotes the $n(m+1)$-dimensional Lebesgue measure. The final multiplicative factor in \eqref{main-result} is given explicitly by equation \eqref{col-markov} below, and defines a discrete transition probability in a process on colour sequences:
\begin{align}
\label{discrete-sum-to1}
\sum_{c^{[m+1]}}
\mathbb{P}_{\rm col}\left(c^{[m]} \rightarrow c^{[m+1]}\right)
=
1,
\end{align}
where the sum is taken over all 
$c^{[m+1]} = \left(c^{[m+1]}_1,\dots,c^{[m+1]}_{n(m+1)}\right) \in [1,n]^{n(m+1)}$.
\end{thm}

\begin{cor}
\label{cor:main}
Let $\mathbb{P}_{t,N} (\Delta \rightarrow \lambda^{[1]} \rightarrow \cdots \rightarrow \lambda^{[N]})$ denote the joint distribution of coloured compositions $\lambda^{[1]},\dots, \lambda^{[N]}$ generated by $N$ applications of the kernel \eqref{kernel} to the trivial state $\Delta$. In the asymptotic regime described by \eqref{coord-scal}, we have the following weak convergence of measures:
\begin{multline*}
\mathbb{P}_{t,N}
\left(\Delta \rightarrow \lambda^{[1]} \rightarrow \cdots \rightarrow \lambda^{[N]}\right)
\\
\rightarrow
\prod_{i=1}^{n}
\rho_{\rm GUE}
\left( (x^{[1]})_i \prec (x^{[2]})_i \prec \cdots \prec (x^{[N]})_i  \right)
dx^{[1,N]}
\cdot
\mathbb{P}_{\rm col}\left(c^{[1]} \prec c^{[2]} \prec \cdots \prec c^{[N]}\right)
\end{multline*}
as $t \rightarrow \infty$, with $dx^{[1,N]} = \prod_{i=1}^{N} dx^{[i]}$ denoting the $nN(N+1)/2$-dimensional Lebesgue measure. Here we have introduced the shorthand
\begin{align*}
\left(x^{[k]}\right)_i = \left(x^{[k]}_{(i-1)k+1},\dots,x^{[k]}_{ik}\right),
\qquad
\forall\ 1 \leq i \leq n,\ \ 1 \leq k \leq N,
\end{align*}
and $\mathbb{P}_{\rm col}(c^{[1]} \prec c^{[2]} \prec \cdots \prec c^{[N]})$ is a joint distribution on colour sequences given explicitly by \eqref{joint-distr-col} below.
\end{cor}

The remainder of the paper is devoted to the proof of this theorem. Throughout the rest of Section \ref{sec:asymp}, we exhibit the splitting of the Markov kernel \eqref{kernel} as shown on the right hand side of \eqref{main-result}; the proof of the sum-to-unity property \eqref{discrete-sum-to1} is deferred to Section \ref{sec:discrete-dist}.

\subsection{Functions $\psi(\lambda^{[m+1]})$ and $\psi(0\cup\lambda^{[m]})$}

We begin by studying the exponents $\psi(\lambda^{[m+1]})$ and $\psi(0\cup\lambda^{[m]})$ that appear within \eqref{kernel}. Under the set of assumptions \eqref{coord-scal}--\eqref{real-seq}, the coordinates $\{\ell_i^{[m]}\}_{1 \leq i \leq nm}$ and $\{\ell_j^{[m+1]}\}_{1 \leq j \leq n(m+1)}$ are strictly increasing. This makes the computation of $\psi(\lambda^{[m+1]})$ and $\psi(0\cup\lambda^{[m]})$ particularly simple; one easily sees that
\begin{align}
\label{psi-inv}
2\psi\left(\lambda^{[m+1]}\right)
=
{\rm inv}\left(c^{[m+1]}\right),
\qquad
2\psi\left(0\cup\lambda^{[m]}\right)
=
{\rm inv}\left((1,...,n)\cup c^{[m]}\right)
=
{\rm inv}\left(c^{[m]}\right)+m \binom{n}{2},
\end{align}
where $(1,...,n)\cup c^{[m]}$ means concatenation of the two participating vectors.

\subsection{Factor $\mathbb{G}_{\lambda^{[m+1]}/0\cup\lambda^{[m]}}(1)$}
\label{ssec:analysis1}

Next, we analyse the quantity $\mathbb{G}_{\lambda^{[m+1]}/0\cup\lambda^{[m]}}(1)$ within \eqref{kernel}. It is given by the one-row partition function
\begin{align}
\label{G-one-row}
\mathbb{G}_{\lambda^{[m+1]}/0\cup\lambda^{[m]}}(1)
=
\tikz{1.2}{
\draw[lgray,line width=4pt,->] (0.5,0) -- (5.5,0);
\foreach\x in {1,...,5}{
\draw[lgray,line width=4pt,->] (\x,-0.5) -- (\x,0.5);
}
\node[left] at (0.5,0) {\fs $\bm{e}_0$};\node[right] at (5.5,0) {\fs $\bm{e}_0$};
%\node[below] at (5,-0.5) {\fs $\cdots$};\node[above] at (5,0.5) {\fs $\cdots$};
%\node[below] at (4,-0.5) {\fs $\cdots$};\node[above] at (4,0.5) {\fs $\cdots$};
%\node[below] at (3,-0.5) {\fs $\cdots$};\node[above] at (3,0.5) {\fs $\cdots$};
%\node[below] at (2,-0.5) {\fs $\cdots$};\node[above] at (2,0.5) {\fs $\cdots$};
\node[below] at (1,-0.5) {\fs $\bm{e}_{[1,n]}$};\node[above] at (1,0.5) {\fs $\bm{e}_0$};
\node at (3.5, -0.7) {$\underbrace{\hspace{4cm}}$};
\node[below] at (3.5,-0.7) {$\lambda^{[m]}$};
\node at (3.5, 0.7) {$\overbrace{\hspace{4cm}}$};
\node[above] at (3.5,0.7) {$\lambda^{[m+1]}$};
}
\end{align}
where within the area marked $\lambda^{[k]}$, $k \in \{m,m+1\}$, the vector 
$\bm{e}_{c^{[k]}_i}$ is present at coordinate $\ell^{[k]}_i$ (in other words, a path of colour $c^{[k]}_i$ is present at position $\ell^{[k]}_i$), for all $1 \leq i \leq kn$. Following \eqref{llt-weights}, we have assumed the vertex weights
\begin{align}
\label{llt-wts}
\tikz{1.2}{
\draw[lgray,line width=4pt,->] (-0.5,0) -- (0.5,0);
\draw[lgray,line width=4pt,->] (0,-0.5) -- (0,0.5);
\node[left] at (-0.5,0) {\fs $\bm{B}$};
\node[right] at (0.5,0) {\fs $\bm{D}$};
\node[below] at (0,-0.5) {\fs $\bm{A}$};
\node[above] at (0,0.5) {\fs $\bm{C}$};
}
=
\bm{1}_{\bm{C}+\bm{D} \in \{0,1\}^n}
\cdot
q^{\varphi(\D,\D)+\varphi(\D,\C)},
\qquad
\A,\B,\C,\D \in \{0,1\}^n,
\end{align}
where the function $\varphi$ is as defined in \eqref{phi-def}. Let us study each of the factors appearing in \eqref{llt-wts} individually. First, we note that the indicator function $\bm{1}_{\bm{C}+\bm{D} \in \{0,1\}^n}$ prevents two paths of the same colour from traversing a vertex. Second, the factor $q^{\varphi(\D,\D)}$ assigns one power of $q$ for every pair of colours which pass through edge $\D$ of a vertex; there are $\binom{|\D|}{2}$ such pairs. Finally, the factor $q^{\varphi(\D,\C)}$ assigns one power of $q$ to each pair of colours $(i,j)$ passing through edges $(\D,\C)$, respectively, with $i<j$.

Now we examine the contribution of each of the factors in the weights \eqref{llt-wts}, when they are multiplied together to form the one-row partition function \eqref{G-one-row}. Multiplying all indicator functions $\bm{1}_{\bm{C}+\bm{D} \in \{0,1\}^n}$ yields the the property that paths of the same colour do not intersect; at the level of the coloured compositions $\lambda^{[m]}$ and 
$\lambda^{[m+1]}$, this translates into the condition that
\begin{align*}
\lambda^{[m+1](i)}_j < \lambda^{[m](i)}_j < \lambda^{[m+1](i)}_{j+1},
\qquad\forall\
1 \leq i \leq n,
\quad
1 \leq j \leq m,
\end{align*}
which we denote simply by writing $c^{[m]} \prec c^{[m+1]}$.

Multiplying all factors $q^{\varphi(\D,\D)}$ requires us to compute the total number of paths $d_i$ going through the $i$-th horizontal edge of the partition function \eqref{G-one-row}, for all $i \geq 1$. The total contribution from these factors is then 
\begin{align*}
\prod_{i \geq 1} 
q^{\binom{d_i}{2}}
=
\prod_{j=1}^{n}
q^{\binom{j}{2} p_j},
\end{align*}
where $p_j$ counts the number of horizontal edges in \eqref{G-one-row} that are occupied by $j$ paths. It is clear that the set $\{p_j\}_{1 \leq j \leq n}$ depends only on the values of the coordinates $\{\ell^{[m]}_i\}_{1 \leq i \leq nm}$ and $\{\ell^{[m+1]}_i\}_{1 \leq i \leq n(m+1)}$, and not on the colour sequences $\{c^{[m]}_i\}_{1 \leq i \leq nm}$ and $\{c^{[m+1]}_i\}_{1 \leq i \leq n(m+1)}$. In view of the interlacing \eqref{interlace-disc} of the coordinates $\{\ell^{[m]}_i\}_{1 \leq i \leq nm}$, $\{\ell^{[m+1]}_i\}_{1 \leq i \leq n(m+1)}$, we may routinely compute $p_j$ for all $1 \leq j \leq n$. We find that, {\it cf.} Figure \ref{fig:gue},
\begin{align*}
p_n
&=
\sum_{k=1}^{m+1}
\ell^{[m+1]}_k
-
\sum_{k=1}^{m}
\ell^{[m]}_k,
\\
p_{n-j}
&=
\sum_{k=1}^{m+1}
\left(
\ell^{[m+1]}_{j(m+1)+k}
-
\ell^{[m+1]}_{(j-1)(m+1)+k}
\right)
-
\sum_{k=1}^{m}
\left(
\ell^{[m]}_{jm+k}
-
\ell^{[m]}_{(j-1)m+k}
\right),
\qquad\forall\ 
j \in [1,n-1].
\end{align*}
We then have
\begin{align*}
\sum_{j=1}^{n}
\binom{j}{2}
p_j
&=
\sum_{j=1}^{n}
\left[
\binom{n-j+1}{2}
-
\binom{n-j}{2}
\right]
\sum_{k=1}^{m+1}
\ell^{[m+1]}_{(j-1)(m+1)+k}
-
\sum_{j=1}^{n}
\left[
\binom{n-j+1}{2}
-
\binom{n-j}{2}
\right]
\sum_{k=1}^{m}
\ell^{[m]}_{(j-1)m+k}
\\
&=
\sum_{j=1}^{n}
(n-j)
\sum_{k=1}^{m+1}
\ell^{[m+1]}_{(j-1)(m+1)+k}
-
\sum_{j=1}^{n}
(n-j)
\sum_{k=1}^{m}
\ell^{[m]}_{(j-1)m+k}
\end{align*}
as the total exponent of $q$ coming from factors of the form $q^{\varphi(\D,\D)}$.

Finally we need to examine the contribution from all factors $q^{\varphi(\D,\C)}$, when multiplying the weights of all vertices in the row \eqref{G-one-row}. In direct contrast to the factors $q^{\varphi(\D,\D)}$, this contribution only depends on the colour sequences 
$\{c^{[m]}_i\}_{1 \leq i \leq nm}$ and $\{c^{[m+1]}_i\}_{1 \leq i \leq n(m+1)}$, and not on the coordinates $\{\ell^{[m]}_i\}_{1 \leq i \leq nm}$ and $\{\ell^{[m+1]}_i\}_{1 \leq i \leq n(m+1)}$. Rather than attempting to write down an explicit formula for this contribution, we express it in terms of the following diagram\footnote{This diagram is not a partition function in the traditional sense, however it turns out to be quite expedient for our subsequent needs.}:
\begin{multline}
\label{ups}
\Upsilon\left(c^{[m]};c^{[m+1]}\right)
=
\\
\tikz{1.3}{
\draw[lgray,line width=4pt,->] (0.5,0) -- (12,0);
\foreach\x in {1,2,3,5,6,7,9,10,11}{
\draw[lgray,line width=1.5pt,->] (\x,0) -- (\x,0.5);
}
\foreach\x in {1.5,2.5,5.5,6.5,9.5,10.5}{
\draw[lgray,line width=1.5pt,->] (\x,-0.5) -- (\x,0);
}
\node[left] at (0.5,0) {$[1,n]$}; \node[right] at (12,0) {$\emptyset$};
\node[below] at (1.5,-0.5) {$c^{[m]}_1$}; \node[above] at (1,0.5) {$c^{[m+1]}_1$};
\node[below] at (2.5,-0.5) {$c^{[m]}_m$}; \node[above] at (3,0.5) {$c^{[m+1]}_{m+1}$};
\node[below] at (2,-0.7) {$\cdots$}; \node[above] at (2,0.5) {$\cdots$};
\node[below] at (6,-0.7) {$\cdots$}; \node[above] at (6,0.5) {$\cdots$};
\node[below] at (9.5,-0.5) {$c^{[m]}_{(n-1)m+1}$}; \node[above] at (9,0.5) {$c^{[m+1]}_{(n-1)(m+1)+1}$};
\node[below] at (10.5,-0.5) {$c^{[m]}_{nm}$}; \node[above] at (11,0.5) {$c^{[m+1]}_{n(m+1)}$};
\node[below] at (10,-0.6) {$\cdots$}; \node[above] at (10.2,0.5) {$\cdots$};
}
\end{multline}
In this diagram, one of each of the colours $\{1,\dots,n\}$ enters via the leftmost horizontal edge; there is also a single colour $c^{[m]}_i$ entering via the $i$-th vertical edge along the base. No colours exit via the rightmost horizontal edge; a single colour $c^{[m+1]}_i$ exits via the $i$-th vertical edge along the top. Note that the diagram is also arranged to denote $n$ individual bundles of interlacing colours, with flow of colours possible along the horizontal line that connects bundles.

The function $\Upsilon\left(c^{[m]};c^{[m+1]}\right)$ is a pure power of $q$. We compute it by tracing all colours in the diagram \eqref{ups} from their starting to finishing location, and assigning a power of $q$ to each of the following events:
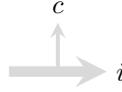
\begin{align*}
\tikz{1.3}{
\draw[lgray,line width=4pt,->] (0.5,0) -- (1.5,0);
\draw[lgray,line width=1.5pt,->] (1,0) -- (1,0.5);
\node[above] at (1,0.5) {$c$};
\node[right] at (1.5,0) {$i$};
}
\end{align*}
which denotes a path of colour $i$ passing underneath a colour $c$, with $i<c$.
 
Bringing all contributions together, we have
\begin{align}
\label{one-row-formula}
\mathbb{G}_{\lambda^{[m+1]}/0\cup\lambda^{[m]}}(1)
=
\bm{1}_{c^{[m]} \prec c^{[m+1]}}
\cdot
\Upsilon\left(c^{[m]};c^{[m+1]}\right)
\cdot
\prod_{j=1}^{n}
\left(
\dfrac{\displaystyle{\prod_{k=1}^{m+1}} q^{\ell^{[m+1]}_{(j-1)(m+1)+k}}}
{\displaystyle{\prod_{k=1}^{m}} q^{\ell^{[m]}_{(j-1)m+k}}}
\right)^{n-j}
\end{align}
We draw attention to the fact that dependence on the colour sequences $\{c^{[m]}_i\}_{1 \leq i \leq nm}$, $\{c^{[m+1]}_i\}_{1 \leq i \leq n(m+1)}$ and coordinates $\{\ell^{[m]}_i\}_{1 \leq i \leq nm}$, $\{\ell^{[m+1]}_i\}_{1 \leq i \leq n(m+1)}$ is completely separated in \eqref{one-row-formula}; further, $\mathbb{G}_{\lambda^{[m+1]}/0\cup\lambda^{[m]}}(1)$ is an analytic function of the coordinates.

\subsection{Factors $\mathbb{G}_{\lambda^{[m]}-1}({\rm Pl}_t)$ and 
$\mathbb{G}_{\lambda^{[m+1]}}({\rm Pl}_t)$}
\label{ssec:analysis2}

We now proceed to the analysis of the factors $\mathbb{G}_{\lambda^{[m]}-1}({\rm Pl}_t)$ and 
$\mathbb{G}_{\lambda^{[m+1]}}({\rm Pl}_t)$, in the denominator and numerator of \eqref{kernel}, under the replacements \eqref{coord-scal}. This computation is essentially the same for each of these factors, up to some straightforward relabelling of parameters, so we will focus on the analysis of $\mathbb{G}_{\lambda^{[m]}-1}({\rm Pl}_t)$ first.

\subsubsection{Recasting the integrand}

Our starting point is \eqref{LLTplanch} with $\nu = \Delta \in \mathcal{S}_{m^n}$ and 
$\mu = \lambda^{[m]}-1 \in \mathcal{S}_{m^n}$:
\begin{multline}
\label{LLTplanch2}
\mathbb{G}_{\lambda^{[m]}-1}({\rm Pl}_t)
=
\frac{q^{nm(nm+1)/2}}{(q-1)^{nm}}
\cdot
\left( \frac{1}{2\pi{\tt i}} \right)^{nm}
\oint_{C_1}
\frac{dy_1}{y_1}
\cdots 
\oint_{C_{nm}}
\frac{dy_{nm}}{y_{nm}}
\\
\times
\prod_{1 \leq i<j \leq nm}
\left(
\frac{y_j-y_i}{y_j-q y_i}
\right)
f_{\breve\lambda^{[m]}-1}(1^{nm};y_1^{-1},\dots,y_{nm}^{-1})
g_{\Delta}(m^n;y_1,\dots,y_{nm})
\prod_{j=1}^{nm}
e^{t y_j}.
\end{multline}
In the first step, we use the action \eqref{T-f} of the Hecke generators to express the function $f_{\breve\lambda^{[m]}-1}$ in terms of $f_{\ell^{[m]}-1}$, noting that $\ell^{[m]}$ is just obtained by sorting the parts of $\breve\lambda^{[m]}$ in increasing order. We have
\begin{multline}
\label{LLTplanch3}
\mathbb{G}_{\lambda^{[m]}-1}({\rm Pl}_t)
=
\frac{q^{nm(nm+1)/2}}{(q-1)^{nm}}
\cdot
\left( \frac{1}{2\pi{\tt i}} \right)^{nm}
\oint_{C_1}
\frac{dy_1}{y_1}
\cdots 
\oint_{C_{nm}}
\frac{dy_{nm}}{y_{nm}}
\\
\times
\prod_{1 \leq i<j \leq nm}
\left(
\frac{y_j-y_i}{y_j-q y_i}
\right)
(T_{\sigma} \cdot f_{\ell^{[m]}-1})(1^{nm};y_1^{-1},\dots,y_{nm}^{-1})
g_{\Delta}(m^n;y_1,\dots,y_{nm})
\prod_{j=1}^{nm}
e^{t y_j},
\end{multline}
where we have denoted $T_{\sigma} = T_{a_1} \dots T_{a_p}$ with $T_a$ given by \eqref{hecke-poly}, and where $(a_1,\dots,a_p) \in [1,nm)^p$ is a minimal-length word such that
%\footnote{There may be many minimal-length words that achieve the reordering \eqref{reorder}; any such choice leads to the same result for $T_{\sigma} \circ f_{\ell^{[m]}-1}$, in view of the braid-type relation in \eqref{hecke1}.} 
%
\begin{align}
\label{reorder}
\mathfrak{s}_{a_1} \cdots \mathfrak{s}_{a_p} \cdot \ell^{[m]} = \breve\lambda^{[m]}.
\end{align}
Using the property \eqref{adjoint} of Hecke generators, and the fact that the product 
$\prod_{j=1}^{nm} e^{t y_j}$ is symmetric with respect to $(y_1,\dots,y_{nm})$, we may recast this as
\begin{multline}
\label{LLTplanch4}
\mathbb{G}_{\lambda^{[m]}-1}({\rm Pl}_t)
=
\frac{q^{nm(nm+1)/2}}{(q-1)^{nm}}
\cdot
\left( \frac{1}{2\pi{\tt i}} \right)^{nm}
\oint_{C_1}
\frac{dy_1}{y_1}
\cdots 
\oint_{C_{nm}}
\frac{dy_{nm}}{y_{nm}}
\\
\times
\prod_{1 \leq i<j \leq nm}
\left(
\frac{y_j-y_i}{y_j-q y_i}
\right)
f_{\ell^{[m]}-1}(1^{nm};y_1^{-1},\dots,y_{nm}^{-1})
(\tilde{T}_{\sigma} \cdot g_{\Delta})(m^n;y_1,\dots,y_{nm})
\prod_{j=1}^{nm}
e^{t y_j},
\end{multline}
where we have denoted $\tilde{T}_{\sigma} = \tilde{T}_{a_p} \dots \tilde{T}_{a_1}$ with $\tilde{T}_a$ given by \eqref{hecke-tilde}, and where the word $(a_1,\dots,a_p)$ is specified as in \eqref{reorder}. 

Finally, we note that if the coordinates $\ell^{[m]}$ may be reordered to yield $\breve{\lambda}^{[m]}$ as in \eqref{reorder}, it also follows that the corresponding colour sequence $c^{[m]}$ reorders according to the rule
\begin{align*}
\mathfrak{s}_{a_1} \cdots \mathfrak{s}_{a_p} \cdot c^{[m]}
=
(1^m,2^m,\dots,n^m),
\end{align*}
or equivalently, 
$\mathfrak{s}_{a_p} \cdots \mathfrak{s}_{a_1} \cdot (1^m,2^m,\dots,n^m) = c^{[m]}$. Using this relation in \eqref{LLTplanch4}, together with the action \eqref{invT-g} of Hecke generators on the function $g_{\Delta}(m^n;y_1,\dots,y_{nm})$, we recover the formula
\begin{multline}
\label{LLTplanch5}
\mathbb{G}_{\lambda^{[m]}-1}({\rm Pl}_t)
=
\frac{q^{nm(nm+1)/2}}{(q-1)^{nm}}
\cdot
\left( \frac{1}{2\pi{\tt i}} \right)^{nm}
\oint_{C_1}
\frac{dy_1}{y_1}
\cdots 
\oint_{C_{nm}}
\frac{dy_{nm}}{y_{nm}}
\\
\times
\prod_{1 \leq i<j \leq nm}
\left(
\frac{y_j-y_i}{y_j-q y_i}
\right)
\prod_{i=1}^{nm}
y_i^{-\ell^{[m]}_i+1}
g_{\Delta}^{c^{[m]}}(m^n;y_1,\dots,y_{nm})
\prod_{j=1}^{nm} e^{ty_j},
\end{multline}
where $g_{\Delta}^{c^{[m]}}(m^n;y_1,\dots,y_{nm})$ denotes a permuted-boundary function of the form \eqref{generic-g-sigma}, and where we have also used the fact that $f_{\ell^{[m]}-1}(1^{nm};y_1^{-1},\dots,y_{nm}^{-1})$ factorizes as in \eqref{factorize} with $s=0$, since $\ell^{[m]}$ is increasing. The formula \eqref{LLTplanch5} explicitly separates the coordinates $\ell^{[m]}$ and the colour sequence $c^{[m]}$; in view of the analytic dependence on the former, we now use it to carry out steepest descent asymptotics.

\subsubsection{$t \rightarrow \infty$ asymptotics via steepest descent}

In this section we compute the $t \rightarrow \infty$ asymptotics of the quantities
\begin{align}
\label{H-funct}
\tilde{H}_{\lambda^{[m]}}(t)
:=
\prod_{j=1}^{n}
\prod_{k=1}^{m}
q^{(n-j)\ell^{[m]}_{(j-1)m+k}}
\cdot
\mathbb{G}_{\lambda^{[m]}-1}({\rm Pl}_t),
\\
\label{H-funct2}
H_{\lambda^{[m+1]}}(t)
:=
\prod_{j=1}^{n}
\prod_{k=1}^{m+1}
q^{(n-j)\ell^{[m+1]}_{(j-1)(m+1)+k}}
\cdot
\mathbb{G}_{\lambda^{[m+1]}}({\rm Pl}_t),
\end{align}
under the assumption that the coordinates $\{\ell^{[m]}_i\}_{1 \leq i \leq nm}$ and $\{\ell^{[m+1]}_i\}_{1 \leq i \leq n(m+1)}$ scale as \eqref{coord-scal}. Note that, by virtue of \eqref{psi-inv} and the expression \eqref{one-row-formula}, the Markov kernel \eqref{kernel} may be expressed as
\begin{multline}
\label{kernel-reexpress}
\mathbb{P}_{t,1}\left(0\cup\lambda^{[m]} \rightarrow \lambda^{[m+1]}\right)
\\
=
\bm{1}_{c^{[m]} \prec c^{[m+1]}}
\cdot
q^{{\rm inv}(c^{[m]})-{\rm inv}(c^{[m+1]})+m \binom{n}{2}}
\Upsilon\left( c^{[m]}; c^{[m+1]} \right)
\dfrac{H_{\lambda^{[m+1]}}(t)}{\tilde{H}_{\lambda^{[m]}}(t)}
\exp\left( -\frac{1-q^n}{1-q}t \right)
\end{multline}
where the colour sequences $c^{[m]}$, $c^{[m+1]}$ are independent of $t$; this means that the $t \rightarrow \infty$ asymptotics of our Markov kernel is indeed recovered by analysis of \eqref{H-funct} and \eqref{H-funct2}.

We lighten our notation by writing $\ell^{[m]}_i \equiv \ell_i$ and $x^{[m]}_i \equiv x_i$ for all $1 \leq i \leq nm$. Distributing the $q$-dependent prefactor in \eqref{H-funct} within the integral \eqref{LLTplanch5}, we have
\begin{multline}
\label{LLTplanch6}
\tilde{H}_{\lambda^{[m]}}(t)
=
\frac{q^{nm(nm+1)/2}}{(q-1)^{nm}}
\cdot
\left( \frac{1}{2\pi{\tt i}} \right)^{nm}
\oint_{C_1}
dy_1
%\frac{dy_1}{y_1}
\cdots 
\oint_{C_{nm}}
dy_{nm}
%\frac{dy_{nm}}{y_{nm}}
\\
\times
\prod_{i=1}^{nm}
\left(\frac{Q_i}{y_i} \right)^{\ell_i}
e^{ty_i}
\prod_{1 \leq i<j \leq nm}
\left(
\frac{y_j-y_i}{y_j-q y_i}
\right)
g_{\Delta}^{c^{[m]}}(m^n;y_1,\dots,y_{nm}),
\end{multline}
where we have defined the vector $\vec{Q} \in \mathbb{C}^{nm}$ by
\begin{align}
\label{vecQ}
\vec{Q}
=
(Q_1,\dots,Q_{nm})
=
\underbrace{(q^{n-1},\dots,q^{n-1})}_{m\ {\rm times}}
\cup
\cdots
\cup
\underbrace{(q,\dots,q)}_{m\ {\rm times}}
\cup
\underbrace{(1,\dots,1)}_{m\ {\rm times}}.
\end{align}
Using the formula \eqref{coord-scal} and the notation \eqref{vecQ}, the coordinates $\ell_i$ are written as $\ell_i = Q_i t + (Q_i t)^{\frac{1}{2}} x_i$ for all $1 \leq i \leq nm$. Making use of this, the univariate factors in the integrand of \eqref{LLTplanch6} read
\begin{align}
\label{univar}
\left( \frac{Q_i}{y_i} \right)^{\ell_i} e^{ty_i}
=
\exp\left[t y_i-\ell_i \log y_i + \ell_i \log Q_i \right]
=
\exp\left[t(y_i-Q_i \log y_i + Q_i \log Q_i) + O(t^{1/2}) \right],
\quad
\text{as}\ \
t \rightarrow \infty.
\end{align}
The $t \rightarrow \infty$ behaviour of \eqref{LLTplanch6} may now be recovered from steepest descent analysis applied to each of the $nm$ integrals. Neglecting for the moment the $O(t^{1/2})$ term above (which gives a sub-leading contribution to the $t \rightarrow \infty$ behaviour), we evaluate the critical point\footnote{The critical point is the value where the first derivative with respect to $y_i$ vanishes.} of the function $y_i-Q_i \log y_i + Q_i \log Q_i$, which is found to be $y_i = Q_i$. Computing the corresponding Taylor series about this point, we have that
\begin{align*}
\left( \frac{Q_i}{y_i} \right)^{\ell_i} e^{ty_i}
=
\exp\left[t \left( Q_i + \frac{(y_i-Q_i)^2}{2Q_i} + O(y_i-Q_i)^3 \right) + O(t^{1/2}) \right],
\quad
\text{as}\ \
t \rightarrow \infty,
\end{align*}
in a neighbourhood of the point $y_i=Q_i$. Following standard steepest descent analysis, the dominant contribution to the $t \rightarrow \infty$ asymptotics of the integral \eqref{LLTplanch6} is obtained by deforming each contour $C_i$ to pass through $Q_i$\footnote{We require that each integration contour $C_i$ may be freely deformed to a contour passing through the corresponding critical point $Q_i$, along which the real part of the exponent decreases as one travels away from $Q_i$. One choice that meets this requirement is to take the $C_i$ to be concentric circles of radii $Q_i$, for all $1 \leq i \leq nm$.}, and reducing the resulting contour integrals to line integrals over small segments $D_i \subset C_i$ travelling through $Q_i$ and traversed in the direction where the function $\frac{(y_i-Q_i)^2}{2Q_i}$ has zero imaginary part. Accordingly\footnote{We will be slightly informal here, omitting the proofs of the tail estimates justifying the below approximations. The latter are fairly standard, and have already been applied numerous times in the literature.}, we may write
\begin{multline}
\label{LLTplanch7}
\tilde{H}_{\lambda^{[m]}}(t)
\sim
\frac{q^{nm(nm+1)/2}}{(q-1)^{nm}}
\cdot
\left( \frac{1}{2\pi{\tt i}} \right)^{nm}
\int_{Q_1-{\tt i}\epsilon}^{Q_1+{\tt i}\epsilon}
dy_1
\cdots 
\int_{Q_{nm}-{\tt i}\epsilon}^{Q_{nm}+{\tt i}\epsilon}
dy_{nm}
\\
\prod_{i=1}^{nm}
\left(\frac{Q_i}{y_i} \right)^{\ell_i}
e^{ty_i}
\prod_{1 \leq i<j \leq nm}
\left(
\frac{y_j-y_i}{y_j-q y_i}
\right)
g_{\Delta}^{c^{[m]}}(m^n;y_1,\dots,y_{nm}),
\quad
\text{as}\ \ t \rightarrow \infty,
\end{multline}
where $\epsilon$ is a small positive real number. Now switching to the local variables $y_i = Q_i+z_i t^{-\frac{1}{2}}$, the univariate factors in \eqref{LLTplanch7} become
\begin{align*}
\left.
\left( \frac{Q_i}{y_i} \right)^{\ell_i} e^{ty_i}
\right|_{y_i = Q_i+z_i t^{-\frac{1}{2}}}
=
\exp\left[
t(Q_i+z_i t^{-\frac{1}{2}})
-
(Q_i t + (Q_i t)^{\frac{1}{2}} x_i)
\log(Q_i+z_i t^{-\frac{1}{2}})
+
(Q_i t + (Q_i t)^{\frac{1}{2}} x_i) \log Q_i
\right],
\end{align*}
and using the fact that
\begin{align*}
\log(Q+\epsilon)
=
\log(Q) + \frac{\epsilon}{Q} - \frac{1}{2}\frac{\epsilon^2}{Q^2} + O(\epsilon^3),
\quad
\text{as}
\ \ 
\epsilon \rightarrow 0,
\end{align*}
we have
\begin{align*}
&
\left.
\left( \frac{Q_i}{y_i} \right)^{\ell_i} e^{ty_i}
\right|_{y_i = Q_i+z_i t^{-\frac{1}{2}}}
\\
&=
\exp\left[
t(Q_i+z_i t^{-\frac{1}{2}})
-
(Q_i t + (Q_i t)^{\frac{1}{2}} x_i)
\left(\log Q_i
+
\frac{z_i t^{-\frac{1}{2}}}{Q_i}
-
\frac{1}{2} \frac{z_i^2 t^{-1}}{Q_i^2} 
+ 
O(t^{-\frac{3}{2}})\right)
+
(Q_i t + (Q_i t)^{\frac{1}{2}} x_i) \log Q_i
\right]
\\
&=
\exp\left[
Q_i t - \frac{x_i z_i}{Q_i^{\frac{1}{2}}} + \frac{z_i^2}{2Q_i} + O(t^{-\frac{1}{2}})
\right],
\quad
\text{as}
\ \
t \rightarrow \infty.
\end{align*}
Turning to other terms in the integral \eqref{LLTplanch7}, we have
\begin{align}
\label{q-vand}
\left.
\frac{y_j-y_i}{y_j-q y_i}
\right|_{\vec{y} = \vec{Q}+\vec{z} t^{-\frac{1}{2}}}
=
\frac{Q_j-Q_i+z_j t^{-1/2}-z_i t^{-1/2}}{Q_j-qQ_i+z_j t^{-1/2} -qz_i t^{-1/2}},
\qquad
1 \leq i< j \leq nm.
\end{align} 
Using the fact that $Q_i = q^{n-\lceil i/m \rceil}$, we see that $Q_j - q Q_i$ in the denominator of \eqref{q-vand} is always nonzero; on the other hand, $Q_j - Q_i$ in the numerator vanishes whenever $\lceil i/m \rceil = \lceil j/m \rceil$ (and is nonzero otherwise). The $t \rightarrow \infty$ behaviour of \eqref{q-vand} then splits into two cases:
\begin{align*}
\left.
\frac{y_j-y_i}{y_j-q y_i}
\right|_{\vec{y} = \vec{Q}+\vec{z} t^{-\frac{1}{2}}}
\sim
\frac{1}{q^{n-\lceil j/m \rceil} - q^{n-\lceil i/m \rceil +1}}
\times
\left\{
\begin{array}{ll}
t^{-1/2} (z_j-z_i),
\qquad
&
\lceil i/m \rceil = \lceil j/m \rceil,
\\
\\
q^{n-\lceil j/m \rceil} - q^{n-\lceil i/m \rceil},
\qquad
&
\lceil i/m \rceil \not= \lceil j/m \rceil,
\end{array}
\right.
\end{align*}
and multiplying these factors over all indices $1 \leq i<j \leq nm$, we have
\begin{multline*}
\left.
\prod_{1 \leq i<j \leq nm}
\left( \frac{y_j-y_i}{y_j-q y_i} \right)
\right|_{\vec{y} = \vec{Q}+\vec{z} t^{-\frac{1}{2}}}
\sim
t^{-\frac{n}{2} \binom{m}{2}}
(1-q)^{-n \binom{m}{2}}
\prod_{j=1}^{n} q^{(j-n) \binom{m}{2}}
\\
\times
\prod_{i=0}^{n-1}
\prod_{1 \leq j<k \leq m}
(z_{im+k}-z_{im+j})
\prod_{1 \leq i<j \leq n}
\left( \frac{q^{n-j}-q^{n-i}}{q^{n-j}-q^{n-i+1}} \right)^{m^2}.
\end{multline*}
After accounting for telescopic cancellations in the final product of this expression, it may be written as 
\begin{align*}
\prod_{1 \leq i<j \leq n}
\left( \frac{q^{n-j}-q^{n-i}}{q^{n-j}-q^{n-i+1}} \right)^{m^2}
=
\left[ \frac{(1-q)^{n}}{(q;q)_n} \right]^{m^2}
\end{align*}
and accordingly we have that
\begin{multline*}
\left.
\prod_{1 \leq i<j \leq nm}
\left( \frac{y_j-y_i}{y_j-q y_i} \right)
\right|_{\vec{y} = \vec{Q}+\vec{z} t^{-\frac{1}{2}}}
\sim
t^{-\frac{n}{2} \binom{m}{2}}
(1-q)^{-n \binom{m}{2}}
q^{-\binom{n}{2} \binom{m}{2}}
\left[ \frac{(1-q)^{n}}{(q;q)_n} \right]^{m^2}
\prod_{i=0}^{n-1}
\prod_{1 \leq j<k \leq m}
(z_{im+k}-z_{im+j}),
\end{multline*}
as $t \rightarrow \infty$. Under the change to local variables, the remaining piece of the integral \eqref{LLTplanch7} (which is polynomial in the variables $y_1,\dots,y_{nm}$) becomes
\begin{align*}
\left. 
g_{\Delta}^{c^{[m]}}(m^n;y_1,\dots,y_{nm})
\right|_{\vec{y} = \vec{Q}+\vec{z} t^{-\frac{1}{2}}}
\sim
g_{\Delta}^{c^{[m]}}(m^n;\vec{Q}),
\qquad
\text{as}\ \ t\rightarrow \infty.
\end{align*}
Combining everything (including a factor of $t^{-nm/2}$ for the change of integration variables), we read off the $t \rightarrow \infty$ asymptotic behaviour of $\tilde{H}_{\lambda^{[m]}}(t)$:
\begin{multline}
\label{LLTplanch8}
\tilde{H}_{\lambda^{[m]}}(t)
\sim
t^{-nm/2}
\exp\left[ \frac{1-q^n}{1-q} mt \right]
t^{-\frac{n}{2} \binom{m}{2}}
(1-q)^{-n \binom{m}{2}}
q^{-\binom{n}{2} \binom{m}{2}}
\left[ \frac{(1-q)^{n}}{(q;q)_n} \right]^{m^2}
g_{\Delta}^{c^{[m]}}(m^n;\vec{Q})
\\
\times
\frac{q^{nm(nm+1)/2}}{(q-1)^{nm}}
\cdot
\left( \frac{1}{2\pi{\tt i}} \right)^{nm}
\int_{-{\tt i}\infty}^{{\tt i}\infty}
dz_1
\cdots 
\int_{-{\tt i}\infty}^{{\tt i}\infty}
dz_{nm}
\prod_{i=1}^{nm}
\exp\left[- \frac{x_i z_i}{Q_i^{1/2}} + \frac{z_i^2}{2 Q_i} \right]
\prod_{i=0}^{n-1}
\prod_{1 \leq j<k \leq m}
(z_{im+k}-z_{im+j}).
\end{multline}
This simplifies to yield
\begin{multline}
\label{LLTplanch9}
\tilde{H}_{\lambda^{[m]}}(t)
\sim
(-1)^{nm}
t^{-\frac{n}{2}\binom{m+1}{2}}
q^{\binom{nm+1}{2}-\frac{1}{2}\binom{m+1}{2} \binom{n}{2}+m\binom{n}{2}}
\exp\left[ \frac{1-q^n}{1-q} mt \right]
\frac{(1-q)^{n\binom{m}{2}}}{(q;q)_n^{m^2}}
g_{\Delta}^{c^{[m]}}(m^n;\vec{Q})
\\
\times
\left( \frac{1}{2\pi{\tt i}} \right)^{nm}
\int_{-{\tt i}\infty}^{{\tt i}\infty}
dz_1
\cdots 
\int_{-{\tt i}\infty}^{{\tt i}\infty}
dz_{nm}
\prod_{i=1}^{nm}
\exp\left[-x_i z_i + \frac{z_i^2}{2} \right]
\prod_{i=0}^{n-1}
\prod_{1 \leq j<k \leq m}
(z_{im+k}-z_{im+j}),
\end{multline}
where we have rescaled the integration variables $z_i \mapsto Q_i^{1/2} z_i$ to obtain the final formula.

Repeating these steps, one easily finds that
\begin{align}
\label{m-to-m+1}
H_{\lambda^{[m+1]}}(t)
=
\left(
q^{-m \binom{n}{2}} 
\tilde{H}_{\lambda^{[m]}}(t)
\right)_{m \mapsto m+1},
\end{align}
which is to be interpreted as taking the right hand side of \eqref{LLTplanch9}, modulo division by $q^{m \binom{n}{2}}$, and replacing all instances of $m$ by $m+1$ in the obvious way. 

The removal of the factor $q^{m \binom{n}{2}}$ requires the following justification. From \eqref{H-funct} and \eqref{H-funct2}, we see that $\tilde{H}_{\lambda^{[m]}}(t)$ depends on the function $\mathbb{G}_{\lambda^{[m]}-1}({\rm Pl}_t)$ rather than $\mathbb{G}_{\lambda^{[m]}}({\rm Pl}_t)$, the latter being the desired quantity that leads to $H_{\lambda^{[m+1]}}(t)$ after the $m \mapsto m+1$ relabelling. Consulting the integral formula \eqref{LLTplanch5}, we see that the only difference between $\mathbb{G}_{\lambda^{[m]}-1}({\rm Pl}_t)$ and $\mathbb{G}_{\lambda^{[m]}}({\rm Pl}_t)$ is that the integrand used for the former contains an extra factor of $\prod_{i=1}^{nm} y_i$ compared with that of the latter. Carrying through steepest descent analysis of $\mathbb{G}_{\lambda^{[m]}}({\rm Pl}_t)$ therefore results in an overall factor $\prod_{i=1}^{nm} Q_i = q^{m \binom{n}{2}}$ less compared with the calculations above, which is the reason that we divide out this factor in \eqref{m-to-m+1}.

\subsubsection{Factorization into GUE corners}

Up to the multiplicative terms in the first line, equation \eqref{LLTplanch9} reveals the factorization of our starting integral \eqref{LLTplanch6} into $n$ identical $m$-dimensional integrals of the form
\begin{align}
\label{I-int}
I(x_1,\dots,x_m)
&=
\left(\frac{1}{2\pi{\tt i}}\right)^{m}
\int_{{\tt i} \cdot \mathbb{R}}
dz_1
\cdots
\int_{{\tt i} \cdot \mathbb{R}}
dz_m
\prod_{1 \leq i < j \leq m}
(z_j-z_i)
\prod_{i=1}^{m}
e^{-x_i z_i + \frac{1}{2} z_i^2},
\\
\nonumber
&=
\prod_{i=1}^{m}
e^{-\frac{1}{2} x_i^2}
\left(\frac{1}{2\pi{\tt i}}\right)^{m}
\int_{{\tt i} \cdot \mathbb{R}}
dz_1
\cdots
\int_{{\tt i} \cdot \mathbb{R}}
dz_m
\prod_{1 \leq i < j \leq m}
(z_j-z_i)
\prod_{i=1}^{m}
e^{\frac{1}{2}(z_i-x_i)^2}.
\end{align}
It is possible to explicitly evaluate the integral \eqref{I-int}, as we now show. Replacing the Vandermonde factor in \eqref{I-int} by its determinant form and using the multilinearity of the determinant, we recover
\begin{align}
\label{I-int2}
I(x_1,\dots,x_m)
&=
\prod_{i=1}^{m}
e^{-\frac{1}{2} x_i^2}
\left(\frac{1}{2\pi {\tt i}}\right)^{m}
\int_{{\tt i} \cdot \mathbb{R}}
dz_1
\cdots
\int_{{\tt i} \cdot \mathbb{R}}
dz_m
\det_{1 \leq i,j \leq m}(z_i^{j-1})
\prod_{i=1}^{m}
e^{\frac{1}{2}(z_i-x_i)^2},
\\
\nonumber
&=
\prod_{i=1}^{m}
e^{-\frac{1}{2} x_i^2}
\det_{1 \leq i,j \leq m}
\left(
\frac{1}{2\pi {\tt i}}
\int_{{\tt i} \cdot \mathbb{R}}
z^{j-1}
e^{\frac{1}{2}(z-x_i)^2}
dz
\right).
\end{align}
Making the change of integration variables ${\tt i}u = z - x_i$ within the second line of \eqref{I-int2}, we have that
\begin{align*}
I(x_1,\dots,x_m)
&=
\prod_{i=1}^{m}
e^{-\frac{1}{2} x_i^2}
\det_{1 \leq i,j \leq m}
\left(
\frac{1}{2\pi}
\int_{-\infty+{\tt i}x_i}^{\infty+{\tt i}x_i}
({\tt i}u+x_i)^{j-1}
e^{-\frac{1}{2}u^2}
du
\right).
\end{align*}
Expanding the factor $({\tt i}u+x_i)^{j-1}$ as a polynomial in $x_i$, this becomes
\begin{align*}
I(x_1,\dots,x_m)
&=
\prod_{i=1}^{m}
e^{-\frac{1}{2} x_i^2}
\det_{1 \leq i,j \leq m}
\left(
\frac{x_i^{j-1}}{2\pi}
\int_{-\infty+{\tt i}x_i}^{\infty+{\tt i}x_i}
e^{-\frac{1}{2}u^2}
du
+
O(x_i^{j-2})
\right),
\end{align*}
and the polynomial term of the form $O(x_i^{j-2})$ can be removed by elementary column transformations. The final result is thus
\begin{align*}
I(x_1,\dots,x_m)
&=
\left(\frac{1}{2\pi}\right)^{\frac{m}{2}}
\prod_{i=1}^{m}
e^{-\frac{1}{2} x_i^2}
\prod_{1 \leq i <j \leq m}
(x_j-x_i).
\end{align*}

\subsection{Final formula}
\label{ssec:formula}

We are now in a position to write the full asymptotic behaviour of the Markov kernel 
$\mathbb{P}_{t,1}(0\cup\lambda^{[m]} \rightarrow \lambda^{[m+1]})$ as $t\rightarrow\infty$. Using \eqref{kernel-reexpress} with $H_{\lambda^{[m]}-1}(t)$ given by
\begin{multline*}
H_{\lambda^{[m]}-1}(t)
\sim
(-1)^{nm}
t^{-\frac{n}{2}\binom{m+1}{2}}
q^{\binom{nm+1}{2}-\frac{1}{2}\binom{m+1}{2} \binom{n}{2}+m\binom{n}{2}}
\\
\times
\exp\left[ \frac{1-q^n}{1-q} mt \right]
\frac{(1-q)^{n\binom{m}{2}}}{(q;q)_n^{m^2}}
g_{\Delta}^{c^{[m]}}(m^n;\vec{Q}^{[m]})
\prod_{i=0}^{n-1}
I\left(x^{[m]}_{im+1},\dots,x^{[m]}_{i(m+1)}\right)
\end{multline*}
and $H_{\lambda^{[m+1]}}(t)$ given by \eqref{m-to-m+1}, we obtain
\begin{multline}
\label{final-formula}
\mathbb{P}_{t,1}(0\cup\lambda^{[m]} \rightarrow \lambda^{[m+1]})
\sim
\bm{1}_{c^{[m]} \prec c^{[m+1]}}
\Upsilon\left(c^{[m]};c^{[m+1]}\right)
q^{{\rm inv}(c^{[m]})-{\rm inv}(c^{[m+1]})}
\cdot
(-1)^{n}
t^{-\frac{n}{2}(m+1)}
\times
\\
q^{\binom{nm+n+1}{2}-\binom{nm+1}{2}-\frac{1}{2}(m+1)\binom{n}{2}}
\frac{(1-q)^{nm}}{(q;q)_n^{2m+1}}
\cdot
\frac{g_{\Delta}^{c^{[m+1]}}\left((m+1)^n;\vec{Q}^{[m+1]}\right)}
{g_{\Delta}^{c^{[m]}}\left(m^n;\vec{Q}^{[m]}\right)}
\prod_{i=1}^{n}
\frac{I\left(x^{[m+1]}_{(i-1)(m+1)+1},\dots,x^{[m+1]}_{i(m+1)}\right)}
{I\left(x^{[m]}_{(i-1)m+1},\dots,x^{[m]}_{im}\right)},
\end{multline}
as $t \rightarrow \infty$. Recall from \eqref{coord-scal} that there is a factor of $(q^{n-\lceil i/(m+1) \rceil} t)^{1/2}$ present in the change of variables from $\ell^{[m+1]}_i \in \mathbb{Z}$ to $x^{[m+1]}_i \in \mathbb{R}$, for all $1 \leq i \leq n(m+1)$. In order to obtain transition densities valid on the scale of the $x^{[m+1]}_i$ variables, we must multiply the above formula by the product of all such factors; namely, by
\begin{align*}
t^{\frac{n}{2} (m+1)} \prod_{i=1}^{n(m+1)} (q^{n-\lceil i/(m+1) \rceil})^{1/2} = 
t^{\frac{n}{2} (m+1)} q^{\frac{1}{2}(m+1)\binom{n}{2}}.
\end{align*}
We then read off the result
\begin{multline}
\label{final-formula2}
\mathbb{P}_{t,1}(0\cup\lambda^{[m]} \rightarrow \lambda^{[m+1]})
\rightarrow
\prod_{i=1}^{n}
\rho_{\rm GUE}
\left(x^{[m]}_{(i-1)m+1},\dots,x^{[m]}_{im} 
\rightarrow
x^{[m+1]}_{(i-1)(m+1)+1},\dots,x^{[m+1]}_{i(m+1)} \right)
dx^{[m+1]}
\\
\times
\bm{1}_{c^{[m]} \prec c^{[m+1]}}
(-1)^{n}
\Upsilon\left(c^{[m]};c^{[m+1]}\right)
q^{{\rm inv}(c^{[m]})-{\rm inv}(c^{[m+1]})}
q^{\binom{nm+n+1}{2}-\binom{nm+1}{2}}
\frac{(1-q)^{nm}}{(q;q)_n^{2m+1}}
\frac{g_{\Delta}^{c^{[m+1]}}\left((m+1)^n;\vec{Q}^{[m+1]}\right)}
{g_{\Delta}^{c^{[m]}}\left(m^n;\vec{Q}^{[m]}\right)}
\end{multline}
as $t \rightarrow \infty$. The convergence in \eqref{final-formula2} is uniform provided that the $x^{[m]}$ and $x^{[m+1]}$ parameters are chosen to vary over compact subsets of $\mathbb{R}$. This completes the first part of the proof of Theorem \ref{thm:main}; it remains to show that the factors present in the second line of \eqref{final-formula2} constitute a valid probability distribution on colour sequences.

\section{Distribution on colour sequences}
\label{sec:discrete-dist}

In the previous section we showed (see \eqref{final-formula2} above) that the Plancherel-specialized LLT Markov kernel \eqref{kernel} splits, under the $t \rightarrow \infty$ asymptotic regime studied, into a product of $n$ independent GUE corners processes multiplied by a further factor valued on colour sequences. Our aim in this section is to show that this extra factor constitutes a discrete probability measure on colour sequences; in showing this, we demonstrate that the right hand side of \eqref{final-formula2} integrates to unity, validating the fact that the set of coloured compositions to which we have restricted our attention captures the full asymptotic behaviour as $t \rightarrow \infty$.

Our primary task will be to better understand the ratio $g_{\Delta}^{c^{[m+1]}} \Big{/}g_{\Delta}^{c^{[m]}}$ appearing in \eqref{final-formula2}. To that end, we begin by defining a family of partition functions that are related to the functions $g_{\Delta}^{c^{[m]}}$ via an explicit symmetry.

\subsection{Partition function $Z$}

Fix two integers $n,m \geq 1$ and a vector $i^{[m]} = (i_1,\dots,i_{nm}) \in [1,n]^{nm}$ such that for all $1 \leq k \leq n$ we have $|\{a: i_a = k\}| = m$. We define the following partition function in the model \eqref{fund-weights}:
\begin{align}
\label{Z-def}
Z\left(x_1,\dots,x_{nm} ; i^{[m]}\right)
=
\tikz{1.2}{
\foreach\y in {1,...,4}{
\draw[lgray,line width=1pt,->] (1.5,\y) -- (4.5,\y);
}
\foreach\x in {1,...,3}{
\draw[lgray,line width=4pt,->] (\x+1,0.5) -- (\x+1,4.5);
}
%spectral parameters
\node[left] at (1,1) {$x_1 \rightarrow$};
\node[left] at (1,2) {$x_2 \rightarrow$};
\node[left] at (1,3) {$\vdots$};
\node[left] at (1,4) {$x_{nm} \rightarrow $};
%bottom labels
\node[below] at (2,0.5) {$\bm{e}_{[1,n]}$};
\node[below] at (3,0.5) {$\cdots$};
\node[below] at (4,0.5) {$\bm{e}_{[1,n]}$};
%top labels
\node[above] at (2,4.5) {$\bm{e}_0$};
\node[above] at (3,4.5) {$\cdots$};
\node[above] at (4,4.5) {$\bm{e}_0$};
%right labels
\node[right] at (4.5,1) {$i_1$};
\node[right] at (4.5,2) {$i_2$};
\node[right] at (4.5,3) {$\vdots$};
\node[right] at (4.5,4) {$i_{nm}$};
%left labels
\node[left] at (1.5,1) {$0$};
\node[left] at (1.5,2) {$0$};
\node[left] at (1.5,3) {$\vdots$};
\node[left] at (1.5,4) {$0$};
}
\end{align}
where each vertex in the $a$-th row of the lattice is assigned rapidity parameter $z = x_a$, for $1 \leq a \leq nm$. We may represent the partition function \eqref{Z-def} algebraically, as follows:
\begin{align*}
Z\left(x_1,\dots,x_{nm} ; i^{[m]}\right)
=
\bra{\bm{e}_{[1,n]}}^{\otimes m}
\DD_{i_1}(x_1) \dots \DD_{i_{nm}}(x_{nm}) 
\ket{\bm{e}_0}^{\otimes m},
\end{align*}
where we recall the row operator definition $\DD_i(x) = T_{0,i}^{\rightarrow}(x;m-1)$ from Section \ref{sec:fused-row}. The partition functions thus defined may be related to those of \eqref{generic-g-sigma}, via the following symmetry:
\begin{prop}
Recall the definition \eqref{nothing} of the trivial element $\Delta \in \mathcal{S}_{m^n}$. For all vectors $i^{[m]} \in [1,n]^{nm}$ we have that
\begin{align}
\label{gZ-sym}
g^{i^{[m]}}_{\Delta}(m^n;x_1,\dots,x_{nm};s)
=
(-s)^{n\binom{m}{2}}
\cdot
Z\left(x_1^{-1},\dots,x_{nm}^{-1} ; i^{[m]}\right)
\Big|_{q \mapsto q^{-1},s \mapsto s^{-1}}
\end{align}
where the variables $q$, $s$ are replaced by their reciprocals in the final partition function.
\end{prop}

\begin{proof}
This is an immediate consequence of the symmetry \eqref{LM-sym} between the vertex weights used to define \eqref{generic-g-sigma} and \eqref{Z-def}.
\end{proof}

\begin{cor}
The $s=0$ and
\begin{align}
\label{x-spec}
(x_1,\dots,x_{nm})
=
\underbrace{(q^{n-1},\dots,q^{n-1})}_{m\ {\rm times}}
\cup
\cdots
\cup
\underbrace{(q,\dots,q)}_{m\ {\rm times}}
\cup
\underbrace{(1,\dots,1)}_{m\ {\rm times}}
\equiv
\vec{Q}^{[m]}
\end{align}
specializations of \eqref{gZ-sym} are given by
\begin{align}
\label{gZ-sym2}
g^{i^{[m]}}_{\Delta}\left(m^n;\vec{Q}^{[m]}\right)
=
\lim_{s \rightarrow \infty}
(-s)^{-n\binom{m}{2}}
\cdot
Z\left(\vec{Q}^{[m]};i^{[m]}\right)
\Big|_{q \mapsto q^{-1}}.
\end{align}
\end{cor}

\subsection{Expansion formula}

\begin{thm}
\label{thm:expand}
Fix a vector $i^{[m]} = (i_1,\dots,i_{nm}) \in [1,n]^{nm}$ such that $|\{a: i_a = k\}| = m$ for all $1 \leq k \leq n$. Then there exist explicit rational functions in $q$, denoted $\Theta\left(i^{[m]};j^{[m+1]}\right)$, such that the following expansion formula holds:
\begin{align}
\label{Z-exp}
g^{i^{[m]}}_{\Delta}\left(m^n;\vec{Q}^{[m]}\right)
=
\sum_{j^{[m+1]}}
\Theta\left(i^{[m]};j^{[m+1]}\right)
g^{j^{[m+1]}}_{\Delta}\left((m+1)^n;\vec{Q}^{[m+1]}\right),
\end{align}
where the sum is over vectors $j^{[m+1]} = \left(j_1,\dots,j_{n(m+1)}\right) \in [1,n]^{n(m+1)}$ such that $|\{a : j_a = k\}| = m+1$ for all $1 \leq k \leq n$.
\end{thm}

The proof of this theorem is split over the subsequent three subsections. In view of the relation \eqref{gZ-sym2}, all properties of the functions $g^{i^{[m]}}_{\Delta}\left(m^n;\vec{Q}^{[m]}\right)$ and $g^{j^{[m+1]}}_{\Delta}\left((m+1)^n;\vec{Q}^{[m+1]}\right)$ may be deduced from those of $Z\left(\vec{Q}^{[m]};i^{[m]}\right) \equiv Z\left(i^{[m]}\right)$ and $Z\left(\vec{Q}^{[m+1]};j^{[m+1]}\right) \equiv Z\left(j^{[m+1]}\right)$; we adopt this approach in our proof of \eqref{Z-exp}.

\subsection{Partition function $\tilde{Z}$}

Our strategy for proving \eqref{Z-exp} is to define another type of partition function, similar to \eqref{Z-def}, and calculate it in two different ways; the two different evaluations effectively yield the left and right hand sides of \eqref{Z-exp}. To that end, for all vectors 
$i^{[m]} = (i_1,\dots,i_{nm}) \in [1,n]^{nm}$ we introduce
\begin{align}
\label{tilde-Z}
\tilde{Z}\left(u;x_1,\dots,x_{nm};i^{[m]}\right)
=
\tikz{1.2}{
\draw[lgray,line width=4pt,->] (0.5,5) -- (4.5,5);
\foreach\y in {1,...,4}{
\draw[lgray,line width=1pt,->] (0.5,\y) -- (4.5,\y);
}
\foreach\x in {0,...,3}{
\draw[lgray,line width=4pt,->] (\x+1,0.5) -- (\x+1,5.5);
}
%spectral parameters
\node[left] at (0,1) {$x_1 \rightarrow$};
\node[left] at (0,2) {$x_2 \rightarrow$};
\node[left] at (0,3) {$\vdots$};
\node[left] at (0,4) {$x_{nm} \rightarrow$};
\node[left] at (0,5) {$(su;r) \rightarrow$};
%bottom labels
\node[below] at (1,0.5) {$\bm{e}_{[1,n]}$};
\node[below] at (2,0.5) {$\bm{e}_{[1,n]}$};
\node[below] at (3,0.5) {$\cdots$};
\node[below] at (4,0.5) {$\bm{e}_{[1,n]}$};
%top labels
\node[above] at (1,5.5) {$\bm{e}_0$};
\node[above] at (2,5.5) {$\bm{e}_0$};
\node[above] at (3,5.5) {$\cdots$};
\node[above] at (4,5.5) {$\bm{e}_0$};
%right labels
\node[right] at (4.5,1) {$i_1$};
\node[right] at (4.5,2) {$i_2$};
\node[right] at (4.5,3) {$\vdots$};
\node[right] at (4.5,4) {$i_{nm}$};
\node[right] at (4.5,5) {$\bm{e}_{[1,n]}$};
%left labels
\node[left] at (0.5,1) {$0$};
\node[left] at (0.5,2) {$0$};
\node[left] at (0.5,3) {$\vdots$};
\node[left] at (0.5,4) {$0$};
\node[left] at (0.5,5) {$\bm{e}_0$};
}
\end{align}
which is effectively obtained by appending an extra fused row and column to the original partition function \eqref{Z-def}. We represent this partition function algebraically as
\begin{align}
\label{tilde-Z-alg}
\tilde{Z}\left(u; x_1,\dots,x_{nm}; i^{[m]}\right)
=
\bra{\bm{e}_{[1,n]}}^{\otimes m+1}
\DD_{i_1}(x_1) \dots \DD_{i_{nm}}(x_{nm}) \DD_{[1,n]}(su;r)
\ket{\bm{e}_0}^{\otimes m+1}.
\end{align}

\subsection{First evaluation of $\tilde{Z}$}

Let us begin by analysing the dependence of the partition function \eqref{tilde-Z} on the parameter $u$. To do so, we need only study the vertices in the top row of \eqref{tilde-Z}. All of these vertices have weight given by \eqref{fused-weights}--\eqref{w-weight} in which $\C = \V = \bm{e}_0$; in that special case the weights simplify as follows:
\begin{align*}
\tilde{L}^{(r,s)}_{su}(\A,\B; \bm{e}_0,\D)
= 
\bm{1}_{\A+\B=\D}
\cdot
u^{|\D|-|\B|}
r^{-2|\A|}
s^{2|\D|}
W^{(r,s)}_u(\A,\B;\bm{e}_0,\D),
\end{align*}
with
\begin{align*}
W^{(r,s)}_u(\A,\B;\bm{e}_0,\D)
&=
\Phi(\bm{e}_0,\D;s^2 r^{-2} u,s^2 u)
\Phi(\bm{e}_0,\B;r^2 u^{-1},r^2)
=
\frac{(r^2;q)_{|\D|} (u;q)_{|\B|}}{(s^2 u;q)_{|\D|} (r^2;q)_{|\B|}}.
\end{align*}
From these expressions we see that $u^{|\B|-|\D|} \cdot (s^2 u;q)_n \cdot \tilde{L}^{(r,s)}_{su}(\A,\B; \bm{e}_0,\D)$ is a polynomial in $u$ of degree $n-|\D|+|\B|$. It follows that $u^{-n} \cdot (s^2 u;q)_n^{m+1} \cdot \tilde{Z}\left(u; x_1,\dots,x_{nm}; i^{[m]}\right)$ is a polynomial in $u$ of degree $nm$; this can be seen by telescoping the degrees of the individual vertices in the top row of the partition function. 

It turns out to be possible to determine all of the zeros of this polynomial explicitly, using the commutation relation \eqref{Cfused-com} with $J=[1,n]$ and $i=i_k$ for each $1 \leq k \leq nm$; indeed, using this relation in \eqref{tilde-Z-alg}, we find that
\begin{multline*}
\tilde{Z}\left(u; x_1,\dots,x_{nm}; i^{[m]}\right)
=
r^{2nm}
\prod_{k=1}^{nm}
q^{n-i_k}
\left(\frac{q x_k -su}{su-r^2 x_k}\right)
\\
\times
\bra{\bm{e}_{[1,n]}}^{\otimes m+1}
\DD_{[1,n]}(su;r)
\DD_{i_1}(x_1) \dots \DD_{i_{nm}}(x_{nm})
\ket{\bm{e}_0}^{\otimes m+1}
\end{multline*}
which allows us to determine that
\begin{align}
\label{all-zeros}
(s^2 u;q)_n^{m+1}
\cdot
\tilde{Z}\left(u; x_1,\dots,x_{nm}; i^{[m]}\right)
&=
\alpha
\cdot
u^n
\cdot
\prod_{k=1}^{nm}
(su-q x_k)
\end{align}
where $\alpha$ is independent of $u$ but may depend on all other parameters. To determine $\alpha$, we seek an appropriate choice for the parameter $u$ in \eqref{all-zeros}. Our choice is motivated by studying the vertex in the top-left corner of the partition function \eqref{tilde-Z}; this vertex is of the form
\begin{align*}
\tikz{0.6}{
\draw[lgray,line width=4pt,->] (-1,0) -- (1,0);
\draw[lgray,line width=4pt,->] (0,-1) -- (0,1);
\node[left] at (-1,0) {\tiny $\bm{e}_0$};\node[right] at (1,0) {\tiny $\A$};
\node[below] at (0,-1) {\tiny $\A$};\node[above] at (0,1) {\tiny $\bm{e}_0$};
\node[left] at (-1.5,0) {$(su;r) \rightarrow$};
}
&=
\tilde{L}_{su}^{(r,s)}(\A,\bm{e}_0;\bm{e}_0,\A)
=
u^{|\A|} r^{-2|\A|} s^{2|\A|}
\Phi(\bm{e}_0,\A;s^2 r^{-2} u,s^2 u)
\Phi(\bm{e}_0,\bm{e}_0;r^2 u^{-1},r^2)
\\
&=
u^{|\A|} r^{-2|\A|} s^{2|\A|}
\frac{(r^2;q)_{|\A|}}{(s^2 u;q)_{|\A|}}.
\end{align*}
From the explicit form of this vertex weight we find that
\begin{align*}
\lim_{u \rightarrow s^{-2} q^{-n+1}}
(s^2 u;q)_n
\cdot
\tilde{L}_{su}^{(r,s)}(\A,\bm{e}_0;\bm{e}_0,\A)
=
\bm{1}_{\A = \bm{e}_{[1,n]}}
\cdot
q^{-(n-1)n}
r^{-2n}
(r^2;q)_n;
\end{align*}
it follows that if we set $u = s^{-2} q^{-n+1}$ in the left hand side of \eqref{all-zeros}, this produces a freezing of the top row and leftmost column in the partition function \eqref{tilde-Z}:
\begin{multline}
\label{pf-split}
\lim_{u \rightarrow s^{-2} q^{-n+1}}
(s^2 u;q)_n^{m+1}
\cdot
\tilde{Z}\left(u; x_1,\dots,x_{nm}; i^{[m]}\right)
\\
=
\lim_{u \rightarrow s^{-2} q^{-n+1}}
\left[
(s^2 u;q)_n^{m+1}
\times
\tikz{1.1}{
\draw[lgray,line width=4pt,->] (0.5,5) -- (4.5,5);
\foreach\y in {1,...,4}{
\draw[lgray,line width=1pt,->] (0.5,\y) -- (4.5,\y);
}
\foreach\x in {0,...,3}{
\draw[lgray,line width=4pt,->] (\x+1,0.5) -- (\x+1,5.5);
}
%spectral parameters
\node[left] at (0,1) {$x_1 \rightarrow$};
\node[left] at (0,2) {$x_2 \rightarrow$};
\node[left] at (0,3) {$\vdots$};
\node[left] at (0,4) {$x_{nm} \rightarrow$};
\node[left] at (0,5) {$(su;r) \rightarrow$};
%bottom labels
\node[below] at (1,0.5) {$\bm{e}_{[1,n]}$};
\node[below] at (2,0.5) {$\bm{e}_{[1,n]}$};
\node[below] at (3,0.5) {$\cdots$};
\node[below] at (4,0.5) {$\bm{e}_{[1,n]}$};
%top labels
\node[above] at (1,5.5) {$\bm{e}_0$};
\node[above] at (2,5.5) {$\bm{e}_0$};
\node[above] at (3,5.5) {$\cdots$};
\node[above] at (4,5.5) {$\bm{e}_0$};
%right labels
\node[right] at (4.5,1) {$i_1$};
\node[right] at (4.5,2) {$i_2$};
\node[right] at (4.5,3) {$\vdots$};
\node[right] at (4.5,4) {$i_{nm}$};
\node[right] at (4.5,5) {$\bm{e}_{[1,n]}$};
%left labels
\node[left] at (0.5,1) {$0$};
\node[left] at (0.5,2) {$0$};
\node[left] at (0.5,3) {$\vdots$};
\node[left] at (0.5,4) {$0$};
\node[left] at (0.5,5) {$\bm{e}_0$};
%paths
\draw[red,line width=0.5pt,->] (0.95,0.5) -- (0.95,5.05) -- (4.3,5.05);
\draw[green,line width=0.5pt,->] (1,0.5) -- (1,5) -- (4.3,5);
\draw[blue,line width=0.5pt,->] (1.05,0.5) -- (1.05,4.95) -- (4.3,4.95);
}
\right]
\end{multline}
where coloured lines flowing through the leftmost column and top row represent the vector $\bm{e}_{[1,n]}$. The quantity \eqref{pf-split} splits into several pieces; apart from the partition function $Z(x_1,\dots,x_{nm};i^{[m]})$ which emerges in the bottom-right corner, there is a contribution from the vertex in the top-left corner, the remaining vertices in the top row, and the remaining vertices in the leftmost column. This leads us to the expression 
\begin{multline*}
\lim_{u \rightarrow s^{-2} q^{-n+1}}
(s^2 u;q)_n^{m+1}
\cdot
\tilde{Z}\left(u; x_1,\dots,x_{nm}; i^{[m]}\right)
=
\lim_{u \rightarrow s^{-2} q^{-n+1}}
\left[ (s^2 u;q)_n \times
\tikz{0.6}{
\draw[lgray,line width=4pt,->] (-1,0) -- (1,0);
\draw[lgray,line width=4pt,->] (0,-1) -- (0,1);
\node[left] at (-1,0) {\tiny $\bm{e}_0$};\node[right] at (1,0) {\tiny $\bm{e}_{[1,n]}$};
\node[below] at (0,-1) {\tiny $\bm{e}_{[1,n]}$};\node[above] at (0,1) {\tiny $\bm{e}_0$};
\node[left] at (-1.5,0) {$(su;r) \rightarrow$};
%\node[below] at (0,-1.5) {$(y_0,s)$};
}
\right]
\\
\times
\lim_{u \rightarrow s^{-2} q^{-n+1}}
\left[ 
(s^2 u;q)_n \times
\tikz{0.6}{
\draw[lgray,line width=4pt,->] (-1,0) -- (1,0);
\draw[lgray,line width=4pt,->] (0,-1) -- (0,1);
\node[left] at (-1,0) {\tiny $\bm{e}_{[1,n]}$};\node[right] at (1,0) {\tiny $\bm{e}_{[1,n]}$};
\node[below] at (0,-1) {\tiny $\bm{e}_0$};\node[above] at (0,1) {\tiny $\bm{e}_0$};
\node[left] at (-2.1,0) {$(su;r) \rightarrow$};
%\node[below] at (0,-1.5) {$(y_j,s)$};
}
\right]^m
\prod_{k=1}^{nm}
\left[
\tikz{0.6}{
\draw[lgray,line width=1pt,->] (-1,0) -- (1,0);
\draw[lgray,line width=4pt,->] (0,-1) -- (0,1);
\node[left] at (-1,0) {\tiny $0$};\node[right] at (1,0) {\tiny $0$};
\node[below] at (0,-1) {\tiny $\bm{e}_{[1,n]}$};\node[above] at (0,1) {\tiny $\bm{e}_{[1,n]}$};
\node[left] at (-1.5,0) {$x_k \rightarrow$};
}
\right]
Z\left(x_1,\dots,x_{nm}; i^{[m]}\right).
\end{multline*}
Explicitly computing the weight of each bracketed expression above, we arrive at the following evaluation of the left hand side of \eqref{all-zeros}:
\begin{multline*}
\lim_{u \rightarrow s^{-2} q^{-n+1}}
(s^2 u;q)_n^{m+1}
\cdot
\tilde{Z}\left(u; x_1,\dots,x_{nm}; i^{[m]}\right)
\\
=
q^{-(n-1)n} r^{-2n} (r^2;q)_n
\Big[ 
s^{2n} (s^{-2}q^{-n+1};q)_n
\Big]^m
\prod_{k=1}^{nm} \frac{1-s q^n x_k}{1-s x_k}
\cdot
Z\left(x_1,\dots,x_{nm}; i^{[m]}\right).
\end{multline*}
Equating this with the corresponding limit of the right hand side of \eqref{all-zeros}, we determine the constant $\alpha$ to be
\begin{align*}
\alpha
=
s^{nm} \left( \frac{s}{r} \right)^{2n} (r^2;q)_n
\frac{\prod_{k=1}^{n} (s^2 q^{n-1}-q^{k-1})^m}{\prod_{k=1}^{nm}(1-sx_k)}
Z\left(x_1,\dots,x_{nm}; i^{[m]}\right).
\end{align*}
Substituting this into \eqref{all-zeros}, we obtain our first evaluation of the partition function:
\begin{multline}
\label{first-eval}
\tilde{Z}\left(u; x_1,\dots,x_{nm}; i^{[m]}\right)
\\
=
u^n
s^{nm}
\left( \frac{s}{r} \right)^{2n}
\frac{(r^2;q)_n}{(s^2 u;q)_n}
\prod_{k=1}^{nm}
\left( \frac{su-q x_k}{1-sx_k} \right)
\prod_{k=1}^{n}
\left(
\frac{s^2 q^{n-1}-q^{k-1}}{1-s^2 q^{k-1} u}
\right)^m
Z\left(x_1,\dots,x_{nm}; i^{[m]}\right).
\end{multline}
Since our ultimate aim is to prove Theorem \ref{thm:expand}, we shall be most interested in setting $u=s^{-1}$, $r=q^{-n/2}$ and choosing $(x_1,\dots,x_{nm})$ as in \eqref{x-spec}. Making these specializations in \eqref{first-eval}, we have
\begin{align}
\label{first-eval2}
\left.
\tilde{Z}\left(s^{-1}; \vec{Q}^{[m]}; i^{[m]} \right)
\right|_{r = q^{-n/2}}
=
(-1)^{nm}
s^{n(m+1)}
q^{n^2+\binom{n}{2}m}
\frac{
(q^{-n};q)_n
(q;q)_n^m
(s^2;q)_n^m
}
{(s;q)_n^{2m+1}}
Z\left(\vec{Q}^{[m]}; i^{[m]}\right).
\end{align}

\subsection{Second evaluation of $\tilde{Z}$}

For the second evaluation of $\tilde{Z}\left(u; x_1,\dots,x_{nm}; i^{[m]}\right)$, we take from the outset $u=s^{-1}$, $r=q^{-n/2}$ and specialize $(x_1,\dots,x_{nm})$ as in \eqref{x-spec}. We then compute \eqref{tilde-Z-alg} via a sequence of manipulations of the row operators; the relations that we need are \eqref{important-relation} and \eqref{row-relation2}. Commutation relation \eqref{important-relation} allows us to reverse the order of a pair of operators $\mathcal{D}_i(x) \mathcal{D}_J(x;q^{-p/2})$, where $J$ has cardinality $p$; \eqref{row-relation2} allows us to split off an unfused row operator from the row operator 
$\mathcal{D}_I(x;q^{-p/2})$ of width $p$, reducing it to a row operator of width $p-1$, and $q$-shifting its spectral parameter $x$. 

For any integer $k \in [1,n]$, let $J(k)$ be a subset of $[1,n]$ with $|J(k)| = k$; this means that, in particular, $J(n) = [1,n]$. Using \eqref{important-relation} repeatedly, we may start from a product of row operators of the form
\begin{align*}
\mathcal{D}_{i_{(k-1)m+1}}(q^{n-k}) \cdots \mathcal{D}_{i_{km}}(q^{n-k})
\mathcal{D}_{J(k)}(q^{n-k};q^{-k/2})
\end{align*}
and drag the width $k$ operator $\mathcal{D}_{J(k)}(q^{n-k};q^{-k/2})$ towards the left of the product. After that, we apply \eqref{row-relation2} and reduce the leftward-emerging row operator to one of width $k-1$ (at the expense of splitting off a single unfused row operator, to its right); this changes its argument from $q^{n-k}$ to $q^{n-k+1}$, and we denote the resulting operator by $\mathcal{D}_{J(k-1)}(q^{n-k+1};q^{-(k-1)/2})$. Repeating this process for all $k \in [1,n]$, beginning with $k=n$ and reducing $k$ by $1$ at each step, it is straightforward to derive the following expansion:
\begin{multline}
\label{psi-expand}
\left.
\tilde{Z}\left(s^{-1}; \vec{Q}^{[m]}; i^{[m]} \right)
\right|_{r = q^{-n/2}}
=
\left[ \frac{(1-q)^{n}}{(q;q)_n} \right]^m
\\
\times
\sum_{j^{[m+1]}}
\Psi\left( i^{[m]}; j^{[m+1]} \right)
\bra{\bm{e}_{[1,n]}}^{\otimes m+1}
\prod_{k=1}^{n}
\Big[
\mathcal{D}_{j_{(k-1)(m+1)+1}}(q^{n-k})
\cdots
\mathcal{D}_{j_{k(m+1)}}(q^{n-k})
\Big]
\ket{\bm{e}_0}^{\otimes m+1}
\end{multline}
where the sum is taken over all vectors $j^{[m+1]} = (j_1,\dots,j_{n(m+1)}) \in [1,n]^{m+1}$ and with the coefficients $\Psi\left( i^{[m]}; j^{[m+1]} \right)$ given by the following one-row partition function:
\begin{multline*}
\Psi\left( i^{[m]}; j^{[m+1]} \right)
=
\\
\tikz{1.3}{
\draw[lgray,line width=4pt,<-] (0.5,0) -- (12,0);
\foreach\x in {1,2,3,5,6,7,9,10,11}{
\draw[lgray,line width=1.5pt,->] (\x,0) -- (\x,0.5);
}
\node[above] at (1,0.5) {$j_1$};
\node[above] at (2,0.5) {$\cdots$};
\node[above] at (3,0.5) {$j_{m+1}$};
\node[above] at (5,0.5) {$\cdots$};
\node[above] at (6,0.5) {$\cdots$};
\node[above] at (7,0.5) {$\cdots$};
\node[above] at (9,0.5) {$j_{(n-1)(m+1)+1}$};
\node[above] at (10.2,0.5) {$\cdots$};
\node[above] at (11,0.5) {$j_{n(m+1)}$};
\foreach\x in {1.5,2.5,5.5,6.5,9.5,10.5}{
\draw[lgray,line width=1.5pt,->] (\x,-0.5) -- (\x,0);
}
\node[below] at (1.5,-0.5) {$i_1$};
\node[below] at (2,-0.5) {$\cdots$};
\node[below] at (2.5,-0.5) {$i_m$};
\node[below] at (5.5,-0.5) {$\cdots$};
\node[below] at (6.5,-0.5) {$\cdots$};
\node[below] at (9.2,-0.5) {$i_{(n-1)m+1}$};
\node[below] at (10.1,-0.5) {$\cdots$};
\node[below] at (10.7,-0.5) {$i_{nm}$};
\node[left] at (0.5,0) {$\emptyset$}; \node[right] at (12,0) {$[1,n]$};
}
\end{multline*}
where colours are conserved in the direction of arrow flow, and with a weight of $q$ assigned to each of the events
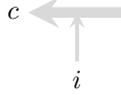
\begin{align*}
\tikz{1.3}{
\draw[lgray,line width=4pt,<-] (0.5,0) -- (1.5,0);
\draw[lgray,line width=1.5pt,->] (1,-0.5) -- (1,0);
\node[below] at (1,-0.5) {$i$};
\node[left] at (0.5,0) {$c$};
}
\end{align*}
which denotes a path of colour $c$ passing above a colour $i$, with $c>i$. Recasting \eqref{psi-expand} in terms of the family of partition functions \eqref{Z-def}, we recover the expansion
\begin{align}
\label{second-eval}
\left.
\tilde{Z}\left(s^{-1}; \vec{Q}^{[m]}; i^{[m]} \right)
\right|_{r = q^{-n/2}}
=
\left[ \frac{(1-q)^{n}}{(q;q)_n} \right]^m
\times
\sum_{j^{[m+1]}}
\bm{1}_{i^{[m]} \prec j^{[m+1]}}
\Psi\left( i^{[m]}; j^{[m+1]} \right)
Z\left( \vec{Q}^{[m+1]}; j^{[m+1]} \right).
\end{align}

\subsection{Comparing evaluations of $\tilde{Z}$}

Comparing equations \eqref{first-eval2} and \eqref{second-eval}, we have shown that
\begin{multline*}
Z\left(\vec{Q}^{[m]}; i^{[m]}\right)
=
(-1)^{n(m+1)}
s^{-n(m+1)}
q^{-\binom{n}{2}(m+1)}
\\
\times
\left( \frac{(1-q)^{n}}{(s^2;q)_n} \right)^m
\left( \frac{(s;q)_n}{(q;q)_n} \right)^{2m+1}
\sum_{j^{[m+1]}}
\bm{1}_{i^{[m]} \prec j^{[m+1]}}
\Psi\left( i^{[m]}; j^{[m+1]} \right)
Z\left( \vec{Q}^{[m+1]}; j^{[m+1]} \right).
\end{multline*}
Inverting $q$ in the previous equation and multiplying both sides by $(-s)^{-n \binom{m}{2}}$, we obtain 
\begin{multline*}
(-s)^{-n \binom{m}{2}}
Z\left(\vec{Q}^{[m]}; i^{[m]}\right)
\Big|_{q \mapsto q^{-1}}
=
q^{\binom{n}{2}(m+1)}
(-s)^{-n}
\left( \frac{(1-q^{-1})^{n}}{(s^2;q^{-1})_n} \right)^m
\left( \frac{(s;q^{-1})_n}{(q^{-1};q^{-1})_n} \right)^{2m+1}
\\
\times
\sum_{j^{[m+1]}}
\bm{1}_{i^{[m]} \prec j^{[m+1]}}
\Psi\left( i^{[m]}; j^{[m+1]} \right)
\Big|_{q \mapsto q^{-1}}
(-s)^{-n \binom{m+1}{2}}
Z\left( \vec{Q}^{[m+1]}; j^{[m+1]} \right)
\Big|_{q \mapsto q^{-1}},
\end{multline*}
and we now take the limit $s \rightarrow \infty$ using \eqref{gZ-sym2}:
\begin{align}
\label{4548}
g^{i^{[m]}}_{\Delta}\left(m^n;\vec{Q}^{[m]}\right)
=
\frac{(q^{-1}-1)^{nm}}{(q^{-1};q^{-1})_n^{2m+1}}
\sum_{j^{[m+1]}}
\bm{1}_{i^{[m]} \prec j^{[m+1]}}
\Psi\left( i^{[m]}; j^{[m+1]} \right)
\Big|_{q \mapsto q^{-1}}
g^{j^{[m+1]}}_{\Delta}\left((m+1)^n;\vec{Q}^{[m+1]}\right).
\end{align}
Finally, rearranging the factors in \eqref{4548}, one recovers
\begin{multline*}
g^{i^{[m]}}_{\Delta}\left(m^n;\vec{Q}^{[m]}\right)
=
(-1)^n
q^{\binom{nm+n+1}{2}-\binom{nm+1}{2}}
\frac{(1-q)^{nm}}{(q;q)_n^{2m+1}}
\\
\times
\sum_{j^{[m+1]}}
\bm{1}_{i^{[m]} \prec j^{[m+1]}}
\Psi\left( i^{[m]}; j^{[m+1]} \right)
\Big|_{q \mapsto q^{-1}}
g^{j^{[m+1]}}_{\Delta}\left((m+1)^n;\vec{Q}^{[m+1]}\right).
\end{multline*}
This completes the proof of Theorem \ref{thm:expand}, with the coefficients in \eqref{Z-exp} identified as
\begin{align}
\label{4792}
\Theta\left( i^{[m]}; j^{[m+1]} \right)
=
\bm{1}_{i^{[m]} \prec j^{[m+1]}}
(-1)^n
q^{\binom{nm+n+1}{2}-\binom{nm+1}{2}}
\frac{(1-q)^{nm}}{(q;q)_n^{2m+1}}
\Psi\left( i^{[m]}; j^{[m+1]} \right)
\Big|_{q \mapsto q^{-1}}.
\end{align}

\subsection{Completing the proof of Theorem \ref{thm:main}}

Rearrangement of \eqref{Z-exp}, using \eqref{4792}, yields the fact that
\begin{align}
\label{sum-to-1-seqs}
\sum_{j^{[m+1]}}
\bm{1}_{i^{[m]} \prec j^{[m+1]}}
(-1)^n
q^{\binom{nm+n+1}{2}-\binom{nm+1}{2}}
\frac{(1-q)^{nm}}{(q;q)_n^{2m+1}}
\Psi\left( i^{[m]}; j^{[m+1]} \right)
\Big|_{q \mapsto q^{-1}}
\frac{
g^{j^{[m+1]}}_{\Delta}\left((m+1)^n;\vec{Q}^{[m+1]}\right)
}
{
g^{i^{[m]}}_{\Delta}\left(m^n;\vec{Q}^{[m]}\right)
}
=1.
\end{align}
This is precisely the result that we need to complete the proof of Theorem \ref{thm:main}; all that remains is to match the summand of \eqref{sum-to-1-seqs} with the second line of equation \eqref{final-formula2}. We observe that all factors match perfectly, modulo the following proposition that takes care of the factors that are not yet manifestly equal:

\begin{prop}
Fix two colour sequences $c^{[m]} \in [1,n]^{nm}$ and $c^{[m+1]} \in [1,n]^{n(m+1)}$ which satisfy the constraints 
$|\{a: c^{[m]}_a = k\}| = m$ and $|\{a: c^{[m+1]}_a = k\}| = m+1$ for all $1 \leq k \leq n$. Assuming also that $c^{[m]} \prec c^{[m+1]}$, the following relation holds:
\begin{align}
\label{statistic-match}
\Psi\left( c^{[m]}; c^{[m+1]} \right)
\Upsilon\left(c^{[m]};c^{[m+1]}\right)
q^{{\rm inv}(c^{[m]})-{\rm inv}(c^{[m+1]})}
=
1.
\end{align}
\end{prop}

\begin{proof}
This statement is equivalent to \cite[Lemma 10.1.2]{ABW21}, but rather than making a detailed match with that result, we give a standalone proof. As in \cite{ABW21}, we shall proceed by induction on $n$.

The $n=1$ case of \eqref{statistic-match} is trivial; in that case we must have $c^{[m]} = 1^m$ and $c^{[m+1]} = 1^{m+1}$, which yields
\begin{align*}
\Psi\left( c^{[m]}; c^{[m+1]} \right)
=
\Upsilon\left(c^{[m]};c^{[m+1]}\right)
=
q^{{\rm inv}(c^{[m]})-{\rm inv}(c^{[m+1]})}
=
1.
\end{align*}
We shall take as our inductive assumption that \eqref{statistic-match} is valid for $n = p-1$, for given $p \geq 2$. Then for any $c^{[m]} \in [1,p-1]^{(p-1)m}$ and $c^{[m+1]} \in [1,p-1]^{(p-1)(m+1)}$, we also have
\begin{align}
\label{ind-step}
\Psi\left( c^{[m]}\cup p^m; c^{[m+1]} \cup p^{m+1} \right)
\Upsilon\left(c^{[m]}\cup p^m;c^{[m+1]} \cup p^{m+1}\right)
q^{{\rm inv}(c^{[m]}\cup p^m)-{\rm inv}(c^{[m+1]} \cup p^{m+1})}
=
1.
\end{align}
Indeed, one can verify that appending a bundle of (maximal) colours $p$ to both $c^{[m]}$ and $c^{[m+1]}$ does not affect either of the partition functions $\Psi$ and $\Upsilon$, neither does it affect the value of the statistic ${\rm inv}$. Equation \eqref{ind-step} then constitutes a solution of \eqref{statistic-match} at $n=p$; to prove that \eqref{statistic-match} holds generally at $n=p$, we seek ``local moves'' that can be applied to $c^{[m]}\cup p^m$ and $c^{[m+1]} \cup p^{m+1}$ to bring them to generic colour sequences in $[1,p]^{pm}$ and $[1,p]^{p(m+1)}$, respectively. These local moves will have the property that they preserve the interlacing property of the colour sequences, and applying them to a solution of \eqref{statistic-match} will yield a new solution. 

The first two local moves that one requires are jumps across bundles:
\begin{align}
\label{move1}
\tikz{0.8}{
\draw[lgray,line width=4pt] (0.5,0) -- (7.5,0);
\foreach\x in {1,2,3,5,6,7}{
\draw[lgray,line width=1.5pt,->] (\x,0) -- (\x,0.5);
}
\node[above] at (1,0.5) {$\cdots$};
\node[above] at (2,0.5) {$\cdots$};
\node[above] at (3,0.5) {$a$};
\node[above] at (5,0.5) {$p$};
\node[above] at (6,0.5) {$\cdots$};
\node[above] at (7,0.5) {$\cdots$};
\foreach\x in {1.5,2.5,5.5,6.5}{
\draw[lgray,line width=1.5pt,->] (\x,-0.5) -- (\x,0);
}
\node[below] at (1.5,-0.5) {$\cdots$};
\node[below] at (2.5,-0.5) {$\cdots$};
\node[below] at (5.5,-0.5) {$\cdots$};
\node[below] at (6.5,-0.5) {$\cdots$};
\draw[line width=2pt] (2.75,0) -- (5.25,0);
}
\mapsto
\tikz{0.8}{
\draw[lgray,line width=4pt] (0.5,0) -- (7.5,0);
\foreach\x in {1,2,3,5,6,7}{
\draw[lgray,line width=1.5pt,->] (\x,0) -- (\x,0.5);
}
\node[above] at (1,0.5) {$\cdots$};
\node[above] at (2,0.5) {$\cdots$};
\node[above] at (3,0.5) {$p$};
\node[above] at (5,0.5) {$a$};
\node[above] at (6,0.5) {$\cdots$};
\node[above] at (7,0.5) {$\cdots$};
\foreach\x in {1.5,2.5,5.5,6.5}{
\draw[lgray,line width=1.5pt,->] (\x,-0.5) -- (\x,0);
}
\node[below] at (1.5,-0.5) {$\cdots$};
\node[below] at (2.5,-0.5) {$\cdots$};
\node[below] at (5.5,-0.5) {$\cdots$};
\node[below] at (6.5,-0.5) {$\cdots$};
\draw[line width=2pt] (2.75,0) -- (5.25,0);
}
\end{align}
\begin{align}
\label{move2}
\tikz{0.8}{
\draw[lgray,line width=4pt] (0.5,0) -- (7.5,0);
\foreach\x in {1,2,3,5,6,7}{
\draw[lgray,line width=1.5pt,->] (\x,0) -- (\x,0.5);
}
\node[above] at (1,0.5) {$\cdots$};
\node[above] at (2,0.5) {$\cdots$};
\node[above] at (3,0.5) {$b$};
\node[above] at (5,0.5) {$c$};
\node[above] at (6,0.5) {$\cdots$};
\node[above] at (7,0.5) {$\cdots$};
\foreach\x in {1.5,2.5,5.5,6.5}{
\draw[lgray,line width=1.5pt,->] (\x,-0.5) -- (\x,0);
}
\node[below] at (1.5,-0.5) {$\cdots$};
\node[below] at (2.5,-0.5) {$a$};
\node[below] at (5.5,-0.5) {$p$};
\node[below] at (6.5,-0.5) {$\cdots$};
\draw[line width=2pt] (2.25,0) -- (5.75,0);
}
\mapsto
\tikz{0.8}{
\draw[lgray,line width=4pt] (0.5,0) -- (7.5,0);
\foreach\x in {1,2,3,5,6,7}{
\draw[lgray,line width=1.5pt,->] (\x,0) -- (\x,0.5);
}
\node[above] at (1,0.5) {$\cdots$};
\node[above] at (2,0.5) {$\cdots$};
\node[above] at (3,0.5) {$b$};
\node[above] at (5,0.5) {$c$};
\node[above] at (6,0.5) {$\cdots$};
\node[above] at (7,0.5) {$\cdots$};
\foreach\x in {1.5,2.5,5.5,6.5}{
\draw[lgray,line width=1.5pt,->] (\x,-0.5) -- (\x,0);
}
\node[below] at (1.5,-0.5) {$\cdots$};
\node[below] at (2.5,-0.5) {$p$};
\node[below] at (5.5,-0.5) {$a$};
\node[below] at (6.5,-0.5) {$\cdots$};
\draw[line width=2pt] (2.25,0) -- (5.75,0);
}
\end{align}
where arrows at the bottom of the diagram indicate colours in $c^{[m]}$, while those at the top indicate colours in $c^{[m+1]}$. All marked colours $a,b,c,p$ are assumed to be distinct (were they not distinct in the second case, the interlacing property of colours would be violated in at least one of the pictures shown), with $p$ being the largest. All colours remain fixed under these moves, apart from $a$ and $p$, which switch places. Let $\Psi_{{\sf L}/{\sf R}}$ and $\Upsilon_{{\sf L}/{\sf R}}$ denote the contributions to the functions $\Psi$ and $\Upsilon$ coming only from the indicated colours and marked regions of the diagrams, on the left/right hand side of both \eqref{move1} and \eqref{move2}.

In the case of \eqref{move1}, one finds that
\begin{align*}
\Psi_{\sf L}\cdot \Upsilon_{\sf L}
=
1,
\qquad
\Psi_{\sf R}\cdot \Upsilon_{\sf R}
=
q,
\end{align*}
but since ${\rm inv}(c^{[m+1]})$ also increases by $1$ under the move \eqref{move1}, we find that \eqref{statistic-match} is preserved. In a similar vein, in the case of \eqref{move2} one obtains
\begin{align*}
\Psi_{\sf L}\cdot \Upsilon_{\sf L}
=
q\cdot q^{\bm{1}_{b>c}+\bm{1}_{b>a}+\bm{1}_{c>a}},
\qquad
\Psi_{\sf R}\cdot \Upsilon_{\sf R}
=
q^{\bm{1}_{b>a}+\bm{1}_{c>a}} \cdot q^{\bm{1}_{b>c}},
\end{align*}
and since ${\rm inv}(c^{[m]})$ also increases by $1$ under the move \eqref{move2}, we again find that \eqref{statistic-match} is preserved.

The remaining two local moves needed are jumps within bundles:
\begin{align}
\label{move3}
\tikz{0.8}{
\draw[lgray,line width=4pt] (0.5,0) -- (4.5,0);
\foreach\x in {1,2,3,4}{
\draw[lgray,line width=1.5pt,->] (\x,0) -- (\x,0.5);
}
\node[above] at (1,0.5) {$\cdots$};
\node[above] at (2,0.5) {$a$};
\node[above] at (3,0.5) {$p$};
\node[above] at (4,0.5) {$\cdots$};
\foreach\x in {1.5,2.5,3.5}{
\draw[lgray,line width=1.5pt,->] (\x,-0.5) -- (\x,0);
}
\node[below] at (1.5,-0.5) {$\cdots$};
\node[below] at (2.5,-0.5) {$b$};
\node[below] at (3.5,-0.5) {$\cdots$};
\draw[line width=2pt] (1.75,0) -- (3.25,0);
}
\mapsto
\tikz{0.8}{
\draw[lgray,line width=4pt] (0.5,0) -- (4.5,0);
\foreach\x in {1,2,3,4}{
\draw[lgray,line width=1.5pt,->] (\x,0) -- (\x,0.5);
}
\node[above] at (1,0.5) {$\cdots$};
\node[above] at (2,0.5) {$p$};
\node[above] at (3,0.5) {$a$};
\node[above] at (4,0.5) {$\cdots$};
\foreach\x in {1.5,2.5,3.5}{
\draw[lgray,line width=1.5pt,->] (\x,-0.5) -- (\x,0);
}
\node[below] at (1.5,-0.5) {$\cdots$};
\node[below] at (2.5,-0.5) {$b$};
\node[below] at (3.5,-0.5) {$\cdots$};
\draw[line width=2pt] (1.75,0) -- (3.25,0);
}
\end{align}
\begin{align}
\label{move4}
\tikz{0.8}{
\draw[lgray,line width=4pt] (0.5,0) -- (5.5,0);
\foreach\x in {1,2,3,4,5}{
\draw[lgray,line width=1.5pt,->] (\x,0) -- (\x,0.5);
}
\node[above] at (1,0.5) {$\cdots$};
\node[above] at (2,0.5) {$\cdots$};
\node[above] at (3,0.5) {$b$};
\node[above] at (4,0.5) {$\cdots$};
\node[above] at (5,0.5) {$\cdots$};
\foreach\x in {1.5,2.5,3.5,4.5}{
\draw[lgray,line width=1.5pt,->] (\x,-0.5) -- (\x,0);
}
\node[below] at (1.5,-0.5) {$\cdots$};
\node[below] at (2.5,-0.5) {$a$};
\node[below] at (3.5,-0.5) {$p$};
\node[below] at (4.5,-0.5) {$\cdots$};
\draw[line width=2pt] (2.25,0) -- (3.75,0);
}
\mapsto
\tikz{0.8}{
\draw[lgray,line width=4pt] (0.5,0) -- (5.5,0);
\foreach\x in {1,2,3,4,5}{
\draw[lgray,line width=1.5pt,->] (\x,0) -- (\x,0.5);
}
\node[above] at (1,0.5) {$\cdots$};
\node[above] at (2,0.5) {$\cdots$};
\node[above] at (3,0.5) {$b$};
\node[above] at (4,0.5) {$\cdots$};
\node[above] at (5,0.5) {$\cdots$};
\foreach\x in {1.5,2.5,3.5,4.5}{
\draw[lgray,line width=1.5pt,->] (\x,-0.5) -- (\x,0);
}
\node[below] at (1.5,-0.5) {$\cdots$};
\node[below] at (2.5,-0.5) {$p$};
\node[below] at (3.5,-0.5) {$a$};
\node[below] at (4.5,-0.5) {$\cdots$};
\draw[line width=2pt] (2.25,0) -- (3.75,0);
}
\end{align}
where all marked colours $a,b,p$ are assumed to be distinct, with $p$ being the largest. As previously, all colours remain fixed under these moves apart from $a$ and $p$, which exchange their places. We again let $\Psi_{{\sf L}/{\sf R}}$ and $\Upsilon_{{\sf L}/{\sf R}}$ denote the contributions to the functions $\Psi$ and $\Upsilon$ coming only from the indicated colours and marked regions of the diagrams, on the left/right hand side of both \eqref{move3} and \eqref{move4}. 

In the case of \eqref{move3}, we get
\begin{align*}
\Psi_{\sf L}\cdot \Upsilon_{\sf L}
=
q^{\bm{1}_{a>b}} \cdot q,
\qquad
\Psi_{\sf R}\cdot \Upsilon_{\sf R}
=
q \cdot q^{1+\bm{1}_{a>b}},
\end{align*}
with the discrepancy between left and right hand sides cured by the fact that ${\rm inv}(c^{[m+1]})$ increases by $1$ under the move in question; hence \eqref{move3} preserves solutions of \eqref{statistic-match}. Finally, in the case of \eqref{move4} one finds that
\begin{align*}
\Psi_{\sf L}\cdot \Upsilon_{\sf L}
=
q \cdot q^{\bm{1}_{a<b}},
\qquad
\Psi_{\sf R}\cdot \Upsilon_{\sf R}
=
q^{\bm{1}_{a<b}} \cdot 1;
\end{align*}
left and right hand sides match after accounting for the fact that ${\rm inv}(c^{[m]})$ is increased by $1$ under this move. Therefore, \eqref{move4} also preserves solutions of \eqref{statistic-match}.

This suffices to prove \eqref{statistic-match} generally at $n=p$, since we have already exhibited one solution \eqref{ind-step}, and it is clear that successive application of the four local moves generates all possible colour sequences. The proof of \eqref{statistic-match} is completed, and with it, the proof of Theorem \ref{thm:main}.
\end{proof}

\subsection{Probability distribution on interlacing triangular arrays}

The results of this section allow us to conclude that the quantity
\begin{multline}
\label{col-markov}
\mathbb{P}_{\rm col} 
\left( c^{[m]} \rightarrow c^{[m+1]} \right)
\\
:=
\bm{1}_{c^{[m]} \prec c^{[m+1]}}
(-1)^n
q^{\binom{nm+n+1}{2}-\binom{nm+1}{2}}
\frac{(1-q)^{nm}}{(q;q)_n^{2m+1}}
\Psi\left( c^{[m]}; c^{[m+1]} \right)
\Big|_{q \mapsto q^{-1}}
\frac{
g^{c^{[m+1]}}_{\Delta}\left((m+1)^n;\vec{Q}^{[m+1]}\right)
}
{
g^{c^{[m]}}_{\Delta}\left(m^n;\vec{Q}^{[m]}\right)
}
\end{multline}
defines a transition probability on colour sequences; this allows one to construct a discrete-time Markov process living on {\it interlacing triangular arrays}, as we define below.

\begin{defn}[Interlacing triangular array]
\label{def:interlace}
Fix integers $n \geq 1$, $N \geq 1$. For all $1 \leq i \leq n$, $1 \leq j \leq k \leq N$, fix positive integers $T^{(i)}_{j,k} \in [1,n]$ subject to two constraints: 

{\bf (a)} For each $k \in [1,N]$, the collection $\{T^{(i)}_{j,k} \}_{1 \leq i \leq n,1 \leq j \leq k}$ is equal to $\{1^k\} \cup \{2^k\} \cup \cdots \cup \{n^k\}$, with the equality being at the level of sets;

{\bf (b)} Let the {\it horizontal coordinate} of the integer $T^{(i)}_{j,k}$ be defined as $c(i,j,k) = iN+j-(N+k)/2$. If $T^{(i)}_{j,k} = T^{(i')}_{j',k} = a \in [1,n]$, $c(i,j,k) < c(i',j',k)$ for some $i,j,i',j'$ and $1< k \leq N$, then we assume that there exists $i'',j''$ such that $T^{(i'')}_{j'',k-1} = a$ and $c(i,j,k) < c(i'',j'',k-1) < c(i',j',k)$; this is the {\it interlacing} property of our array. 
 
We refer to such a collection of positive integers as an {\it interlacing triangular array} of {\it rank} $n$ and {\it height} $N$. Let $\mathcal{T}_N(n)$ denote the set of all interlacing triangular arrays of rank $n$ and height $N$.
\end{defn}

\begin{rmk}
Every interlacing triangular array in $\mathcal{T}_N(n)$ is in one-to-one correspondence with a string $c^{[1]} \prec c^{[2]} \prec \cdots \prec c^{[N]}$ of interlacing colour sequences $c^{[k]} \in [1,n]^{nk}$. The colour sequence $c^{[k]}$ is obtained simply by reading off the $k$-th row of the interlacing triangular array.
\end{rmk}

\begin{ex}[$n=2$, $N=3$]
A permissible interlacing triangular array of rank $2$ and height $3$:
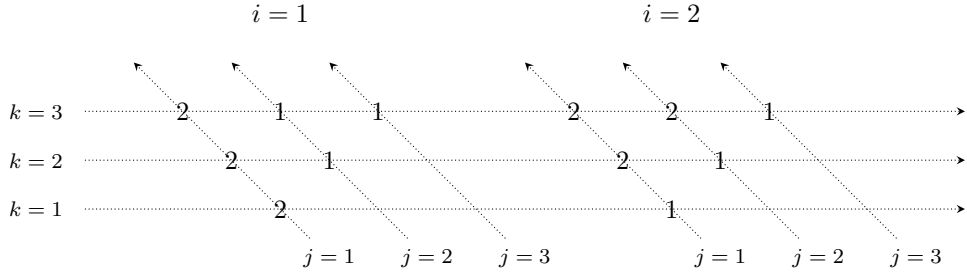
\begin{align*}
\tikz{1.3}{
\node at (0,2) {$i=1$};
\node at (4,2) {$i=2$};
%%%%%%%%%
\node at (0.5,-0.5) {\fs $j=1$};
\node at (1.5,-0.5) {\fs $j=2$};
\node at (2.5,-0.5) {\fs $j=3$};
\node at (4.5,-0.5) {\fs $j=1$};
\node at (5.5,-0.5) {\fs $j=2$};
\node at (6.5,-0.5) {\fs $j=3$};
%%%%%%%%%
\node at (-2.5,0) {\fs $k=1$};
\node at (-2.5,0.5) {\fs $k=2$};
\node at (-2.5,1) {\fs $k=3$};
%%%%%%%%%
\draw[densely dotted,->] (-2,0) -- (7,0);
\draw[densely dotted,->] (-2,0.5) -- (7,0.5);
\draw[densely dotted,->] (-2,1) -- (7,1);
\draw[densely dotted,->] (0.3,-0.3) -- (-1.5,1.5);
\draw[densely dotted,->] (1.3,-0.3) -- (-0.5,1.5);
\draw[densely dotted,->] (2.3,-0.3) -- (0.5,1.5);
\draw[densely dotted,->] (4.3,-0.3) -- (2.5,1.5);
\draw[densely dotted,->] (5.3,-0.3) -- (3.5,1.5);
\draw[densely dotted,->] (6.3,-0.3) -- (4.5,1.5);
%%%%%%%%%
\node at (0,0) {2};
\node at (-0.5,0.5) {2};
\node at (0.5,0.5) {1};
\node at (-1,1) {2};
\node at (0,1) {1};
\node at (1,1) {1};
%%%%%%%%%
\node at (4,0) {1};
\node at (3.5,0.5) {2};
\node at (4.5,0.5) {1};
\node at (3,1) {2};
\node at (4,1) {2};
\node at (5,1) {1};
}
\end{align*}
Recall that the integers in the array are written collectively in the form $T^{(i)}_{j,k}$. In this picture, the index $1 \leq i \leq 2$ increases from left to right and labels individual triangular arrays; the index $1 \leq k \leq 3$ labels the row in question, and the index $1 \leq j \leq k$ is used to label diagonals in each triangular array.

Reading the numbers in the $k$-th row we recover the elements of the set $\{1^k\} \cup \{2^k\}$, which is constraint {\bf (a)}; also, the interlacing constraint {\bf (b)} holds separately both for the coordinates of the numbers $1$ and $2$.

In this example one has $c^{[1]} = (2,1)$, $c^{[2]} = (2,1,2,1)$, $c^{[3]} = (2,1,1,2,2,1)$.
\end{ex}

\begin{ex}[$n=3$, $N=3$]
Under the translation red=1, green=2 and blue=3, Figure \ref{fig:gue} corresponds with the following interlacing triangular array:
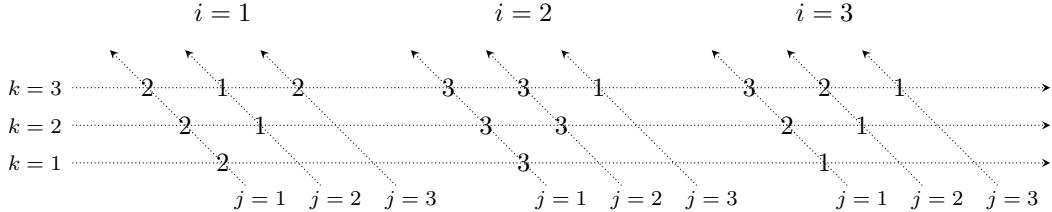
\begin{align*}
\tikz{1}{
\node at (0,2) {$i=1$};
\node at (4,2) {$i=2$};
\node at (8,2) {$i=3$};
%%%%%%%%%
\node at (0.5,-0.5) {\fs $j=1$};
\node at (1.5,-0.5) {\fs $j=2$};
\node at (2.5,-0.5) {\fs $j=3$};
\node at (4.5,-0.5) {\fs $j=1$};
\node at (5.5,-0.5) {\fs $j=2$};
\node at (6.5,-0.5) {\fs $j=3$};
\node at (8.5,-0.5) {\fs $j=1$};
\node at (9.5,-0.5) {\fs $j=2$};
\node at (10.5,-0.5) {\fs $j=3$};
%%%%%%%%%
\node at (-2.5,0) {\fs $k=1$};
\node at (-2.5,0.5) {\fs $k=2$};
\node at (-2.5,1) {\fs $k=3$};
%%%%%%%%%
\draw[densely dotted,->] (-2,0) -- (11,0);
\draw[densely dotted,->] (-2,0.5) -- (11,0.5);
\draw[densely dotted,->] (-2,1) -- (11,1);
\draw[densely dotted,->] (0.3,-0.3) -- (-1.5,1.5);
\draw[densely dotted,->] (1.3,-0.3) -- (-0.5,1.5);
\draw[densely dotted,->] (2.3,-0.3) -- (0.5,1.5);
\draw[densely dotted,->] (4.3,-0.3) -- (2.5,1.5);
\draw[densely dotted,->] (5.3,-0.3) -- (3.5,1.5);
\draw[densely dotted,->] (6.3,-0.3) -- (4.5,1.5);
\draw[densely dotted,->] (8.3,-0.3) -- (6.5,1.5);
\draw[densely dotted,->] (9.3,-0.3) -- (7.5,1.5);
\draw[densely dotted,->] (10.3,-0.3) -- (8.5,1.5);
%%%%%%%%%
\node at (0,0) {2};
\node at (-0.5,0.5) {2};
\node at (0.5,0.5) {1};
\node at (-1,1) {2};
\node at (0,1) {1};
\node at (1,1) {2};
%%%%%%%%%
\node at (4,0) {3};
\node at (3.5,0.5) {3};
\node at (4.5,0.5) {3};
\node at (3,1) {3};
\node at (4,1) {3};
\node at (5,1) {1};
%%%%%%%%%
\node at (8,0) {1};
\node at (7.5,0.5) {2};
\node at (8.5,0.5) {1};
\node at (7,1) {3};
\node at (8,1) {2};
\node at (9,1) {1};
}
\end{align*}
In this example one has $c^{[1]} = (2,3,1)$, $c^{[2]} = (2,1,3,3,2,1)$, $c^{[3]} = (2,1,2,3,3,1,3,2,1)$.
\end{ex}

\begin{cor}
\label{cor:joint-col}
Let $T \in \mathcal{T}_n(N)$ be an interlacing triangular array generated by $N$ successive applications of the Markov kernel \eqref{col-markov}. Then the array $T^{(i)}_{j,k}$, $1 \leq i \leq n$, 
$1 \leq j \leq k \leq N$, has joint distribution
\begin{multline}
\label{joint-distr-col}
\mathbb{P}_{\rm col}
\left( T^{(i)}_{j,k} = c^{[k]}_{(i-1)k+j} ; 1 \leq i \leq n,\ 1 \leq j \leq k \leq N \right)
\\
=
\bm{1}_{c^{[1]} \prec \cdots \prec c^{[N]}}
(-1)^{nN} q^{\binom{nN+1}{2}}
\frac{(1-q)^{n \binom{N}{2}}}{(q;q)_{n}^{N^2}}
g^{c^{[N]}}_{\Delta}\left(N^n;\vec{Q}^{[N]}\right)
\prod_{i=1}^{N}
\Psi \left( c^{[i-1]} ; c^{[i]} \right)
\Big|_{q \mapsto q^{-1}}.
\end{multline}
\end{cor}

\subsection{Explicit calculations}

In this subsection we document some explicit calculations in the case $n=2$, based on direct application of the formula \eqref{col-markov}. All factors appearing in \eqref{col-markov} are straightforwardly computed, with the exception of the functions $g^{c^{[m]}}_{\Delta}$ and $g^{c^{[m+1]}}_{\Delta}$. To evaluate the latter, we make use of the following factorization result, which follows from \cite[Proposition 11.6.1]{ABW21}:
\begin{prop}
Choosing $c^{[m]}$ to be the increasing colour sequence $1^m \cup \cdots \cup n^m \in [1,n]^{nm}$, one has the following formula:
\begin{align}
\label{factored-g}
g^{1^m \cup \cdots \cup n^m}_{\Delta}(m^n;x_1,\dots,x_{nm})
=
q^{-m^2 \binom{n}{2}}
(1-q^{-1})^{nm}
\prod_{k=0}^{n-1}
\prod_{1 \leq i<j \leq m}
(q^{-1} x_{mk+j}-x_{mk+i}).
\end{align}
\end{prop}

The initial condition \eqref{factored-g}, used in conjunction with the exchange relation \eqref{invT-g}, allows $g^{c^{[m]}}_{\Delta}$ and $g^{c^{[m+1]}}_{\Delta}$ to be efficiently implemented on a computer. We are then in a position to explicitly evaluate the transition probabilities \eqref{col-markov}; see Figure \ref{fig:n=2}.

\begin{figure}
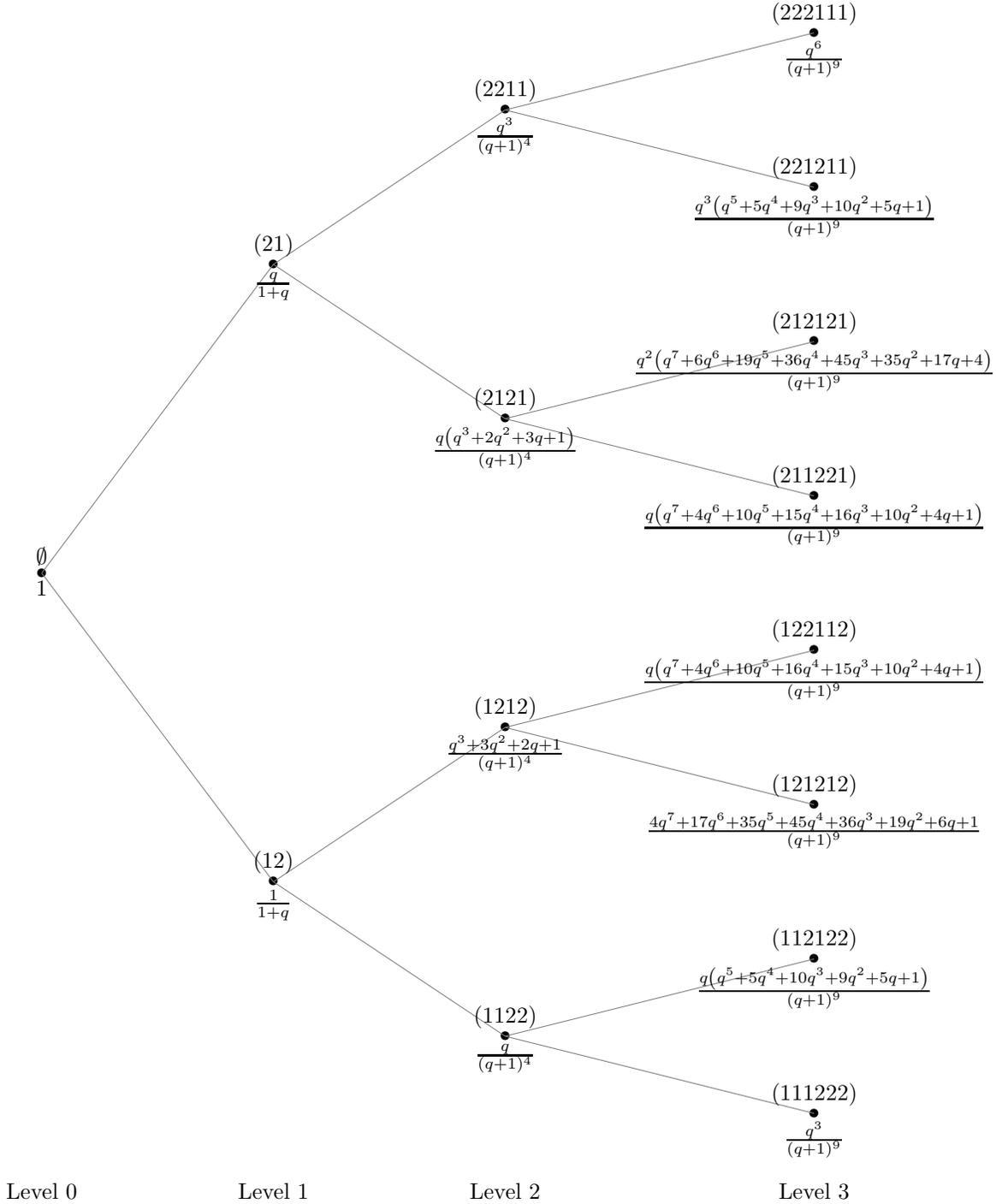

%%%nodes
\tikz{1.2}{
\foreach\y in {0}{
\node at (2,0) {$\bullet$};
}
\foreach\y in {-4,4}{
\node at (5,\y) {$\bullet$};
}
\foreach\y in {-6,-2,2,6}{
\node at (8,\y) {$\bullet$};
}
\foreach\y in {-7,-5,-3,-1,1,3,5,7}{
\node at (12,\y) {$\bullet$};
}
%%%paths
\draw[gray] (2,0) -- (5,-4);
\draw[gray] (2,0) -- (5,4);
\draw[gray] (5,-4) -- (8,-6);
\draw[gray] (5,-4) -- (8,-2);
\draw[gray] (5,4) -- (8,2);
\draw[gray] (5,4) -- (8,6);
\draw[gray] (8,-6) -- (12,-7);
\draw[gray] (8,-6) -- (12,-5);
\draw[gray] (8,-2) -- (12,-3);
\draw[gray] (8,-2) -- (12,-1);
\draw[gray] (8,2) -- (12,1);
\draw[gray] (8,2) -- (12,3);
\draw[gray] (8,6) -- (12,5);
\draw[gray] (8,6) -- (12,7);
%%%labels
\node[above] at (2,0) {$\emptyset$};
\node[above] at (5,-4) {$(12)$};
\node[above] at (5,4) {$(21)$};
\node[above] at (8,-6) {$(1122)$};
\node[above] at (8,-2) {$(1212)$};
\node[above] at (8,2) {$(2121)$};
\node[above] at (8,6) {$(2211)$};
\node[above] at (12,-7) {$(111222)$};
\node[above] at (12,-5) {$(112122)$};
\node[above] at (12,-3) {$(121212)$};
\node[above] at (12,-1) {$(122112)$};
\node[above] at (12,1) {$(211221)$};
\node[above] at (12,3) {$(212121)$};
\node[above] at (12,5) {$(221211)$};
\node[above] at (12,7) {$(222111)$};
%%%probs
\node[below] at (2,0) {$1$};
\node[below] at (5,-4) {$\frac{1}{1+q}$};
\node[below] at (5,4) {$\frac{q}{1+q}$};
\node[below] at (8,-6) {$\frac{q}{(q+1)^4}$};
\node[below] at (8,-2) {$\frac{q^3+3 q^2+2 q+1}{(q+1)^4}$};
\node[below] at (8,2) {$\frac{q \left(q^3+2 q^2+3 q+1\right)}{(q+1)^4}$};
\node[below] at (8,6) {$\frac{q^3}{(q+1)^4}$};
\node[below] at (12,-7) {$\frac{q^3}{(q+1)^9}$};
\node[below] at (12,-5) {$\frac{q \left(q^5+5 q^4+10 q^3+9 q^2+5 q+1\right)}{(q+1)^9}$};
\node[below] at (12,-3) {$\frac{4 q^7+17 q^6+35 q^5+45 q^4+36 q^3+19 q^2+6 q+1}{(q+1)^9}$};
\node[below] at (12,-1) {$\frac{q \left(q^7+4 q^6+10 q^5+16 q^4+15 q^3+10 q^2+4 q+1\right)}{(q+1)^9}$};
\node[below] at (12,1) {$\frac{q \left(q^7+4 q^6+10 q^5+15 q^4+16 q^3+10 q^2+4 q+1\right)}{(q+1)^9}$};
\node[below] at (12,3) {$\frac{q^2 \left(q^7+6 q^6+19 q^5+36 q^4+45 q^3+35 q^2+17 q+4\right)}{(q+1)^9}$};
\node[below] at (12,5) {$\frac{q^3 \left(q^5+5 q^4+9 q^3+10 q^2+5 q+1\right)}{(q+1)^9}$};
\node[below] at (12,7) {$\frac{q^6}{(q+1)^9}$};
%%%
\node at (2,-8) {Level 0};
\node at (5,-8) {Level 1};
\node at (8,-8) {Level 2};
\node at (12,-8) {Level 3};
}
\caption{The result of calculations in the case $n=2$. Each node at level $m$ corresponds to a colour sequence $c^{[m]}$; the function written below it is the probability 
$\mathbb{P}_{\rm col}(c^{[m]})$ of arriving at the colour sequence $c^{[m]}$ after $m$ applications of the Markov kernel \eqref{col-markov} to the trivial sequence $c^{[0]}=\emptyset$. Connected nodes indicate colour sequences that interlace; the resulting graph is a complete binary tree, meaning that each colour sequence $c^{[m]}$ has unique ancestry back to the root $\emptyset$. This uniqueness property ceases to hold for $n \geq 3$.}
\label{fig:n=2}
\end{figure}

\subsection{A positivity conjecture}

Studying the probabilities that appear in Figure \ref{fig:n=2}, one notices that they are always positive polynomials in $q$ over a common denominator that is easily predicted. Analysis of examples for $n \geq 3$ reveals that this structure appears to hold generally. This leads us to formulate the following positivity conjecture: 
\begin{conj}\label{conj:pos}
Fix integers $m,n \geq 1$ and a colour sequence $c^{[m]} \in [1,n]^{nm}$. Let $\mathbb{P}_{\rm col}(c^{[m]})$ denote the probability of arriving at the colour sequence $c^{[m]}$ after $m$ applications of the Markov kernel \eqref{col-markov} to the trivial sequence $c^{[0]}=\emptyset$. Then one has that
\begin{align*}
\mathbb{P}_{\rm col}\left(c^{[m]}\right)
=
\mathcal{P}\left(c^{[m]}\right)
\cdot
\left( \prod_{i=1}^{n} \frac{1-q}{1-q^i} \right)^{m^2}
\quad
\text{where}\ \ 
\mathcal{P}\left(c^{[m]}\right) \in \mathbb{N}[q].
\end{align*}
\end{conj}

\appendix

\section{Interlacing triangles and graph colourings}
\label{sec:app}

In this appendix we turn to the problem of enumerating the number of elements in the set $\mathcal{T}_N(n)$ from Definition \ref{def:interlace} (that is, computing the size of the support of $\mathbb{P}_{\rm col}$). This turns out to be a triviality for $n=1$ and $n=2$; for $n=3$ and $n=4$ we are able to conjecture a relation between the cardinality of $\mathcal{T}_N(n)$ and certain graph colourings. An elegant bijective proof of the $n=3$ conjecture was already given in \cite{GaetzGao}. The $n=4$ case remains open.

\medskip

For $n=1$, an interlacing triangular array consists of a single triangle filled with the number $1$; as there is only one such arrangement, it follows that $|\mathcal{T}_N(1)| = 1$ for all $N \geq 1$. 

\medskip

For $n=2$, we have the following elementary result:
\begin{prop}
\label{prop:n=2}
For all $N \geq 1$, $|\mathcal{T}_N(2)| = 2^N$.
\end{prop}

\begin{proof}
The only triangular arrays in rank $2$ which respect the interlacing constraint {\bf (b)} are those in which numbers remain constant along diagonals in the left triangle, and along anti-diagonals in the right; further, once we choose the numbers assigned to diagonals in the left triangle, this completely determines the right one, in view of constraint {\bf (a)} (and the fact that numbers remain constant along its anti-diagonals). Hence there are exactly $2^N$ possibilities. 
\end{proof}

For $n=3$ and $n=4$ we have no direct results around enumeration. We do, however, make two conjectures relating the cardinality of the set $\mathcal{T}_N(n)$ to colourings of certain families of graphs. 

\begin{defn}
Given a graph $G$ and an integer $m \geq 1$, an $m$-colouring of $G$ is an assignment of a label $l \in [1,m]$ to each vertex $v \in G$ such that $l \not= l'$ if $v$ and $v'$ are connected by an edge.
\end{defn}

\begin{conj}
\label{conj:triangle}
Let $G^{\triangle}_N$ denote the triangular graph
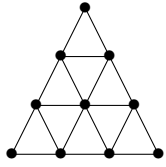
\begin{align*}
\tikz{1.3}{
\node at (0,0) {$\bullet$};
\node at (0.5,0) {$\bullet$};
\node at (1,0) {$\bullet$};
\node at (1.5,0) {$\bullet$};
\node at (0.25,0.5) {$\bullet$};
\node at (0.75,0.5) {$\bullet$};
\node at (1.25,0.5) {$\bullet$};
\node at (0.5,1) {$\bullet$};
\node at (1,1) {$\bullet$};
\node at (0.75,1.5) {$\bullet$};
%%%%%%%%%%
\draw (0,0) -- (1.5,0);
\draw (0.25,0.5) -- (1.25,0.5);
\draw (0.5,1) -- (1,1);
\draw (0.25,0.5) -- (0.5,0);
\draw (0.5,1) -- (1,0);
\draw (0.75,1.5) -- (1.5,0);
\draw (0,0) -- (0.75,1.5);
\draw (0.5,0) -- (1,1);
\draw (1,0) -- (1.25,0.5);
}
\end{align*}
where the number of vertices along one side of the triangle is equal to $N+1$. Let $\mathfrak{g}^{\triangle}_N(4)$ denote the number of $4$-colourings of $G^{\triangle}_N$.\footnote{The sequence $\mathfrak{g}^{\triangle}_N(4)$ appears as A153467 in the Online Encyclopaedia of Integer Sequences; {\tt https://oeis.org/A153467}.} We conjecture that
\begin{align*}
4\cdot|\mathcal{T}_N(3)| = \mathfrak{g}^{\triangle}_N(4), \qquad \forall\ N \geq 1.
\end{align*} 
\end{conj}

\begin{ex}[$n=3$, $N=1$]
For $n=3$ and $N=1$, triangular tuples just correspond with arrangements of $\{1,2,3\}$ along a line; hence $|\mathcal{T}_1(3)| = |\mathfrak{S}_3| = 6$. On the other hand, the possible $4$-colourings of the graph $G^{\triangle}_1 = \tikz{1.3}{\node at (0,0) {$\bullet$}; \node at (0.5,0) {$\bullet$}; \node at (0.25,0.5) {$\bullet$}; \draw (0,0) -- (0.5,0) -- (0.25,0.5) -- (0,0);}$ (in which we fix the top vertex to have label $1$, which means an undercounting by an overall factor of $4$) are
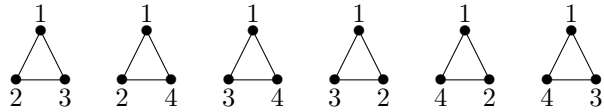
\begin{align*}
\tikz{1.3}{\node at (0,0) {$\bullet$}; \node at (0.5,0) {$\bullet$}; \node at (0.25,0.5) {$\bullet$}; \draw (0,0) -- (0.5,0) -- (0.25,0.5) -- (0,0); \node[above] at (0.25,0.5) {1}; \node[below] at (0,0) {2}; \node[below] at (0.5,0) {3};}
\quad
\tikz{1.3}{\node at (0,0) {$\bullet$}; \node at (0.5,0) {$\bullet$}; \node at (0.25,0.5) {$\bullet$}; \draw (0,0) -- (0.5,0) -- (0.25,0.5) -- (0,0); \node[above] at (0.25,0.5) {1}; \node[below] at (0,0) {2}; \node[below] at (0.5,0) {4};}
\quad
\tikz{1.3}{\node at (0,0) {$\bullet$}; \node at (0.5,0) {$\bullet$}; \node at (0.25,0.5) {$\bullet$}; \draw (0,0) -- (0.5,0) -- (0.25,0.5) -- (0,0); \node[above] at (0.25,0.5) {1}; \node[below] at (0,0) {3}; \node[below] at (0.5,0) {4};}
\quad
\tikz{1.3}{\node at (0,0) {$\bullet$}; \node at (0.5,0) {$\bullet$}; \node at (0.25,0.5) {$\bullet$}; \draw (0,0) -- (0.5,0) -- (0.25,0.5) -- (0,0); \node[above] at (0.25,0.5) {1}; \node[below] at (0,0) {3}; \node[below] at (0.5,0) {2};}
\quad
\tikz{1.3}{\node at (0,0) {$\bullet$}; \node at (0.5,0) {$\bullet$}; \node at (0.25,0.5) {$\bullet$}; \draw (0,0) -- (0.5,0) -- (0.25,0.5) -- (0,0); \node[above] at (0.25,0.5) {1}; \node[below] at (0,0) {4}; \node[below] at (0.5,0) {2};}
\quad
\tikz{1.3}{\node at (0,0) {$\bullet$}; \node at (0.5,0) {$\bullet$}; \node at (0.25,0.5) {$\bullet$}; \draw (0,0) -- (0.5,0) -- (0.25,0.5) -- (0,0); \node[above] at (0.25,0.5) {1}; \node[below] at (0,0) {4}; \node[below] at (0.5,0) {3};}
\end{align*}
and indeed $4\cdot|\mathcal{T}_1(3)| = \mathfrak{g}^{\triangle}_1(4)$. 
\end{ex}

\begin{conj}
\label{conj:a5}
Let $G^{\varhexstar}_N$ denote the graph
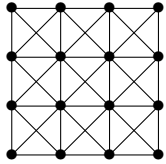
\begin{align*}
\tikz{1.3}{
\node at (0,0) {$\bullet$};
\node at (0.5,0) {$\bullet$};
\node at (1,0) {$\bullet$};
\node at (1.5,0) {$\bullet$};
%%%
\node at (0,0.5) {$\bullet$};
\node at (0.5,0.5) {$\bullet$};
\node at (1,0.5) {$\bullet$};
\node at (1.5,0.5) {$\bullet$};
%%%
\node at (0,1) {$\bullet$};
\node at (0.5,1) {$\bullet$};
\node at (1,1) {$\bullet$};
\node at (1.5,1) {$\bullet$};
%%%
\node at (0,1.5) {$\bullet$};
\node at (0.5,1.5) {$\bullet$};
\node at (1,1.5) {$\bullet$};
\node at (1.5,1.5) {$\bullet$};
%%%%%%%
\draw (0.5,0) -- (0,0.5);
\draw (1,0) -- (0,1);
\draw (1.5,0) -- (0,1.5);
\draw (1.5,0.5) -- (0.5,1.5);
\draw (1.5,1) -- (1,1.5);
\draw (1,0) -- (1.5,0.5);
\draw (0.5,0) -- (1.5,1);
\draw (0,0) -- (1.5,1.5);
\draw (0,0.5) -- (1,1.5);
\draw (0,1) -- (0.5,1.5);
\draw (0,0) -- (0,1.5);
\draw (0.5,0) -- (0.5,1.5);
\draw (1,0) -- (1,1.5);
\draw (1.5,0) -- (1.5,1.5);
\draw (0,0) -- (1.5,0);
\draw (0,0.5) -- (1.5,0.5);
\draw (0,1) -- (1.5,1);
\draw (0,1.5) -- (1.5,1.5);
}
\end{align*}
where two vertices share an edge if they are connected via a king move on the chessboard (that is, they a connected via a unit horizontal, vertical, or diagonal step), and the number of vertices along one side of the square is equal to $N+1$. Let $\mathfrak{g}^{\varhexstar}_N(5)$ denote the number of $5$-colourings of $G^{\varhexstar}_N$.\footnote{The sequence $\mathfrak{g}^{\varhexstar}_N(5)$ appears as A068294 in the Online Encyclopaedia of Integer Sequences; {\tt https://oeis.org/A068294}.} We conjecture that
\begin{align*}
5 \cdot |\mathcal{T}_N(4)| = \mathfrak{g}^{\varhexstar}_N(5),
\qquad
\forall\ N \geq 1.
\end{align*}
\end{conj}

%bib
\bibliographystyle{alpha}
\bibliography{references}

\end{document}